\documentclass[]{interact}

\usepackage[numbers,sort&compress]{natbib}
\bibpunct[, ]{[}{]}{,}{n}{,}{,}
\renewcommand\bibfont{\fontsize{10}{12}\selectfont}

\usepackage[english]{babel}
\usepackage{hyperref}       
\usepackage{url}            
\usepackage{booktabs}       
\usepackage{amsfonts}       
\usepackage{tablefootnote}
\usepackage{verbatim}
\usepackage{nicefrac}       
\usepackage{microtype}      
\usepackage{lipsum}
\usepackage{mathtools}
\usepackage{threeparttable}

\usepackage{amsmath,amssymb,amsthm}
\usepackage{algorithm,algorithmic}
\usepackage{pifont}
\usepackage{cases}
\usepackage{subcaption,graphicx}
\usepackage{stackengine}    
\usepackage{wrapfig}
\usepackage{enumitem}

\theoremstyle{plain}
\newtheorem{theorem}{Theorem}[section]
\newtheorem{lemma}[theorem]{Lemma}
\newtheorem{corollary}[theorem]{Corollary}

\theoremstyle{definition}
\newtheorem{definition}[theorem]{Definition}

\newtheorem{assumption}[theorem]{Assumption}

\theoremstyle{remark}
\newtheorem{remark}{Remark}

\newcommand{\eps}{\varepsilon}

\newcommand{\ol}{\overline}
\newcommand{\one}{\mathbf{1}}

\newcommand{\circledOne}{\text{\ding{172}}}
\newcommand{\circledTwo}{\text{\ding{173}}}
\newcommand{\circledThree}{\text{\ding{174}}}
\newcommand{\circledFour}{\text{\ding{175}}}
\newcommand{\circledFive}{\text{\ding{176}}}

\renewcommand{\le}{\leqslant}

\renewcommand{\hat}{\widehat}

\newcommand{\numberthis}{\addtocounter{equation}{1}\tag{\theequation}}

\DeclareMathOperator*{\argmin}{arg\,min}
\DeclareMathOperator*{\argmax}{arg\,max}
\DeclareMathOperator*{\Argmin}{Arg\,min}
\DeclareMathOperator*{\Argmax}{Arg\,max}
\DeclareMathOperator{\spn}{span}
\DeclareMathOperator{\kernel}{Ker}

\newcommand{\R}{\mathbb{R}}
\newcommand{\Z}{\mathbb{Z}}
\newcommand{\V}{\mathbb{V}}

\newcommand{\mA}{{\bf A}}

\newcommand{\mW}{{\bf W}}

\newcommand{\sA}{{\mathcal{A}}}

\newcommand{\sP}{{\mathcal{P}}}
\newcommand{\sQ}{{\mathcal{Q}}}

\newcommand{\sS}{{\mathcal{S}}}
\newcommand{\sT}{{\mathcal{T}}}
\newcommand{\sU}{{\mathcal{U}}}
\newcommand{\sV}{{\mathcal{V}}}

\newcommand{\sX}{{\mathcal{X}}}
\newcommand{\sY}{{\mathcal{Y}}}
\newcommand{\sZ}{{\mathcal{Z}}}

\newcommand{\bp}{{\bf p}}
\newcommand{\bq}{{\bf q}}

\newcommand{\bs}{{\bf s}}
\newcommand{\bt}{{\bf t}}
\newcommand{\bu}{{\bf u}}
\newcommand{\bv}{{\bf v}}

\newcommand{\bx}{{\bf x}}
\newcommand{\by}{{\bf y}}
\newcommand{\bz}{{\bf z}}

\newcommand{\mWx}{\mW_{\bx}}
\newcommand{\mWy}{\mW_{\by}}

\newcommand{\ds}{\displaystyle}
\newcommand{\norm}[1]{\left\| #1 \right\|}

\newcommand{\angles}[1]{\left\langle #1 \right\rangle}
\newcommand{\cbraces}[1]{\left( #1 \right)}
\newcommand{\sbraces}[1]{\left[ #1 \right]}
\newcommand{\braces}[1]{\left\{ #1 \right\}}

\usepackage[colorinlistoftodos,bordercolor=orange,backgroundcolor=orange!20,linecolor=orange,textsize=scriptsize]{todonotes}

\def \W{\mathcal W}

\def\X{\mathcal X}

\def\Pp{\mathcal P}

\def\R{\mathbb R}

\def\U{\mathcal{U}}
\def\V{\mathcal{V}}
\def\Z{\mathcal{Z}}

\def\e{\varepsilon}
\def\la{\langle}
\def\ra{\rangle}

\def\x{\mathbf {x}}
\def\z{\mathbf {z}}

\def\W{\mathbf {W}}
\def\u{\mathbf {u}}
\def\b{\mathbf {b}}
\def\v{\mathbf {v}}
\def\A{\boldsymbol{A}}

\def\one{{\mathbf 1}}

\def\p{\mathbf{p}}
\def\q{\mathbf{q}}

\def\lm{\lambda}

\def\<#1,#2>{\langle #1,#2\rangle}

\newcommand{\add}[1]{{\color{black}#1}}
\newcommand{\revoms}[1]{{\color{black}#1}}

\begin{document}


\title{Decentralized Saddle Point Problems via Non-Euclidean Mirror Prox
\footnote{The work of A. Rogozin was supported by a grant for research centers in the field of artificial intelligence, provided by the Analytical Center for the Government of the Russian Federation in accordance with the subsidy agreement [agreement identifier 000000D730321P5Q0002] and the agreement with the Moscow Institute of Physics and Technology dated November 1, 2021 [grant number 70-2021-00138].}
}

\author{
    \name{
        Alexander Rogozin\textsuperscript{a,b}\thanks{CONTACT Alexander Rogozin. Email: aleksandr.rogozin@phystech.edu},
        Aleksandr Beznosikov\textsuperscript{a,b},
        Darina Dvinskikh\textsuperscript{b},
        Dmitry Kovalev\textsuperscript{c},
        Pavel Dvurechensky\textsuperscript{e},
        Alexander Gasnikov\textsuperscript{a,b,d}
    }
    \affil{
        \textsuperscript{a}Moscow Institute of Physics and Technology, Moscow, Russia;\\
        \textsuperscript{b}National Research University Higher School of Economics, Moscow, Russia;\\
        \textsuperscript{c}King Abdullah University of Science and Technology, Thuwal, Saudi Arabia;\\
        \textsuperscript{d}Institute for Information Transmission Problems of RAS, Moscow, Russia;\\
        \textsuperscript{e}Weierstrass Institute for Applied Analysis and Statistics, Berlin, Germany;
    }
}

\maketitle

\begin{abstract}
We consider smooth convex-concave saddle point problems in the decentralized distributed setting, where a finite-sum objective is distributed among the nodes of a computational network. At each node, the local objective depends on the groups of local and global variables. For such problems, we propose a decentralized distributed algorithm with $O(\varepsilon^{-1})$ communication and oracle calls complexities to achieve accuracy $\varepsilon$ in terms of the duality gap and in terms of consensus between nodes. Further, we prove lower bounds for the communication and oracle calls complexities and show that our algorithm matches these bounds, i.e., it is optimal. In contrast to existing decentralized algorithms, our algorithm admits non-euclidean proximal setup, including, e.g., entropic. We illustrate the work of the proposed algorithm on the prominent problem of computing Wasserstein barycenters (WB), where a non-euclidean proximal setup arises naturally in a bilinear saddle point reformulation of the WB problem.
\end{abstract}


\section{Introduction}

In the last few years, we observe an increased interest in the research on algorithms for saddle-point problems (SPP), in particular, motivated by modern applications in GANs training \cite{gidel2018variational}, 
\add{reinforcement learning \cite{Omidshafiei2017:rl, wai2018multi, jin2020efficiently}, optimal transport \cite{jambulapati2019direct} and distributed control \cite{NECOARA2011756}. In addition to the above modern applications, and besides their classical examples in  economics, equilibrium theory, game theory \cite{facchinei2007finite}, saddle-point problems remain popular in supervised learning (with non-separable loss \citep{Thorsten}; with non-separable regularizer~\citep{bach2011optimization}),  unsupervised learning (discriminative clustering \citep{NIPS2004_64036755};  matrix factorization \citep{bach2008convex}), image denoising \citep{esser2010general,chambolle2011first}, robust optimization \citep{BenTal2009:book}, optimization with separable or semi-definite constraints \cite{mateos2015distributed}, and non-smooth optimization via smooth reformulations \citep{nesterov2005smooth,nemirovski2004prox}.}
Decentralized algorithms for SPPs are also an active area of research \cite{liu2019decentralized,mukherjee2020decentralized}.

For convex optimization, the theory of decentralized first-order methods is currently well-developed:
the lower bounds on the number of communication rounds and oracle calls are well-known,
and algorithms  converging according to these lower bounds are developed, see \cite{nedic2009distributed,jakovetic2014fast,arjevani2015communication,lan2017communication,scaman2017optimal,sca2018,uribe2020dual,kovalev2020optimal,gorbunov2020recent,dvinskikh2021decentralized,song2021optimal}.
For convex-concave decentralized SPPs a complete theory exists only in the strongly convex-concave case for unconstrained problems in the euclidean proximal setup \cite{kovalev2022optimal}. Moreover, the lower bounds from \cite{beznosikov2020local,beznosikov2021distributed,kovalev2022optimal} were obtained only for the case when strong convexity and strong concavity parameters are the same. 



For non-strongly convex-concave case, distributed SPP with local and global variables were studied in \cite{mateos2015distributed}, where the authors proposed a subgradient-based algorithm for \add{non-smooth problems} with $O(1/\sqrt N)$ convergence guarantee ($N$ is the number of communication rounds). 
\add{Paper \cite{zhang2021primal} introduced an Extra-gradient algorithm for distributed multi-block SPP with affine constraints. Their method covers the Euclidean case and the algorithm has $O(1/N)$ convergence rate.}
\add{Our paper proposes an algorithm based on adding Lagrangian multipliers to consensus constraints, which is analogical to \cite{zhang2021primal}, but our method works} in a general proximal smooth setup \add{and} achieves $O(1/N)$ convergence rate. \add{Moreover, it has} an enhanced dependence on the condition number of the network.

An optimal method for convex-concave SPPs
is Mirror-prox algorithm \cite{nemirovski2004prox,ouyang2019lower}
with  general proximal setup. 
The Mirror-prox algorithm can be performed in a decentralized manner, however, it is not known whether its \add{optimality is preserved}. 
In this paper, we prove that  Mirror-prox remains optimal  even in a decentralized case  w.r.t. the dependence on the desired accuracy $\varepsilon$ and condition number $\chi$ of communication network if we split communication and oracle complexities by Chebyshev acceleration trick (see, e.g. \cite{li2021accelerated}).


Finally,  we show how the proposed method can be applied  to prominent  problem of computing  Wasserstein barycenters to tackle the problem of instability of regularization-based approaches under a small value of regularizing parameter. The idea is based on the saddle point reformulation of the Wasserstein barycenter problem (see \cite{dvinskikh2020improved}).  Wasserstein barycenters, which define the mean of objects that can be modeled as probability measures on a metric space (images, texts, videos), are used in many fields including Bayesian computations \cite{srivastava2015wasp}, texture mixing \cite{rabin2011wasserstein}, clustering  ($k$-means for probability measures) \cite{del2019robust}, shape interpolation and color transferring \cite{Solomon2015},  statistical estimation of template models \cite{boissard2015distribution} and neuroimaging \cite{gramfort2015fast}.

In the numerical experiments conducted on different network architectures, we demonstrate a better approximation of the true barycenter by the proposed distributed Mirror-prox algorithm  in comparison with \add{regularization} based approaches. 
      
\noindent\textbf{Contribution.} Our contribution can be summarized as follows.
\begin{itemize}[leftmargin=*]
    \item We provide a decentralized Mirror-prox based algorithm for  convex-concave SPPs and prove its optimality in the euclidean proximal setup;
    \item We provide the lower bounds on the number of communication steps and oracle calls per node for convex-concave SPPs 
    in the euclidean proximal setup. 
\end{itemize}

\noindent\textbf{Paper organization.} This paper is organized as follows. Section \ref{sec:problem_statement} presents a  saddle point problem of interest along with its decentralized reformulation. In Section \ref{sec:distributed_mirror_prox}, we provide the main algorithm of the paper to solve such kind of problems. In Section \ref{sec:lower_bounds}, we present the lower complexity bounds for saddle point problems without individual variables. Finally in Section \ref{sec:wb}, we show how the proposed algorithm can be applied to the  problem computing Wasserstein barycenters .

\noindent\textbf{Notation.}
 For a prox-function $d(x)$, we define the corresponding  Bregman divergence: $B(x, y) = d(x) -d(y) - \la  \nabla d(y), x - y \ra.$ 
 For some norm $\|\cdot\|$, we define its dual norm $\|\cdot\|_*$ in a usual way: $\|s\|_* = \max_{x\in \X} \{ \la x,s \ra : \|x\|\leq 1 \}.$
For two vectors $x,y$ of the same size, denotations $x/y$ and $x \odot y$  stand for the element-wise product and  element-wise division respectively. When functions, such as $log$ or $exp$, are used on vectors, they are always applied element-wise.
We use bold symbol for  column vector $\x = [x_1^\top,\cdots,x_m^\top]^\top \in \mathbb{R}^{mn}$, where $x_1,...,x_m\in \R^n$.
Then we refer to the $i$-th component of  vector $\x$ as $ x_i\in \R^n$ and to the $j$-th component of  vector $x_i$ as $[x_i]_j$.  We denote by 
$\boldsymbol 1$ the vector of ones.
For matrices $A$ and $B$ we denote their Kronecker product as $A\otimes B$.

\section{Problem statement}\label{sec:problem_statement}

In this paper, we consider a saddle point problem of the sum-type with global and local groups of variables 
\begin{align}\label{eq:problem_initial}
    \min_{\substack{x\in\bar\sX\\\bp\in\sP}} \max_{\substack{y\in\bar\sY\\\bq\in\sQ}}~ {f(x, \bp, y, \bq) \triangleq \frac{1}{m}}\sum_{i=1}^m f_i(x, p_i, y, q_i),
\end{align}
where $\bp \triangleq (p_1^\top \in \sP_1 \ldots p_m^\top\in \sP_m)^\top \in \sP \triangleq \sP_1\times\ldots\times\sP_{m}$ and $\bq\triangleq (q_1^\top \in \sQ_1 \ldots q_m^\top \in \sQ_m)^\top \in \sQ \triangleq \sQ_1\times\ldots\times\sQ_{m}$. Variables $x, p_i, y, q_i$ have dimensions $d_x, d_p, d_y, d_q$ respectively for all $i=1,\dots,m$. For problem \eqref{eq:problem_initial}, we make the following assumption.

\begin{assumption}\label{assum:convex_compact}
\item
    \begin{enumerate}[leftmargin=*]
    \setlength\itemsep{-0.3em}
    	\item Sets $\sP_i$, $\sQ_i$ ($i = 1, \ldots, m$) and $ \bar\sX$, $\bar\sY$ are \add{convex and compact}.
    	\item Function $f_i(\cdot, \cdot, y, q_i)$ is convex on $\bar\sX\times\sP_i$ for every fixed $y\in\bar\sY,~ q_i\in\sQ_i$ for all $i=1,\dots,m$.
    	\item Function $f_i(x, p_i, \cdot, \cdot)$ is concave on $\bar\sY\times\sQ_i$ for every fixed $x\in\bar\sX,~p_i\in\sP_i$ for all $i=1,\dots,m$.
    \end{enumerate}
\end{assumption}

\subsection{Decentralized setup and communication matrix}
We suppose that there is a communication network of agents (machines/computing nodes) which we represent by  a connected undirected graph $\mathcal{G} = (V, E)$.  Each agent $i$ privately 
holds  its function $f_i$   ($i=1,\dots,m$). Every pair of agents $(i,j)$ can communicate iff $(i,j)\in E$. We can represent these communication constraints imposed
by the network through a particular matrix ${ W}$ satisfying the following assumption
\begin{assumption}\label{assum:mixing_matrix}
\item
\begin{enumerate}[leftmargin=*]
\setlength\itemsep{-0.3em}
    \item ${ W}$ is symmetric positive semi-definite matrix
	\item (Network compatibility) For all $i,j = 1,\dots,m$ the entry of $ W$: $[{ W}]_{ij} = 0$ if $(i, j)\notin E$ and $i\ne j$.
	\item (Kernel property) For any $v = [v_1,\ldots,v_m]^\top\in\R^m$, ${ W} v = 0$ if and only if $v_1 = \ldots = v_m$, i.e.  $\kernel { W} = \spn\braces{\one}$.
\end{enumerate}    
\end{assumption}
An example of matrix satisfying this assumption is the graph Laplacian ${ W}{\in \mathbb{R}^{m\times m}}$:
{\begin{align}\label{eq:matrx_def}
[W]_{ij} \triangleq \begin{cases}
-1,  & \text{if } (i,j) \in E,\\
\text{deg}(i), &\text{if } i= j, \\
0,  & \text{otherwise,}
\end{cases}
\end{align}}
where $\text{deg}(i)$ is the degree of the node $i$, i.e., the number of neighbors of the node.
The main characteristic of the network is the graph condition number \[\chi \triangleq \chi( W) \triangleq {\lambda_{\max}( W)}/{\lambda_{\min}^+( W)},\] where $\lambda_{\max}( W)$ and $\lambda_{\min}^+( W)$ denote the maximal and the minimal positive eigenvalues of $W$ respectively.

\subsection{{Distributed problem reformulation}}
Let matrices $W_x\in\R^{m\times m}$ and $W_y\in\R^{m\times m}$ be associated with variables $x$ and $y$, respectively. Introduce communication matrices $\mWx = W_x\otimes I_{d_x},~ \mWy = W_y\otimes I_{d_y}$. We can equivalently rewrite the problem \eqref{eq:problem_initial} in a distributed setup  replacing the constraint $x_1=\dots=x_m$ by $\mWx\bx = 0$ and $y_1=\dots=y_m$ by $\mWy\by = 0$, respectively. Let us introduce $F(\bx, \bp, \by, \bq) \triangleq \sum_{i=1}^m f_i(x_i, p_i, y_i, q_i)$ and $\sX = \bar\sX\times\ldots\times\bar\sX$ and $\sY = \bar\sY\times\ldots\times\bar\sY$.
\begin{assumption}\label{assum:bounded_gradients}
    There exist positive scalars $M_x, M_y$ such that {for all $i = 1,\ldots,m$ and} for any $x_i\in\bar\sX$, $y_i\in\bar\sY$, $p_i\in\sP_i$, $q_i\in\sQ_i$ it holds $\norm{\nabla_x f_i(x_i, p_i, y_i, q_i)}_2\leq M_x$, $\norm{\nabla_y f_i(x_i, p_i, y_i, q_i)}_2\leq M_y$.
\end{assumption}

Our method requires rewriting problem \eqref{eq:problem_initial} by adding Lagrangian multipliers $\bs, \bz$. Moreover, we bound the norms of the dual variables at the solution.


\begin{theorem}\label{th:constrained_saddle}
    Introduce
    \begin{align*}
        \bar R_\sZ^2 = \frac{\bar\alpha {m} M_x^2}{( \lambda_{\min}^+(\mWx))^2},~ \bar R_\sS^2 = \frac{\bar\beta {m} M_y^2}{( \lambda_{\min}^+(\mWy))^2},
    \end{align*}
    where $\bar\alpha, \bar\beta\in (1, +\infty) \cup \braces{+\infty}$.
    Problem \eqref{eq:problem_initial} is equivalent to
    \begin{align}\label{eq:problem_reform_constrained}
            \min_{\substack{\bp\in\sP, \bx\in\sX \\ \norm{\bs}_2\le \bar R_\sS}} \max_{\substack{\by\in\sY, \bq\in\sQ \\ \norm{\bz}\le \bar R_\sZ}} \sbraces{F(\bx, \bp, \by, \bq) +  \angles{\bz, \mWx\bx} +  \angles{\bs, \mWy\by}}.
    \end{align}
\end{theorem}
Bounding the norms of dual variables is a known result in minimization \cite{lan2020communication,gorbunov2019optimal}. We prove Theorem \ref{th:constrained_saddle} \revoms{by considering minimization over $(\bx, \bp, \bs)$ and maximization over $(\by, \bq, \bz)$ separately. We decompose a saddle-point problem into two optimization problems and apply to each of them the following result.}


\begin{lemma}\label{lemma:constrained_dual}
	Let $\Theta\subseteq\R^d$ be a closed convex set and $h(\theta): \Theta\to\R$ be a convex differentiable function. Consider a problem with affine constraints
	\begin{align}\label{eq:lemma_affine_constraints}
		\min_{\theta\in\Theta} h(\theta) \quad\text{s.t. } \mA\theta = b 
	\end{align}
	
	\begin{enumerate}[leftmargin=*]
	    \item Dual problem has a solution $\nu^*$ such that $\norm{\nu^*}_2^2\leq \frac{\norm{\nabla h(\theta^*)}_2^2}{(\sigma_{\min}^+(\mA))^2} =: R_\nu^2$, where $\sigma_{\min}^+(\mA)$ denotes the minimal non-zero singular value of $\mA$ {(initially presented in \cite{lan2020communication})}.
	    \item Let $R > R_{{\nu}}$ and consider a constrained dual problem
    	\begin{align}
    		\max_{\norm{\nu}_2\leq R} &\min_{\theta\in\Theta} \sbraces{h(\theta) + \angles{\nu, \mA\theta - b}} = \max_{\norm{\nu}_2\leq R} \varphi(\nu) \label{eq:lemma_constrained_saddle}
    	\end{align}
    	If $(\tilde\theta, \tilde\nu)$ is a saddle point of \eqref{eq:lemma_constrained_saddle}, then $\tilde\theta$ is a solution of \eqref{eq:lemma_affine_constraints}.
	\end{enumerate}
\end{lemma}

\begin{proof}
	For part 1, see \cite{lan2020communication}.
	
	\noindent Introduce function
	\begin{align*}
		\psi(\theta) &= \max_{\norm{\nu}_2\leq R} \sbraces{h(\theta) + \angles{\nu, \mA\theta - b}}.
	\end{align*}
	Since $\revoms{\Theta}$ is a compact set, it holds 
	\begin{align*}
		\max_{\norm{\nu}_2\leq R} \varphi(\nu) = \min_{\theta\in\Theta} \psi(\theta)
	\end{align*}
	by Sion-Kakutani theorem (see i.e. Theorem D.4.2 in \cite{ben2011lectures}). Moreover, since 
	\begin{align*}
		\nu^*\in\Argmax_{\nu\in\R^m}\varphi(\nu) \text{ and } \nu^*\in B_R(0),
	\end{align*}
	we have 
	\begin{align*}
		\nu^*\in\Argmax_{\norm{\nu}_2\leq R}\varphi(\nu) \text{ and } \max_{\nu\in\R^m} \varphi(\nu) = \max_{\norm{\nu}_2\leq R} \varphi(\nu).
	\end{align*}
	Also note that $\ds h(\revoms{\tilde\theta}) = \min_{\theta\in\Theta} \psi(\theta)$ by Theorem D.4.1 in \cite{ben2011lectures}. Combining the three facts
	\begin{align*}
		\revoms{\tilde\theta}\in\Argmin_{\theta\in\Theta} \psi(\theta),~ \nu^*\in\Argmax_{\norm{\nu}_2\leq R} \varphi(\nu),~ \min_{\theta\in\Theta} \psi(\theta) = \max_{\norm{\nu}_2\leq R} \varphi(\nu)
	\end{align*}
	we obtain that $(\revoms{\tilde\theta}, \nu^*)$ is a saddle point of \eqref{eq:lemma_constrained_saddle} by Theorem D.4.1 of \cite{ben2011lectures}. \revoms{Note that $(\tilde\theta, \tilde\nu)$ is a saddle point of~\eqref{eq:lemma_constrained_saddle}, as well, but we do not need $\tilde\nu$ in the analysis}. Therefore for $\nu\in B_R(0)$ we have
	\begin{align*}
		h(\revoms{\tilde\theta}) + \angles{\nu^*, \mA\revoms{\tilde\theta} - b} &\geq h(\revoms{\tilde\theta}) + \angles{\revoms{\nu^*}, \mA\revoms{\tilde\theta} - b} \\
		\angles{\nu^* - \nu, \mA\revoms{\tilde\theta} - b} &\geq 0.
	\end{align*}
	Taking into account that $\norm{\nu^*}_2\leq R_\nu < R$, we imply $\nu^*\in\text{int}B_R(0)$ and therefore $\mA\revoms{\tilde\theta} - b  = 0$. \revoms{Therefore, we have} 
	\begin{align*}
		h(\hat\theta) &\leq h(\theta)~ \forall \theta\in\Theta: \mA \revoms{\theta} = b
	\end{align*}
	which concludes the proof.
\end{proof}

\begin{proof}[Proof of Theorem~\ref{th:constrained_saddle}]
We are free to use Sion--Kakutani theorem since the sets $\sX, \sP, \sY, \sQ$ are compact.
\begin{align}\label{eq:saddle_reformulation_1}
    \min_{\substack{\mWx\bx = 0 \\ \bp\in\sP, \bx\in\sX}} \max_{\substack{\mWy\by = 0 \\ \by\in\sY, \bq\in\sQ}} F(\bx, \bp, \by, \bq)
    &= \max_{\substack{\mWy\by = 0 \\ \by\in\sY, \bq\in\sQ}} \min_{\substack{\mWx\bx = 0 \\ \bp\in\sP, \bx\in\sX}} F(\bx, \bp, \by, \bq) \notag \\
	&= \max_{\substack{\mWy\by = 0 \\ \by\in\sY, \bq\in\sQ}} \min_{\bp\in\sP} \sbraces{\min_{\substack{\mWx\bx = 0 \\ \bx\in\sX}} F(\bx, \bp, \by, \bq)}
\end{align}
For any fixed pair $(\bx, \by, \bq)\in \sX\times\sZ\times\sQ$ function $F(\bx, \bp, \by, \bq)$ is convex in $\bp$. Consider problem
\begin{align}\label{eq:min_on_p}
	\min_{\bx\in\sX} F(\bx, \bp, \by, \bq) \text{ s.t. } \mWx\bx = 0
\end{align}
and denote 
\begin{align*}
	\bx^*(\bp, \by, \bq) = \argmin_{\substack{\mWx\bx = 0 \\ \bx\in\sX}}\sbraces{F(\bx, \bp, \by, \bq)}.
\end{align*}
By Lemma \ref{lemma:constrained_dual}, part 1, there exists a solution of dual to \eqref{eq:min_on_p} with norm bounded by
\begin{align}\label{eq:Rdef}
	\frac{\norm{\nabla_\bx F(\bx^*(\bp, \by, \bq), \bp, \by, \bq)}_2^2}{(\lambda_{\min}^+(\mWx))^2} \leq \bar R_\sZ.
\end{align}
By Lemma \ref{lemma:constrained_dual} (part 2), problem \eqref{eq:min_on_p} is equivalent to
\begin{align}\label{eq:min_on_p_saddle}
	\max_{\norm{\bz}_2\leq \bar R_\sZ} \min_{\bx\in\sX} \sbraces{F(\bx, \bp, \by, \bq) + \angles{\bz, \mWx\bx}}
\end{align}
in the sense that for any saddle point $(\hat\bx(\bp, \by, \bq), \hat\bz(\bx, \by, \bq))$ of \eqref{eq:min_on_p_saddle} we have that $\hat\bx(\bp, \by, \bq)$ is a solution of \eqref{eq:min_on_p}.

Returning to \eqref{eq:saddle_reformulation_1}, we get
\begin{align*}
	\max_{\substack{\mWy\by = 0 \\ \by\in\sY, \bq\in\sQ}} &\min_{\bp\in\sP} \sbraces{\min_{\substack{\mWx\bx = 0 \\ \bx\in\sX}} F(\bx, \bp, \by, \bq)} \\
	&= \max_{\substack{\mWy\by = 0 \\ \by\in\sY, \bq\in\sQ}} \min_{\bp\in\sP} \sbraces{\max_{\norm{\bz}_2\leq \bar R_\sZ} \min_{\bx\in\sX} \cbraces{F(\bx, \bp, \by, \bq) + \angles{\bz, \mWx\bx}}} \\
	& = \min_{\bp\in\sP, \bx\in\sX} \max_{\substack{\by\in\sY, \bq\in\sQ \\ \mWy\by = 0, \norm{\bz}_2\leq \bar R_\sZ}} \sbraces{F(\bx, \bp, \by, \bq) + \angles{\bz, \mWx\bx}} \\
	& = \min_{\bp\in\sP, \bx\in\sX} \max_{\substack{\bq\in\sQ \\ \norm{\bz}_2\leq \bar R_\sZ}} \sbraces{\max_{\substack{\by\in\sY \\ \mWy\by = 0}} \sbraces{F(\bx, \bp, \by, \bq) + \angles{\bz, \mWx\bx}}}. \numberthis\label{eq:saddle_reformulation_2}
\end{align*}
Now we introduce Lagrange multipliers for the constraints $\mWy\by = 0$, as well. Analogously to the case with $\mWx\bx = 0$ constraints, consider a problem
\begin{align}\label{eq:max_on_r}
    \max_{\by\in\sY} \sbraces{F(\bx, \bp, \by, \bq) + \angles{\bz, \mWx\bx}}~ \text{s.t.}~ \mWy\by = 0
\end{align}
and introduce its solution 
\begin{align*}
    \by^*(\bx, \bp, \bq, \bz) = \argmax_{\substack{\mWy\by = 0 \\ \by\in\sY}} \sbraces{F(\bx, \bp, \by, \bq) + \angles{\bz, \mWx\bx}}.
\end{align*}
Analogously, the dual solution norm to problem \eqref{eq:max_on_r} can be bounded as
\begin{align}
\label{eq:R_u_def}
    \frac{\norm{\nabla_\by F(\bx, \bp, \by^*(\bx, \bp, \bq, \bz), \bq)}_2^2}{(\lambda_{\min}^+ (\mWy))^2} \leq \bar R_\sS^2,~ \beta \in (1, +\infty].
\end{align}
And we get a saddle-point reformulation of \eqref{eq:max_on_r}:
\begin{align*}
    \min_{\norm{\bs}_2\leq \bar R_\sS} \max_{\bq\in\sQ} \sbraces{F(\bx, \bp, \by, \bq) + \angles{\bz, \mWx\bx} + \angles{\bs, \mWy\by}}.
\end{align*}
Substituting this reformulation into \eqref{eq:saddle_reformulation_2}, we obtain
\begin{align}
    &\min_{\bp\in\sP, \bx\in\sX} \max_{\substack{\bq\in\sQ \\ \norm{\bz}_2\leq \bar R_\sZ}} \sbraces{\max_{\substack{\by\in\sY \\ \mWy\by = 0}} \sbraces{F(\bx, \bp, \by, \bq) + \angles{\bz, \mWx\bx}}} \notag \\
    &\quad = \min_{\bp\in\sP, \bx\in\sX} \max_{\substack{\bq\in\sQ \\ \norm{\bz}_2\leq \bar R_\sZ}} \sbraces{\min_{\norm{\bs}_2\leq \bar R_\sS} \max_{\by\in\sY} \sbraces{F(\bx, \bp, \by, \bq) + \angles{\bz, \mWx\bx} + \angles{\bs, \mWy\by}}} \notag  \\
    &\quad = \min_{\bp\in\sP, \bx\in\sX} \max_{\substack{\by\in\sY, \bq\in\sQ \\ \norm{\bz}_2\leq \bar R_\sZ}} \min_{\norm{\bs}_2\leq \bar R_\sS} \sbraces{F(\bx, \bp, \by, \bq) + \angles{\bz, \mWx\bx} + \angles{\bs, \mWy\by}} \notag \\
    &\quad = \max_{\substack{\by\in\sY, \bq\in\sQ \\ \norm{\bz}_2\leq \bar R_\sZ}} \min_{\substack{\bp\in\sP, \bx\in\sX \\ \norm{\bs}_2\leq \bar R_\sS}} \sbraces{F(\bx, \bp, \by, \bq) + \angles{\bz, \mWx\bx} + \angles{\bs, \mWy\by}}, \label{eq:saddle_reformulation_3}
\end{align}
which is an equivalent reformulation of \eqref{eq:problem_initial}. Note that cases $\bar R_\sS = +\infty, \bar R_\sZ = +\infty$ are also supported in this proof. A $\min$-$\max$ reformulation is obtained analogously by adding Lagrange multipliers in different order.
\end{proof}

From Theorem~\ref{th:constrained_saddle} we immediately obtain the following corollary by setting $\bar\alpha = \bar\beta = +\infty$.
\begin{corollary}\label{cor:saddle_point_reformulation}
    Problem \eqref{eq:problem_initial} can be equivalently rewritten as follows
    \begin{align}\label{eq:problem_reform}
    	\min_{\substack{\bx\in\sX, ~\bp\in\sP \\ \bs\in\R^{m d_y}}} ~\max_{\substack{\by\in\sY, ~\bq\in\sQ \\ \bz\in\R^{m d_x}}} ~\sbraces{F(\bx, \bp, \by, \bq) + \angles{\bz, \mWx\bx} + \angles{\bs, \mWy\by}},
    \end{align}
    in the sense that for any saddle point $(\bx^*, \bp^*, \by^*, \bq^*, \bs^*, \bz^*)$ of \eqref{eq:problem_reform} we have $x_1^* = \ldots = x_m^* = x^*$, $y_1^* = \ldots = y_m^* = y^*$, and the point $(x^*, \bp^*, y^*, \bq^*)$ is a saddle point of \eqref{eq:problem_initial}.
\end{corollary}

\section{Distributed algorithm for saddle-point problems}\label{sec:distributed_mirror_prox}

Now we provide a decentralized distributed algorithm to solve saddle-point problem \eqref{eq:problem_initial}. The pseudo-code is listed in Algorithm \ref{alg:mirror_prox_distr} and the main result on convergence rate is given in Theorem \ref{Th:distr_MP}.
The main idea is to use reformulation \eqref{eq:problem_reform} and apply mirror prox algorithm \cite{nemirovski2004prox} for its solution. This requires careful analysis in two aspects. First, the Lagrange multipliers $\bz,\bs$ are not constrained, while the convergence rate result for the classical Mirror-Prox algorithm \cite{nemirovski2004prox} is proved for problems on compact sets. Second, we need to show that the updates can be organized via only local communications between the nodes in the network. 

\subsection{Algorithm}
For each variable $t_i\in\{x_i, p_i, y_i, q_i, s_i, z_i\}$ and corresponding set $\sT_i\in\braces{\bar\sX, \sP_i, \bar\sY, \sQ_i, \R^{d_y}, \R^{d_x}}$, we assume there is a norm $\norm{t_i}_{\sT_i}$. Our analysis \revoms{supports} arbitrary norms for $x_i, p_i, y_i, q_i$, but we use the Euclidean norm for dual variables $s_i, z_i$, i.e. $\norm{s_i}_{\sS_i} = \norm{s_i}_2,~ \norm{z_i}_{\sZ_i} = \norm{z_i}_2$. With each norm $\norm{t_i}_{\sT_i}$ we associate a 1-strongly convex w.r.t. this norm prox-function $d_{\sT_i}(t_i)$ and the corresponding Bregman divergence $B_{\sT_i}(t_i,\breve{t}_i) = d_{\sT_i}(t_i) - d_{\sT_i}(\breve t_i) - \angles{\nabla d_{\sT_i}(\breve t_i), t_i - \breve t_i}$. For dual variables we have prox-functions  $d_{\sS_i}(u_i)=\frac{1}{2}\norm{u_i}_{2}^2$, $d_{\sZ_i}(z_i)=\frac{1}{2}\norm{z_i}_2^2$ and Bregman divergences $B_{\sS_i}(s_i,\breve{s}_i)=\frac{1}{2}\norm{s_i-\breve{s}_i}_{2}^2$, $B_{\sZ_i}(z_i,\breve{z}_i)=\frac{1}{2}\norm{z_i-\breve{z}_i}_{2}^2$.

Having introduced $B_{\sT_i}(t_i, \breve t_i)$, we define the following \textit{Mirror step}:
\begin{equation}
    \label{eq:mirror}
    {\rm Mirr}(g_i;t_i;\mathcal{T}_i)=\argmin_{t\in \mathcal{T}_i}  \sbraces{\angles{g_i, t} + B_{\sT_i}(t,t_i)},
\end{equation}
where $g_i$ is an element of the corresponding dual space which defines the step direction.

Our decentralized algorithm for saddle-point problem \eqref{eq:problem_reform} is listed as Algorithm \ref{alg:mirror_prox_distr}. In each iteration each agent $i$ makes two Mirror step updates in each of its six local variables. Besides using respective gradient of the local objective $f_i$ to define the step direction, some updates include aggregating the variables of other agents. Consider for example the update $x_i^{k+\frac{1}{2}}$, in which agent $i$ needs to calculate the sum $\boldsymbol{\Sigma}_{j}[\mWx]_{ij}z_j^k$. Due to the Assumption \ref{assum:mixing_matrix}, i.e. network compatibility of matrix $W$, this update requires aggregating local only from neighbouring agents. Thus, each iteration of the algorithm is performed in a distributed manner.

In Algorithm \ref{alg:mirror_prox_distr}, we denote $\nabla_t f_i^\ell = \nabla_t f_i(x_i^\ell, p_i^\ell, y_i^\ell, q_i^\ell)$ for $t\in\braces{x, p, y, q}$ and $\ell\in\braces{k, k + \frac{1}{2}}$ 

In the next two subsections, we introduce two necessary components which allow us to prove the convergence rate theorem for our algorithm. These are smoothness assumptions and localization of the solution to the saddle-point problem \eqref{eq:problem_reform}. 

\setlength{\textfloatsep}{3pt}
\begin{algorithm}[ht!]
\caption{Decentralized Mirror-Prox}
{
\begin{algorithmic}[1]\label{alg:mirror_prox_distr}
    \REQUIRE{\add{Initial guess $(\bx^0, \bp^0, \by^0, \bq^0, \z^0, \bs^0)$ such that $\ds\bt = [t_1^\top, \ldots, t_m^\top]^\top$, $\ds t_i = \argmin_{t\in\sT_i} d_{\sT_i}(t)$ for each $t_i\in\braces{x_i, p_i, y_i, q_i, z_i, s_i}$ and corresponding set $\sT_i\in\braces{\bar\sX, \sP_i, \bar\sY, \sQ_i, \R^{d_y}, \R^{d_x}}$, $i = 1, \ldots, m$.}}
    \FOR{$k = 0, 1, \ldots, N - 1$}
        \STATE{Compute $\tilde\bx^k = \mWx\bx^k;~ \tilde\bs^k = \mWy\bs^k;~ \tilde\by^k = \mWy\by^k;~ \tilde\bz^k = \mWx\bz^k$.}
        \STATE{Each node computes
            \begin{align*}
                &x_i^{k+\frac{1}{2}} = {\rm Mirr}\left(
                \alpha(\nabla_{x} f_i^k
                + \tilde z_i^k); x_i^k; \bar\sX\right),~~~~
                s_i^{k+\frac{1}{2}} = {\rm Mirr} \cbraces{\alpha\tilde y_i^k; s_i^k; \R^{d_y}}, \\
                &y_i^{k+\frac{1}{2}} = {\rm Mirr} \cbraces{-\alpha(\nabla_r f_i^k + \tilde s_i^k); y_i^k; \bar\sY},~~
                z_i^{k+\frac{1}{2}} = {\rm Mirr} \cbraces{-\alpha\tilde x_i^k; z_i^k; \R^{d_x}}.
            \end{align*}
            }
        \STATE{Compute $\tilde\bx^{k+\frac{1}{2}} = \mWx\bx^{k+\frac{1}{2}};~ \tilde\bs^{k+\frac{1}{2}} = \mWy\bs^{k+\frac{1}{2}};~ \tilde\by^{k+\frac{1}{2}} = \mWy\by^{k+\frac{1}{2}};~ \tilde\bz^{k+\frac{1}{2}} = \mWx\bz^{k+\frac{1}{2}}$.}
        \STATE{Each node computes
            \begin{align*}
                &x_i^{k+1} = {\rm Mirr}\left(
                \alpha(\nabla_{p} f_i^{k+\frac{1}{2}}
                + \tilde z_i^{k+\frac{1}{2}}); x_i^k; \bar\sX\right),~~~~
                s_i^{k+1} = {\rm Mirr} \cbraces{\alpha \tilde y_i^{k+\frac{1}{2}}; s_i^k; \R^{d_y}}, \\
                &y_i^{k+1} = {\rm Mirr} \cbraces{-\alpha(\nabla_r f_i^{k+\frac{1}{2}} + \tilde s_i^{k+\frac{1}{2}}); y_i^k; \bar\sY},~~
                z_i^{k+1} = {\rm Mirr} \cbraces{-\alpha\tilde x_i^{k+\frac{1}{2}}; z_i^k; \R^{d_x}}.
            \end{align*}
            }
    \ENDFOR
    \ENSURE{For $\bt\in\braces{\bx, \bp, \by, \bq, \bs, \bz}$ compute $\ds\hat \bt^{N} = \frac{1}{N}\sum_{k=0}^{N-1} \bt^{k+\frac{1}{2}}$.}
\end{algorithmic}
}
\end{algorithm}

\subsection{Localizing the solution and smoothness assumptions}


As it was noted above, the standard analysis of Mirror-Prox requires the feasible sets to be compact. Although we run Mirror-Prox algorithm on problem~\eqref{eq:problem_reform} with unconstrained variables $\bs$ and $\bz$, we still can bound these variables according to Theorem~\ref{th:constrained_saddle}.


\begin{lemma}\label{lemma:constrained_lagrange_multipliers}
    Let Assumption \ref{assum:bounded_gradients} be satisfied and  
    $R_\sZ^2 = {2m} M_x^2(\lambda_{\min}^+(\mWx))^{-2}$, $R_\sS^2 = {2m} M_y^2(\lambda_{\min}^+(\mWy))^{-2}$,
    where $\lambda_{\min}^+(\cdot)$ is the minimal non-zero eigenvalue of matrix. Then there exists a saddle point $(\bx^*, \bp^*, \by^*, \bq^*, \bs^*, \bz^*)$ of problem \eqref{eq:problem_reform} such that $\norm{\bs^*}_2\leq R_\sS$, $\norm{\bz^*}_2\leq R_\sZ$.
\end{lemma}
\begin{proof}
Putting $\bar\alpha = \bar\beta = 2$ in Theorem~\ref{th:constrained_saddle}, we immediately obtain the proof.
\end{proof}

Next, we introduce the second important component of the convergence rate analysis, namely the smoothness assumption on the objective $F$. 
To set the stage we first introduce a general definition of Lipschitz-smooth function of two variables.
Having defined norms and Bregman divergences, let $t_i\in\braces{x_i,p_i,s_i,y_i,q_i,z_i}$ and correspondingly $\sT_i\in\braces{\bar\sX,\sP_i,\R^{d_y},\bar\sY,\sQ_i,\R^{d_x}}$. Introduce $\bt = (t_1^\top, \ldots, t_m^\top)^\top$ and $\sT = \sT_1\times\ldots\times\sT_m$ we define $\norm{\bt}_{\sT}^2 = \sum_{i=1}^m \norm{t_i}_{\sT_i}^2$. 
\begin{definition}
Consider norms $\norm{\cdot}_\xi,~ \norm{\cdot}_\eta$ and their dual norms $\norm{\cdot}_{\xi,*},~ \norm{\cdot}_{\eta,*}$. A differentiable function $G(\xi, \eta)$ is called $(L_{\xi\xi},L_{\xi\eta},L_{\eta\xi},L_{\eta\eta})$-smooth w.r.t. norms $\norm{\cdot}_\xi,~ \norm{\cdot}_\eta$ if
\begin{align*}
    &\norm{\nabla_\xi G(\xi, \eta) - \nabla_\xi G(\xi', \eta)}_{\xi,*} \leq L_{\xi\xi}\norm{\xi - \xi'}_{\xi}, \\
    &\norm{\nabla_\xi G(\xi, \eta) - \nabla_\xi G(\xi, \eta')}_{\xi,*} \leq L_{\xi\eta} \norm{\eta - \eta'}_{\eta}, \\
    &\norm{\nabla_\eta G(\xi, \eta) - \nabla_\eta G(\xi', \eta)}_{\eta,*} \leq L_{\eta\xi} \norm{\xi - \xi'}_{\xi}, \\
    &\norm{\nabla_\eta G(\xi, \eta) - \nabla_\eta G(\xi, \eta')}_{\eta,*} \leq L_{\eta\eta} \norm{\eta - \eta'}_{\eta}.
\end{align*}
\end{definition}

In order to formulate the smoothness assumptions on function $F$, we group the minimization variables $(\bx,\bp)$ and maximization variables $(\by,\bq)$.
\begin{assumption}\label{assum:F_Lip_smooth}
	The function $F$ is $(L_{(\bx,\bp)(\bx,\bp)},L_{(\bx,\bp)(\by,\bq)},L_{(\by,\bq)(\bx,\bp)},L_{(\by,\bq)(\by,\bq)})$-smooth w.r.t. norms $\norm{\cdot}_{(\sX,\sP)},~ \norm{\cdot}_{(\sY,\sQ)}$ defined as
    \begin{align*}
    	&\norm{\cbraces{\bx,\bp}}_{(\sX,\sP)}^2 = \norm{\bx}_{\sX}^2 + \norm{\bp}_{\sP}^2, \qquad \norm{\cbraces{\by,\bq}}_{(\sY,\sQ)}^2 = \norm{\by}_{\sY}^2 + \norm{\bq}_{\sQ}^2. \\
    \end{align*}
\end{assumption}

\subsection{Main result}

To present a convergence bound for Algorithm \ref{alg:mirror_prox_distr}, we  introduce additional notation. For each $\sT_i\in\braces{\bar\sX, \sP_i, \bar\sY, \sQ_i, \R^{d_y}, \R^{d_x}}$, we introduce 
\begin{equation}\label{eq:R_def_for_prox}
    \ds R_{\sT_i}^2 = \max_{t, t'\in\sT_i} d_{\sT_i}(t) - d_{\sT_i}(t').
\end{equation}


For simplicity of further derivations, we aggregate all the variables in two blocks: minimization variables and maximization variables. Introducing $\bu=(\bx^\top,\bp^\top,\bs^\top)^\top$ and $\bv=(\by^\top,\bq^\top,\bz^\top)^\top$  makes it possible  to rewrite the problem \eqref{eq:problem_reform} in the following simplified form.

\begin{equation}\label{eq:saddle_point}
    \min_{\bu \in A_\sU} \max_{\bv \in A_\sV} \ 
    S(\bu,\bv) \triangleq F(\bx, \bp, \by, \bq) + \angles{\bs, \mWy\by} + \angles{\bz, \mWx\bx},
\end{equation}
where $\sA_\sU = \sX\times\sP\times\R^{nd_q}$ and $\sA_\sV = \sY\times\sQ\times \R^{nd_p}$. The sets $A_{\sU}$ and $A_{\sV}$ are closed and convex but  unbounded. We also define $\sU = \sX\times\sP\times\{\bs:~ \norm{\bs}_2\leq R_\sS\}$ and  $\sV = \sY\times\sQ\times\{\bz:~ \norm{\bz}_2\leq R_\sZ\}$.  Convexity properties of $F$ imply that $S(\cdot, \bv)$ is convex for any fixed $\bv$ and $S(\bu, \cdot)$ is concave for any fixed $\bu$.

Introduce norms 
\begin{align}\label{eq:normsUV}
    \norm{\bu}_{\sU}^2 &= \frac{\norm{\bx}_{\sX}^2}{R_\sX^2} + \frac{\norm{\bp}_{\sP}^2}{R_\sP^2} + \frac{\norm{\bs}_{2}^2}{R_\sS^2}, \qquad
    \norm{\bv}_{\sV}^2 = \frac{\norm{\by}_{\sY}^2}{R_\sY^2} + \frac{\norm{\bq}_{\sQ}^2}{R_\sQ^2} + \frac{\norm{\bz}_{2}^2}{R_\sZ^2}.
\end{align}
\revoms{Also introduce prox-functions $d_\sU(\bu),~ d_\sV(\bv)$ and Bregman divergences
\begin{align*}
	B_\sU(\bu, \breve\bu) &= \frac{B_\sX(\bx, \breve\bx)}{R_\sX^2} + \frac{B_\sP(\bp, \breve\bp)}{R_\sP^2} + \frac{B_\sS(\bs, \breve\bs)}{R_\sS^2}, \\
	B_\sV(\bv, \breve\bv) &= \frac{B_\sY(\by, \breve\by)}{R_\sY^2} + \frac{B_\sQ(\bq, \breve\bq)}{R_\sQ^2} + \frac{B_\sZ(\bz, \breve\bz)}{R_\sZ^2}.
\end{align*}
}

\begin{lemma}\label{lemma:smothness_constants}
    Function $S(\bu,\bv)$ defined in \eqref{eq:saddle_point} is $(L_{\bu\bu}, L_{\bu\bv}, L_{\bv\bu}, L_{\bv\bv})$-smooth, where
    \begin{align*}
        L_{\bu\bu} &= L_{(\bx,\bp)(\bx,\bp)} (R_\sX^2 + R_\sP^2), \; L_{\bv\bv} = L_{(\by,\bq)(\by,\bq)} (R_\sY^2 + R_\sQ^2), \\
        L_{\bu\bv} &= \sqrt{2}\cbraces{L_{(\bx,\bp)(\by,\bq)} \sqrt{(R_\sX^2 + R_\sP^2)(R_\sY^2 + R_\sQ^2)} + R_\sX R_\sZ\norm{\mWx}_{2\to(\bx,*)} + R_\sY R_\sS\norm{\mWy}_{\by\to2}}, \\
        L_{\bv\bu} &= \sqrt{2}\cbraces{L_{(\by,\bq)(\bx,\bp)} \sqrt{(R_\sX^2 + R_\sP^2)(R_\sY^2 + R_\sQ^2)} + R_\sX R_\sZ\norm{\mWx}_{\bx\to2} + R_\sY R_\sS\norm{\mWy}_{2\to(\by,*)}}.
    \end{align*}
\end{lemma}
The proof is mostly technical and therefore omitted to the Appendix \ref{app:mp_smoothness_constants}.

Let us now discuss the second main aspect of the analysis, i.e. unboundedness of the feasible sets $A_\sU$, $A_\sV$, which does not allow to directly apply the standard analysis of \cite{nemirovski2004prox}. Due to the special structure of the problem \eqref{eq:problem_reform} and the Assumption \ref{assum:bounded_gradients}, we have at the disposal bounds $R_\sZ, R_\sS$ defined in Lemma \ref{lemma:constrained_lagrange_multipliers} such that the optimal values $\bz^*$ and $\bs^*$ satisfy $\norm{\bz^*}_2\leq R_\sZ$, $\norm{\bs^*}_2\leq R_\sS$. We have that the saddle-point $(\bu^*,\bv^*)$ belongs to  $\sU\times \sV$. Let us define 
\begin{align}\label{eq:def_L_zeta}
L_\zeta=2\max\{L_{\bu\bu},L_{\bu\bv},L_{\bv\bu},L_{\bv\bv}\}.
\end{align}

Finally, applying the analysis of Mirror-Prox, we obtain the main result.
\begin{theorem}
\label{Th:distr_MP}
Let Assumptions \ref{assum:mixing_matrix}, \ref{assum:bounded_gradients} and \ref{assum:F_Lip_smooth} hold. \add{Let Algorithm \ref{alg:mirror_prox_distr} be run for $N$ iterations with stepsize $\alpha=1/L_\zeta$, where $L_\zeta$ is defined in \eqref{eq:def_L_zeta}. Introduce $\ol{\hat\bt}^N = \frac{1}{m}\one\one^\top \hat\bt^N$ for $\bt\in\braces{\bx, \by, \bp, \bq}$.}
We have
\begin{align*}
    &\add{\max_{\substack{\by\in\sY, \bq\in\sQ \\ \mWy\by=0}} F\cbraces{\ol{\hat\bx}^N, \hat\bp^N, \by, \bq} - \min_{\substack{\bx\in\sX, \bp\in\sP \\ \mWx\bx = 0}} F(\bx, \bp, \ol{\hat\by}^N, \hat\bq^N)\leq \frac{(4 + 17\sqrt{2})L_\zeta}{N}}
\end{align*}
and
\begin{align*}
    \norm{\mWx\hat\bx^N}_2\leq \frac{17L_\zeta}{NR_\sZ}, \qquad \norm{\mWy\hat\by^N}_2\leq \frac{17L_\zeta}{NR_\sS}.
\end{align*}
\end{theorem}
\revoms{
Constant $L_\zeta$ is dependent on problem parameters as defined in~\eqref{eq:def_L_zeta}. Let us provide an explicit dependence for completeness.
\begin{corollary}
	Under conditions of Theorem~\ref{Th:distr_MP} we have
	\begin{align*}
		&\max_{\substack{\by\in\sY, \bq\in\sQ \\ \mWy\by=0}} F\cbraces{\ol{\hat\bx}^N, \hat\bp^N, \by, \bq} - \min_{\substack{\bx\in\sX, \bp\in\sP \\ \mWx\bx = 0}} F(\bx, \bp, \ol{\hat\by}^N, \hat\bq^N) \\
		&\quad\leq\frac{(72+ 8\sqrt{2})m}{N} \Bigg(L_{(\bx,\bp)(\bx,\bp)}(R_\sX^2 + R_\sP^2) + L_{(\by,\bq)(\by,\bq)}(R_\sY^2 + R_\sP^2) \\
		&\qquad+ (L_{(\bx,\bp)(\by,\bq)} + L_{(\by,\bq)(\bx,\bp)})\sqrt{(R_\sX^2 + R_\sP^2)(R_\sY^2 + R_\sQ^2)} \\
		&\qquad+ \frac{M_xR_\sX(\norm{\mWx}_{2\to(\bx,*)} + \norm{\mWx}_{\bx\to 2})}{\lambda_{\min}^+(\mWx)} + \frac{M_yR_\sY(\norm{\mWy}_{2\to(\by,*)} + \norm{\mWy}_{\by\to 2})}{\lambda_{\min}^+(\mWy)} \Bigg).
	\end{align*}
\end{corollary}
}
To prove Theorem \ref{Th:distr_MP} we first show that the iterates of Algorithm \ref{alg:mirror_prox_distr} naturally correspond to the iterates of a general Mirror-Prox algorithm applied to problem \eqref{eq:problem_reform}. Then we extend the standard analysis of the general Mirror-Prox algorithm to account for unbounded feasible sets. 

The feasible set $\sA_\sU \times \sA_\sV$ is separable w.r.t. each local feasible set $\bar\sX \times \sP_i \times  \R^{d_q} \times \bar\sY \times \bar\sQ_i \times  \R^{d_p}$ for an agent $i$. Thus, if we define a variable $\zeta = (\bu^\top , \bv^\top)^\top \in \sA := \sA_\sU \times \sA_\sV$, and the operator $g(\zeta)$
\begin{align*}
    g(\zeta) = \begin{bmatrix*}[l]
        ~~\nabla_\bu S(\bu,\bv) \\
        ~~ -\nabla_\bv S(\bu,\bv) 
    \end{bmatrix*}=
    \begin{bmatrix*}[l]
        ~~\nabla_\bx F(\bx, \bp, \by, \bq) + \mWx\bz \\
        ~~\nabla_\bp F(\bx, \bp, \by, \bq) \\
        ~~\mWy\by \\
        -\nabla_\by F(\bx, \bp, \by, \bq) - \mWy\bs \\
        -\nabla_\bq F(\bx, \bp, \by, \bq) \\
        -\mWx\bx
    \end{bmatrix*},
\end{align*}
then the updates of Algorithm \ref{alg:mirror_prox_distr} are equivalent to the updates of a general Mirror-Prox algorithm listed as Algorithm \ref{alg:mirror_prox_general}.

\begin{algorithm}[H]
\caption{Mirror-Prox}
\label{alg:mirror_prox_general}
\begin{algorithmic}[1]
    \REQUIRE{Starting point $\ds\add{\zeta^0 = \argmin_{\zeta\in Q} d_\zeta(\zeta)}$, stepsize $\alpha > 0$.}
    \FOR{$k = 0, 1, \ldots$}
        \STATE 
        \begin{align}
        \label{eq:gen_MP_steps}
            \zeta^{k+\frac{1}{2}} &= \argmin_{\zeta \in Q} \alpha\angles{g(\zeta^k), \zeta} + B(\zeta, \zeta^k) \\
            \zeta^{k+1} &= \argmin_{\zeta \in Q} \alpha\angles{g(\zeta^{k+\frac{1}{2}}), \zeta} + B(\zeta, \zeta^k)
        \end{align}
    \ENDFOR
    \ENSURE       $\hat{\zeta}^k = \frac{1}{k}\sum_{\ell=0}^{k-1}\zeta^{\ell+\frac{1}{2}}$.
\end{algorithmic}
\end{algorithm}

Next, we analyze the general Mirror-Prox Algorithm \ref{alg:mirror_prox_general}.
Let us define the norm $\norm{\zeta}_{\zeta}^2 = \norm{\bu}_{\sU}^2 + \norm{\bv}_{\sV}^2$, the corresponding prox-function $d_{\zeta}(\zeta)$ and Bregman divergence $B_{\zeta}(\zeta,\breve \zeta)=B_{\bu}(\bu,\breve \bu)+B_{\bv}(\bv,\breve \bv)$. Note that the corresponding dual norm is $\norm{\zeta}_{\zeta,*}^2=\norm{\bu}_{\sU,*}^2+\norm{\bv}_{\sV,*}^2$. Under these definitions, following the standard analysis in \cite{bubeck2014theory}, we obtain that the operator $g(\zeta)$ is $L_{\zeta}$-Lipschitz-continuous with respect to the norm $\norm{\zeta}_{\zeta}$ with
\begin{equation}
\label{eq:L_zeta}
  L_{\zeta}=2\max\{L_{\bu\bu},L_{\bv\bv},L_{\bu\bv},L_{\bv\bu}\}.
\end{equation}

Next, we analyze the iterations of Algorithm \ref{alg:mirror_prox_general} under the assumption that the operator $g(\zeta)$ is $L_{\zeta}$-Lipschitz-continuous.
Let us fix some iteration $k\geq 0$. 
By the optimality conditions in \eqref{eq:gen_MP_steps}, we have, for any $\zeta\in Q$,
\begin{align}
 & \alpha \langle g(\zeta^k) + \nabla d_{\zeta}(\zeta^{k+\frac{1}{2}}) -  \nabla d_{\zeta}(\zeta^k), \zeta- \zeta^{k+\frac{1}{2}} \rangle \geq 0, \label{eq:Th:UMPGenRatePr1} \\
 &\alpha \langle g(\zeta^{k+\frac{1}{2}}) + \nabla d_{\zeta}(\zeta^{k+1}) - \nabla d_{\zeta}(\zeta^k), \zeta - \zeta^{k+1} \rangle \geq 0. \label{eq:Th:UMPGenRatePr2}
\end{align}
Whence, for all $\zeta\in Q$,
\begin{align*}
\langle g(\zeta^{k+\frac{1}{2}})&, \zeta^{k+\frac{1}{2}} - \zeta \rangle  =  \langle g(\zeta^{k+\frac{1}{2}}) , \zeta^{k+1} - \zeta \rangle + \langle g(\zeta^{k+\frac{1}{2}}), \zeta^{k+\frac{1}{2}} - \zeta^{k+1}\rangle \\
& \hspace{-1em} \stackrel{\eqref{eq:Th:UMPGenRatePr2} }{\leq}   \frac{1}{\alpha} \langle \nabla d_{\zeta}(\zeta^k) -  \nabla d_{\zeta}(\zeta^{k+1}), \zeta^{k+1} - \zeta \rangle + \langle g(\zeta^{k+\frac{1}{2}}), \zeta^{k+\frac{1}{2}} - \zeta^{k+1}\rangle \\
& \hspace{-1em} =  \frac{1}{\alpha} B_{\zeta}(\zeta,\zeta^{k})-\frac{1}{\alpha} B_{\zeta}(\zeta,\zeta^{k+1}) - \frac{1}{\alpha} B_{\zeta}(\zeta^{k+1},\zeta^{k}) + \langle g(\zeta^{k+\frac{1}{2}}), \zeta^{k+\frac{1}{2}} - \zeta^{k+1}\rangle, 
\end{align*}
where the last equality uses the definition of Bregman divergence $B_{\zeta}(\zeta,\breve{\zeta}) = d_{\zeta}(\zeta)-(d_{\zeta}(\breve{\zeta})+\langle \nabla d_{\zeta}(\breve{\zeta}), \zeta - \breve{\zeta}\rangle)$. 
Further, for all $\zeta\in Q$,
\begin{align*}
&\langle g(\zeta^{k+\frac{1}{2}}), \zeta^{k+\frac{1}{2}} - \zeta^{k+1}\rangle - \frac{1}{\alpha} B_{\zeta}(\zeta^{k+1},\zeta^{k}) \\
&\quad=  \langle g(\zeta^{k+\frac{1}{2}}) - g(\zeta^{k}) , \zeta^{k+\frac{1}{2}} - \zeta^{k+1}\rangle - \frac{1}{\alpha} B_{\zeta}(\zeta^{k+1},\zeta^{k}) + \langle g(\zeta^{k}) , \zeta^{k+\frac{1}{2}} - \zeta^{k+1}\rangle \\
& 
\stackrel{\eqref{eq:Th:UMPGenRatePr1} \; \text{with} \; \zeta=\zeta^{k+1}}{\leq}
\langle g(\zeta^{k+\frac{1}{2}}) - g(\zeta^{k}) , \zeta^{k+\frac{1}{2}} - \zeta^{k+1}\rangle \\
&\qquad\qquad+ \frac{1}{\alpha} \langle \nabla d_{\zeta}(\zeta^k) -  \nabla d_{\zeta}(\zeta^{k+\frac{1}{2}}), \zeta^{k+\frac{1}{2}} - \zeta^{k+1} \rangle - \frac{1}{\alpha} B_{\zeta}(\zeta^{k+1},\zeta^{k})  \\
&\quad =  \langle g(\zeta^{k+\frac{1}{2}}) - g(\zeta^{k}) , \zeta^{k+\frac{1}{2}} - \zeta^{k+1}\rangle  - \frac{1}{\alpha}B(\zeta^{k+\frac{1}{2}},\zeta^{k}) - \frac{1}{\alpha} 
B(\zeta^{k+1},\zeta^{k+\frac{1}{2}})\\
&\quad\stackrel{\circledOne}{\leq}
\|g(\zeta^{k+\frac{1}{2}}) - g(\zeta^{k})\|_{\zeta,*} \|\zeta^{k+\frac{1}{2}} - \zeta^{k+1}\|_{\zeta} - \frac{1}{2\alpha}\left( \|\zeta^{k+\frac{1}{2}} -\zeta^{k}\|_{\zeta}^2 + \|\zeta^{k+\frac{1}{2}} - \zeta^{k+1}\|_{\zeta}^2  \right) \\
&\quad \stackrel{\circledTwo}{\leq}
L_{\zeta}\|\zeta^{k+\frac{1}{2}} - \zeta^{k})\|_{\zeta} \|\zeta^{k+\frac{1}{2}} - \zeta^{k+1}\|_{\zeta} - \frac{L_{\zeta}}{2}\left( \|\zeta^{k+\frac{1}{2}} -\zeta^{k}\|_{\zeta}^2 + \|\zeta^{k+\frac{1}{2}} - \zeta^{k+1}\|_{\zeta}^2  \right) \leq 0,
\end{align*}
\revoms{where in $\circledOne$ we used that $B_{\zeta}(\zeta,\breve{\zeta}) \geq \frac{1}{2}\revoms{\|\zeta-\breve{\zeta}\|_{\zeta}^2}$ and in $\circledTwo$ we substituted $\alpha = 1/L_\zeta$ and used that $g(\zeta)$ is $L_{\zeta}$-Lipschitz-continuous}.

Combining the above two inequalities and the choice $\alpha=1/L_{\zeta}$, we obtain, for all $\zeta\in Q$ and $i \geq 0$,
$$
    \langle g(\zeta^{i+\frac{1}{2}}), \zeta^{i+\frac{1}{2}} - \zeta \rangle \leq  L_{\zeta} B_{\zeta}(\zeta,\zeta^{i})-L_{\zeta} B_{\zeta}(\zeta,\zeta^{i+1}).
$$
Summing up these inequalities for $i$ from 0 to $k-1$, we have:
$$
    \sum_{i=0}^{k-1} \langle g(\zeta^{i+\frac{1}{2}}), \zeta^{i+\frac{1}{2}} - \zeta \rangle \leq  L_{\zeta} (B_{\zeta}(\zeta,\zeta^{0})- B_{\zeta}(\zeta,\zeta^{k})) \leq L_{\zeta} B_{\zeta}(\zeta,\zeta^{0}).
$$

Now we use the connection between $S(\bu,\bv)$ and the operator $g(\zeta)$.
By convexity of $S(\bu,\bv)$ in $\bu$ and concavity of $S(\bu,\bv)$ w.r.t $\bv$, we have, for all $\bu \in \sA_\sU$,
\begin{align*}
\frac{1}{k} \sum_{i=0}^{k-1} \left\langle \nabla_{\bu}
S(\bu^{i},\bv^{i}),\bu^{i}-\bu \right \rangle & \geq  \frac{1}{k} \sum_{i=0}^{k-1} (S(\bu^{i},\bv^{i})-S(\bu,\bv^{i})) \geq  \frac{1}{k} \sum_{i=0}^{k-1}  S(\bu^{i},\bv^{i})-S(\bu,\widehat{\bv}^k).
  \end{align*}
In the same way, we obtain, for all $\bv \in \sA_\sV$,
$$
    \frac{1}{k} \sum_{i=0}^{k-1} \left\langle -\nabla_{\bv} S(\bu^{i},\bv^{i}),\bv^{i}-\bv\right \rangle \geq -\frac{1}{k} \sum_{i=0}^{k-1} S(\bu^{i},\bv^{i})+S(\widehat{\bu}^k,\bv).
$$
Summing these inequalities, by the definition of $g(\zeta)$ we obtain that, for all $\bu\in \sA_\sU$, $\bv\in \sA_\sV$,
$$
    S(\widehat{\bu}^k,\bv) - S(\bu,\widehat{\bv}^k) \leq  \frac{1}{k} \sum_{i=0}^{k-1}\langle g(\zeta^{i+\frac{1}{2}}), \zeta^{i+\frac{1}{2}} - \bu \rangle \leq \frac{L_{\zeta}}{k}  B_{\zeta}(\zeta,\zeta^{0}).
$$

It remains to deduce the target accuracy in the value of function $F$ and to the solution accuracy of problem \eqref{eq:problem_reform}. Let $\hat\bu^N = ((\hat\bx^N)^\top, (\hat\bp^N)^\top, (\hat\bs^N)^\top)^\top$, $\hat\bv^N = ((\hat\by^N)^\top, (\hat\bq^N)^\top, (\hat\bz^N)^\top)^\top$. For all $\bx\in\sX,~ \bp\in\sP,~ \by\in\sY,~ \bq\in\sQ,~ \bs\in B_{R_\sS}(0),~ \bz\in B_{R_\sZ}(0)$, we have
\begin{align*}
    &F(\hat\bx^N, \hat\bp^N, \by, \bq) + \angles{\bz, \mWx\hat\bx^N} + \angles{\hat\bs^N, \mWy\by} - F(\bx, \bp, \hat\by^N, \hat\bq^N) - \angles{\hat\bz^N, \mWx\bx} - \angles{\bs, \mWy\hat\by^N} \\
    &\qquad\leq \frac{L_\zeta B_\zeta(\zeta, \zeta^0)}{N}.
\end{align*}

To get an upper bound on consensus residual, we substitute $\bx = \hat\bx^N,~ \by = \hat\by^N,~ \bp = \hat\bp^N,~ \bq = \hat\bq^N,~ \bz = \hat\bz^N + \frac{\mWx\hat\bx^N R_\sZ}{\norm{\mWx\hat\bx^N}_2},~ \bs = 0$ and get
\begin{align*}
    R_\sZ&\angles{\frac{\mWx\hat\bx^N}{\norm{\mWx\hat\bx^N}}, \mWx\hat\bx^N} \\
    &\leq \frac{L_\zeta}{N} \cbraces{\frac{B_\bx(\bx^N, \bx^0)}{R_\sX^2} + \frac{B_\by(\by^N, \by^0)}{R_\sY^2} + \frac{B_\bp(\bp^N, \bp^0)}{R_\sP^2} + \frac{B_\bq(\bq^N, \bq^0)}{R_\sQ^2}} \\
    &\quad+ \frac{L_\zeta}{2NR_\sZ^2}\norm{\hat\bz^N + \frac{\mWx\bx^N R_\sZ}{\norm{\mWx\hat\bx^N}}}_2^2 \\
    &\leq \frac{4L_\zeta}{N} + \frac{L_\zeta}{N}\cbraces{\frac{\norm{\hat\bz^N}_2^2}{R_\sZ^2} + 1}
    \overset{\circledOne}{\leq} \frac{17L_\zeta}{N}.
\end{align*}
Here in $\circledOne$ we bounded $\norm{\hat\bz^N}_2$ the following way:
\revoms{
\begin{align*}
    &\norm{\hat\bz^N}_2^2 \leq 2R_\sZ^2 B_\zeta(\hat\zeta^N, \zeta^0) \leq 2R_\sZ^2 B_\zeta(\zeta^*, \zeta^0) \\
    &\quad\leq 2R_\sZ^2\cbraces{\frac{B_\sX(\bx^*, \bx^0)}{R_\sX^2} + \frac{B_\sY(\by^*, \by^0)}{R_\sY^2} + \frac{B_\sP(\bp^*, \bp^0)}{R_\sP^2} + \frac{B_\sQ(\bq^*, \bq^0)}{R_\sQ^2} + \frac{B_\sS(\bs^*, \bs^0)}{R_\sS^2} + \frac{B_\sZ(\bz^*, \bz^0)}{R_\sZ^2}} \\
    &\quad\leq 12R_\sZ^2.
\end{align*}
}
As a result, we get
\begin{align*}
    \norm{\mWx\hat\bx^N}_2\leq \frac{17L_\zeta}{NR_\sZ} = \frac{17L_\zeta}{N} \frac{\lambda_{\min}^+(\mWx)}{M_x\sqrt{2}} = \frac{17\sqrt{2}L_\zeta}{2NM_x} \lambda_{\min}^+(\mWx).
\end{align*}
Analogously we have
\begin{align*}
    \norm{\mWy\hat\by^N}_2\leq \frac{17L_\zeta}{NR_\sS} = \frac{17L_\zeta}{N} \frac{\lambda_{\min}^+(\mWy)}{M_y\sqrt{2}} = \frac{17\sqrt{2}L_\zeta}{2NM_y} \lambda_{\min}^+(\mWy).
\end{align*}

Let us now estimate the duality gap in objective function. Recall that
\begin{align*}
    &F(\hat{\bf x}^N, \hat{\bf p}^N, {\bf y}, {\bf q}) + \langle{\bf z}, {{\bf W}_{\bf x}}\hat{\bf x}^N\rangle + \langle\hat{\bf s}^N, {{\bf W}_{\bf y}}{\bf y}\rangle - F({\bf x}, {\bf p}, \hat{\bf y}^N, \hat{\bf q}^N) - \langle\hat{\bf z}^N, {{\bf W}_{\bf x}}{\bf x}\rangle - \langle{\bf s}, {{\bf W}_{\bf y}}\hat{\bf y}^N\rangle \\
    &\leq\frac{L_\zeta}{N} \left(\frac{B_{\mathcal{X}}({\bf x}, {\bf x}^0)}{R_{\mathcal{X}}^2} + \frac{B_{\mathcal{P}}({\bf p}, {\bf p}^0)}{R_{\mathcal{P}}^2} + \frac{B_{\mathcal{Y}}({\bf y}, {\bf y}^0)}{R_{\mathcal{Y}}^2} + \frac{B_{\mathcal{Q}}({\bf q}, {\bf q}^0)}{R_{\mathcal{Q}}^2} + \frac{B_{\mathcal{Z}}({\bf z}, {\bf z}^0)}{R_{{\mathcal{Z}}}^2} + \frac{B_{\mathcal{S}}({\bf s}, {\bf s}^0)}{R_{{\mathcal{S}}}^2}\right).
\end{align*}
Also recall that $B_{\mathcal{T}}({\bf t}, {\bf t}^0)\leq R_{\mathcal{T}}^2$ for ${\bf t}\in\{{\bf x}, {\bf p}, {\bf y}, {\bf q}\}$ and ${\mathcal{T}}\in\{{\mathcal{X}}, {\mathcal{P}}, {\mathcal{Y}}, {\mathcal{Q}}\}$, respectively.
Setting ${\bf z} = {\bf z}^0 = 0,~ {\bf s} = {\bf s}^0 = 0$ we obtain
\begin{align*}
    F(\hat{\bf x}^N, \hat{\bf p}^N, {\bf y}, {\bf q}) + \langle\hat{\bf s}, {{\bf W}_{\bf y}}{\bf y}\rangle - F({\bf x}, {\bf p}, \hat{\bf y}^N, \hat{\bf q}^N) - \langle\hat{\bf z}, {{\bf W}_{\bf x}}{\bf x}\rangle \leq \frac{4L_\zeta}{N}.
\end{align*}
Note that
\begin{align*}
    \max_{{{\bf W}_{\bf y}}{\bf y} = 0}&\left[F(\hat{\bf x}^N, \hat{\bf p}^N, {\bf y}, {\bf q}) + \langle\hat{\bf s}^N, {{\bf W}_{\bf y}}{\bf y}\rangle\right] = \max_{{{\bf W}_{\bf y}}{\bf y} = 0} F(\hat{\bf x}^N, \hat{\bf p}^N, {\bf y}, {\bf q}), \\
    \min_{{{\bf W}_{\bf x}}{\bf x} = 0}&\left[F({\bf x}, {\bf p}, \hat{\bf y}^N, \hat{\bf q}^N) + \langle\hat{\bf z}^N, {{\bf W}_{\bf x}}{\bf x}\rangle\right] = \min_{{{\bf W}_{\bf x}}{\bf x} = 0} F({\bf x}, {\bf p}, \hat{\bf y}^N, \hat{\bf q}^N).
\end{align*}
Then
\begin{align}\label{eq:naive_dual_gap}
    \max_{{{\bf W}_{\bf y}}{\bf y}=0} F(\hat{\bf x}^N, \hat{\bf p}^N, {\bf y}, {\bf q}) - \min_{{{\bf W}_{\bf x}}{\bf x}=0} F({\bf x}, {\bf p}, \hat{\bf y}^N, \hat{\bf q}^N)\leq \frac{4L_\zeta}{N}.
\end{align}
In the criterion above, the $\max$ and $\min$ operations are taken over consensus sets ${{\bf W}_{\bf x}}{\bf x} = 0$ and ${{\bf W}_{\bf y}}\by = 0$, but the variables $\hat{\bf x}, \hat{\bf y}$ may not satisfy consensus constraints. The criterion can be updated as follows.

Introduce $\bar{\hat{\bf x}}^N = \frac{1}{m}{\bf 1}{\bf 1}^\top\hat{\bf x}^N,~ \bar{\hat{\bf y}}^N = \frac{1}{m}{\bf 1}{\bf 1}^\top\hat{\bf y}^N$. Let us recall that gradients of $F$ are uniformly bounded (see Assumption 3.1 of the paper). Using convexity of $F$ w.r.t ${\bf x}$ we obtain
\begin{align*}
    F(\bar{\hat{\bf x}}^N, \hat{\bf p}^N, {\bf y}, {\bf q})
    &\leq F(\hat{\bf x}^N, \hat{\bf p}^N, {\bf y}, {\bf q}) - \langle\nabla F_{\bf x}(\bar{\hat{\bf x}}^N, \hat{\bf p}^N, {\bf y}, {\bf q}), \hat{\bf x}^N - \bar{\hat{\bf x}}^N)\rangle \\
    &\leq F(\hat{\bf x}^N, \hat{\bf p}^N, {\bf y}, {\bf q}) + \|\nabla F_{\bf x}(\bar{\hat{\bf x}}^N, \hat{\bf p}^N, {\bf y}, {\bf q})\|\cdot \|\hat{\bf x}^N - \bar{\hat{\bf x}}^N\|_2 \\
    &\leq F(\hat{\bf x}^N, \hat{\bf p}^N, {\bf y}, {\bf q}) + \sqrt{m}M_x\cdot \frac{\|{{\bf W}_{\bf x}}(\hat{\bf x}^N - \bar{\hat{\bf x}}^N)\|_2}{\lambda_{\min}^+({{\bf W}_{\bf x}})} \\
    &\overset{\text{\ding{172}}}{\leq} F(\hat{\bf x}^N, \hat{\bf p}^N, {\bf y}, {\bf q}) + \frac{17L_\zeta\sqrt m}{N} \frac{M_x}{R_{\mathcal{Z}} \lambda_{\min}^+({{\bf W}_{\bf x}})} \\
    &\overset{\text{\ding{173}}}{=} F(\hat{\bf x}^N, \hat{\bf p}^N, {\bf y}, {\bf q}) + \frac{17L_\zeta}{\sqrt{2}N}.
\end{align*}
Here in $\text{\ding{172}}$ we used the bound on consensus violation from Theorem 3.6 of the paper and in $\text{\ding{173}}$ we recalled the definition of $R_{\mathcal{Z}}$ from Lemma 3.2. Carrying out an analogical estimate of $F({\bf x}, {\bf p}, \bar{\hat{\bf y}}^N, \hat{\bf q}^N)$, we obtain
\begin{align*}
    \max_{{{\bf W}_{\bf y}}{\bf y} = 0} F(\bar{\hat{\bf x}}^N, \hat{\bf p}^N, {\bf y}, {\bf q}) - \min_{{{\bf W}_{\bf x}}{\bf x} = 0} F({\bf x}, {\bf p}, \bar{\hat{\bf y}}^N, \hat{\bf q}^N)
    \leq \frac{(4 + 17\sqrt{2})L_\zeta}{N}.
\end{align*}

\subsection{Discussion for the Euclidean case}

Theorem \ref{Th:distr_MP} ensures a $O(1/\eps)$ convergence guarantee but still does not provide explicit dependencies on the network characteristics. We derive these dependencies in the simplified Euclidean case.
\begin{corollary}\label{cor:euclidean_case}
Let \add{$L \triangleq \max\{ L_{(\bx,\bp)(\bx,\bp)}, L_{(\bx,\bp)(\by,\bq)}, L_{(\by,\bq)(\bx,\bp)}, L_{(\by,\bq)(\by,\bq)} \}$, $M \triangleq \max\{M_x, M_y\}$ and $\chi = \max\{\chi(W_x), \chi(W_y) \}$}. Let us also upper bound the sizes of the sets: for $\sT_i\in\braces{\bar\sX, \sP_i, \bar\sY, \sQ_i}$, let $R_{\sT_i}^2\leq R^2$.
For $\bt\in\braces{\bx, \bp, \by, \bq}$ let $\bt = (t_1^\top, \ldots, t_m^\top)^\top$ and introduce the average values $\hat t^N_{av} = \frac{1}{m}\sum_{i=1}^m \hat t^N_i$. Then after $N$ iterations of Algorithm \ref{alg:mirror_prox_distr} it holds
\begin{align*}
    &\add{\max_{\substack{y\in\bar\sY,\\ \bq\in\sQ}}\frac{1}{m} \sum_{i=1}^m f_i(\hat x_{av}^N, \hat p_i^N, y, q_i) - \min_{\substack{x\in\bar\sX,\\ \bp\in\sP}} \frac{1}{m} \sum_{i=1}^m f_i(x, p_i, \hat y_{av}^N, \hat q_i^N)} = O\cbraces{\max\braces{\frac{LR^2}{N}, \frac{MR\chi}{N}}}, \\
    &\cbraces{\frac{1}{m}\sum_{i=1}^m \norm{\hat x_i^N - \hat x_{av}^N}_2^2}^{1/2} = O\cbraces{\max\braces{\frac{LR^2}{MN}, \frac{R\chi}{N}}}, \\
    &\cbraces{\frac{1}{m}\sum_{i=1}^m \norm{\hat y_i^N - \hat y_{av}^N}_2^2}^{1/2} = O\cbraces{\max\braces{\frac{LR^2}{MN}, \frac{R\chi}{N}}}.
\end{align*}
\end{corollary}


Corollary \ref{cor:euclidean_case} states the convergence rate in terms of number of iterations of Algorithm \ref{alg:mirror_prox_distr}. Every iteration corresponds to one communication round and one oracle call per node, resulting in equivalent oracle and communication complexities.
\begin{remark}\label{rem:reaches_lower_bounds}
    Let assumptions of Corollary \ref{cor:euclidean_case} be satisfied. Additionally, let the critical point of each $f_i(x_i, p_i, y_i, q_i)$ lie in the set $\bar\sX\times\sP_i\times\bar\sY\times\sQ_i$ ($i = 1, \ldots, m$). Then we have $M\leq LR$. The Chebyshev acceleration technique (see i.e. \cite{li2021accelerated}) makes it possible to reduce $\chi$ to $O(1)$ at the cost of performing $\sqrt{\chi}$ communication rounds at each iteration. As a result, Algorithm \ref{alg:mirror_prox_distr} requires $O(LR^2/\eps)$ oracle calls per node and $O(LR^2\sqrt{\chi}/\eps)$ communication rounds yield a point $(\hat\bx, \hat\bp, \hat\by, \hat\bq)$ such that
    \begin{align*}
        \add{\max_{ \substack{y\in\bar\sY,\\ \bq\in\sQ}}\frac{1}{m} \sum_{i=1}^m f_i(\hat x_{av}^N, \hat p_i^N, y, q_i) - \min_{\substack{x\in\bar\sX,\\ \bp\in\sP}} \frac{1}{m} \sum_{i=1}^m f_i(x, p_i, \hat y_{av}^N, \hat q_i^N)} \leq \eps.
    \end{align*}
    As will be shown in Section \ref{sec:lower_bounds}, these oracle and communication complexities are optimal.
\end{remark}

\section{Lower bounds for distributed saddle point problems}\label{sec:lower_bounds}


In this section we present results on lower bounds. In obtaining them, we focus on the Euclidean case. Before proceeding directly to the lower bounds, we need to identify the class of algorithms for which \revoms{they} are valid.
To describe this class of first-order methods, we use a similar definition of Black-Box procedure as in \cite{scaman2017optimal}. We assume that one local iteration costs $t$ time units, and the communication round costs $\tau$ time units. 
Additionally, information can be transmitted only along the undirected edge of the network. Communications and local updates can take place in parallel and asynchronously. More formally, it can be described as follows.

\begin{assumption} \label{ass_bbp}

Each machine $i$ can collect past values for $x,p$ and $y,q$ in internal memories $H^x_{i,T} \subseteq \mathcal{R}^{d_x}, H^p_{i,T} \subseteq \mathcal{R}^{d_{p_i}}, H^y_{i,T} \subseteq \mathcal{R}^{d_y}, H^q_{i,T} \subseteq \mathcal{R}^{d_{q_i}}$ at time $T \geq 0$.

\noindent\textbf{Initialization.} We assume that all nodes start with $x^0 = 0$, $p^0_i = 0$, $y^0 = 0$, $q^0_i = 0$, then $ H^x_{i,0} = \{0\}$, $H^p_{i,0} = \{0\}$, $ H^y_{i,0} = \{0\}$, $H^q_{i,0} = \{0\}$ for all $i$.

\noindent\textbf{Connection.} The internal memories are updated either by communications $\bar H^x_{i,T}, \bar H^y_{i,T}$ or local steps $\hat H^x_{i,T}, \hat H^p_{i,T}, \hat H^y_{i,T}, \hat H^q_{i,T}$:
\begin{equation*}
    H^x_{i,T}  \subseteq  \bar H^x_{i,T} \cup \hat H^x_{i,T},~~~ H^p_{i,T}  \subseteq  \hat H^p_{i,T},~~~
   H^y_{m,T}  \subseteq  \bar H^y_{m,T} \cup \hat H^y_{m,T}, ~~~ H^q_{i,T}  \subseteq \hat H^q_{i,T}.
\end{equation*}

\noindent\textbf{Communication.} Each machine $i$ can start to communicate with neighbors at time $T-\tau$, then, after $\tau$ units of time, the update of memories using communications can be rewritten as follows:
\begin{eqnarray*}
   \bar H^x_{i,T}  \subseteq  \text{span}\left\{\bigcup_{(j,i) \in \mathcal{E}} H^x_{j,T-\tau}\right\}, ~~~
   \bar H^y_{i,T}  \subseteq  \text{span}\left\{\bigcup_{(j,i) \in \mathcal{E}} H^y_{j,T-\tau}\right\}.
\end{eqnarray*}


\noindent\textbf{Local computation.} Locally, each device $i$ can find $\nabla_x f_i(x,p,y,q)$, $\nabla_{p_i} f_i(x,p,y,q)$, $\nabla_y f_i(x,p,y,q)$, $\nabla_{q_i} f_i(x,p,y,q)$ for all $x \in H^x_{i,T-t}, p \in H^p_{i,T-t}, y \in H^y_{i,T-t}, q \in H^q_{i,T-t}$. Then, the local update is
\begin{align*}
   \hat H^x_{i,T} &\subseteq  \text{span}\left\{ x, \nabla_x f_i(x',p',y',q'), {\rm Mirr}(\gamma \nabla_x f_i(x',p',y',q');x;\sX)\right\}, \\
   \hat H^p_{i,T} &\subseteq  \text{span}\left\{ p, \nabla_{p_i} f_i(x',p',y',q'), {\rm Mirr}(\gamma \nabla_{p_i} f_i(x',p',y',q');p;\sP_i)\right\}, \\
   \hat H^y_{i,T} &\subseteq  \text{span}\left\{ y, \nabla_y f_i(x',p',y',q'), {\rm Mirr}(\gamma \nabla_{y} f_i(x',p',y',q');y;\sY)\right\}, \\
   \hat H^q_{i,T} &\subseteq  \text{span}\left\{ q, \nabla_q f_i(x',p',y',q'), {\rm Mirr}(\gamma \nabla_{q_i} f_i(x',p',y',q');q;\sQ_i)\right\},
\end{align*}
for all $x,x' \in H^x_{i,T-t}, p,p' \in H^p_{i,T-t}, y,y' \in H^y_{i,T-t}, q,q' \in H^q_{i,T-t}$ and all $\gamma \in \R$. ${\rm Mirr}$ is defined above in \eqref{eq:mirror}.

\noindent\textbf{Output.} When the algorithm ends (after $T_0$ units of time), we have $m$ local outputs $x^{f}_i \in H^x_{i,T_0} $, $p^{f}_i \in H^p_{i,T_0} $, $y^{f}_i \in H^y_{i,T_0} $, $q^{f}_i \in H^q_{i,T_0} $. 
Suppose the final global output is calculated as follows: 
\begin{equation*}
   x^{f} \in \text{span}\left\{\bigcup_{i=1}^m H^x_{i,T_0} \right\},~~~ p^f_i \in H^p_{i,T_0},~~~y^{f} \in \text{span}\left\{\bigcup_{i=1}^m H^y_{i,T_0} \right\},~~~q^f_i \in H^q_{i,T_0}.
\end{equation*}
\end{assumption}
Note that the above procedure do not communicate on variables $p_i$ and $q_i$, because they are local and unique for each node, in contrast with variables $x$ and $y$. It is easy to check that Algorithm \ref{alg:mirror_prox_distr} also satisfies the definition of the procedure above.

The idea of proving any lower bounds for optimization problems comes from the first results on lower bounds for convex minimization \cite{nesterov2004introduction}. The essence is to give an example of a "bad" function.
In the case of distributed problems, it is also necessary to "badly"  divide this function between $m$ devices \cite{scaman2017optimal}.
We consider a special case of problem \eqref{eq:problem_initial} (the variables $\bp$ and $\bq$ are dummies, and the sets $\sP$ and $\sQ$ are empty):
\begin{align}
    \label{eq:fun_lower_bounds}
    \min_{x\in\bar\sX} \max_{y\in\bar\sY} f(x,y) = \frac{1}{m}\sum_{i=1}^m f_i(x, y),
\end{align}
where $\bar\sX, \bar\sY \subset \R^d$ are balls with the centers at point $0$ and some size $R$ (see \eqref{eq:R_def_for_prox}). We will define the size $R$ and $n$ more precisely in the proof (see Appendix \ref{app:lower_bounds}). Since we concentrate entirely on the Euclidean case, therefore the operator ${\rm Mirr}$ in the procedure definition above can be replaced by the Euclidean projection operator: ${\rm Mirr}(g,x,X) = \text{proj}_{X} (x - g)$ with $\text{proj}_{X} (y) = \arg\min_{z \in X} \|z-y \|_2$.

The global function $f(x,y)$ in \eqref{eq:fun_lower_bounds} is exactly an example of a bilinear function from the work on non-distributed lower bounds for the strongly convex-strongly concave case \cite{zhang2019lower}, but with strong convexity and strong concavity constants $\mu \sim \tfrac{\varepsilon}{R^2}$ (see \eqref{t144} in Appendix \ref{app:lower_bounds}). In fact, if we write the dual function for this $f(x,y)$ from \eqref{t144}, we have exactly the problem from \cite{nesterov2004introduction}, but with a smoothness constant $\sim\tfrac{L^2}{\mu}$ (see Lemma \ref{lem2} in Appendix \ref{app:lower_bounds}). The problem \eqref{eq:fun_lower_bounds} is decomposed into summands $f_i$ in such a way that if we want to get closer to the solution, we need to communicate (see also \cite{scaman2017optimal}). This idea is explained in more detail in Appendix \ref{app:lower_bounds}. Here we only give a simplified version of the main theorem.

\begin{theorem}  \label{Th:lower_bound}
For any given $L >0$, and $\chi \geq 1$, there exist \\ 
$\bullet$ functions $f_i$ defined on $\bar\sX\times\sP_i\times\bar\sY\times\sQ_i$ ($i = 1, \ldots, m$), such that they satisfy Assumption \ref{assum:convex_compact}, the averaged function $f= \tfrac{1}{m} \sum_{i=1}^m f_i$  is $L$-smooth and all the sets $\bar\sX\times\sP_i\times\bar\sY\times\sQ_i$ have size $R$, \\
$\bullet$ a connected graph $\mathcal{G}$ and a gossip matrix $W$ associated with it. \revoms{The condition number of the gossip matrix equals $\chi(W) = \chi$}.\\
Then for any decentralized algorithm satisfying Assumption~\ref{ass_bbp} the minimal time $T$ for yielding a $\eps$-solution $(\hat x, \hat\bp, \hat y, \hat\bq)$ (i.e. $\ds\max_{y \in \bar\sY, \bq\in\sQ}f(\hat x, \hat\bp, y, \bq) - \min_{x \in \bar\sX, \bp\in\sP} f(x, \bp, \hat y, \hat\bq) \leq \varepsilon$) is given by
\[
T = \Omega\left( \frac{L R^2}{\varepsilon} (t + \sqrt{\chi}\tau)\right).
\]

\end{theorem}
Theorem \ref{Th:lower_bound} shows that the minimal number of oracle calls per node is $\Omega(LR^2/\eps)$ and the minimal number of comunication rounds is $\Omega(LR^2\sqrt{\chi(W)}/\eps)$. In Remark \ref{rem:reaches_lower_bounds} we showed that in the euclidean case Algorithm \ref{alg:mirror_prox_distr} with Chebyshev acceleration reaches these lower bounds when the critical point of each $f_i(x_i, p_i, y_i, q_i)$ lies in $\bar\sX\times\sP_i\times\bar\sY\times\sQ_i$ ($i = 1, \ldots, m$).

\section{Application to the Wasserstein barycenters}\label{sec:wb}

Now we show the benefits of representing  some convex problems as  convex-concave problems on the example of the Wasserstein barycenter (WB) problem and  solve it by the DMP algorithm. Similarly to Section \eqref{sec:distributed_mirror_prox}, we consider a SPP in proximal setup and introduce Lagrangian multipliers for the common variables. However, in the Section \ref{sec:distributed_mirror_prox} we obtained results in a general setup without additional knowledge about cost functions and sets. On the contrary, in this section we utilize the special structure of the WB problem and introduce slightly different norms. After that, we get a convergence guarantee by applying Theorem~\ref{Th:distr_MP}.

 The  (fixed-support) WB of  probability measures $y_1, y_2\,..., y_m $ from the probability simplex $\Delta_n$ is a solution of the following  finite-dimensional optimization problem
\begin{equation}\label{eq:W_bary}
   \min_{x \in \Delta_n} \frac{1}{m} \sum_{i=1}^m \mathcal  W(x, y_i),
\end{equation}
where $\mathcal W(x,y) = \min_{\pi \in  U(x,y)} \langle C, \pi \rangle$
is  optimal transport  between two histograms $x, y \in \Delta_n$, $C \in \R^{n\times n}_+$ is a given  ground cost matrix, and
$\pi$ is a transport plan which belongs to the transportation polytope $ U =\{\pi \in \mathbb R_{+}^{n\times n}, \pi\boldsymbol{1} = x, \pi^\top\boldsymbol{1} = y\}$.  
As  we consider a general cost matrix $C$, the optimal transport problem  is more general than the problem defining a Wasserstein distance. 

\subsection{Saddle point problem formulation}
Next we reformulate the WB problem \eqref{eq:W_bary} as a saddle point problem  \cite{dvinskikh2020improved}. To do so, we introduce
  stacked column vector $b_i = [
x^\top, y_i^\top]^\top$,
vectorized cost matrix $d$ of $C$, vectorized transport plan $p \in \Delta_{n^2}$ of $\pi$,  the incidence matrix $A=\{0,1\}^{2n\times n^2}$, vectors $q_i\in [-1,1]^{2n}$ for all $i=1,...,m$. Then the problem \revoms{\eqref{eq:W_bary}} can be equivalently rewritten as 
\begin{align}\label{eq:def_saddle_prob}
    \min_{ \substack{x\in \Delta_n, \\
    p_1,\dots, p_m\in \Delta_{n^2}
    } } \max_{q_1,\dots, q_m\in [-1,1]^{2n}}  \frac{1}{m}\sum_{i=1}^m  f_i(x,p_i,q_i) \triangleq d^\top p_i +2\|d\|_\infty\left(q_i^\top Ap_i -b_i^\top q_i\right).
\end{align}
Clearly, this problem fits the template \eqref{eq:problem_initial}. Then we can  rewrite  it for stacked column vectors up to a mulplicative term $1/m$ to fit the template \eqref{eq:problem_reform}
\begin{align}\label{eq:alm_distrNew}
         \min_{ \substack{ \p \in \Pp, \x \in \X}}   ~ \max_{ \q \in \sQ}
          F(\p,\x,\q) \triangleq 
       \boldsymbol d^\top \p +2\|d\|_\infty\left(\q^\top\boldsymbol A \p -\b^\top \q\right)\text{ s.t.: } x_1=x_2=\ldots=x_m.
        \end{align}
        where  $\x = [x_1^\top\in \Delta_n,\dots,x_m^\top \in \Delta_n]^\top  \in \X \triangleq \prod^m \Delta_{n}$ (here $\prod^m \Delta_{n}$ is the Cartesian product of $m$ simplices), $\p = [p_1^\top \in \Delta_{n^2},\dots,p_m^\top \in \Delta_{n^2}]^\top \in \Pp  \triangleq \prod^m \Delta_{n^2} $,
and  $\q = [q_1^\top,\ldots,q_m^\top]^\top \in \sQ \triangleq [-1,1]^{2nm}$, $\b = [x_1^\top, y_1^\top, ..., x_m^\top, y_m^\top]^\top$, 
$\boldsymbol d = [d^\top, \ldots, d^\top ]^\top $ and
   $\boldsymbol A = {\rm diag}\{A, ..., A\} \in \{0,1\}^{2mn\times mn^2}$ is  the block-diagonal matrix.

Finally, to enable distributed computation of  problem \eqref{eq:alm_distrNew}, we replace the constraint $ x_1 = \dots = x_m$
by $\W \x = 0$ (matrix $\W \triangleq  W \otimes I_n$, where $W$ is defined in  \eqref{eq:matrx_def}).  Then we  introduce   the Lagrangian dual variable $\z = [z_1^\top, \ldots, z^\top_m]^\top \in \Z \triangleq \R^{nm}$ corresponding to the constraint $\W \x=0$ and rewrite the problem \eqref{eq:alm_distrNew}  as follows
\begin{align}\label{eq:def_saddle_prob2}
  \min_{\bu \in \sU} \max_{\bv \in A_\sV} 
    &~S(\u,\v)\triangleq   \boldsymbol d^\top \p +2\|d\|_\infty\left(\q^\top\boldsymbol A \p -\b^\top \q \right) +\la \z,\W\x \ra,
\end{align}
\add{where  $\u = (\x,\p), \v = (\q,\z)$ and $\U \triangleq \sX \times \sP $, $\sA_\sV \triangleq \sQ \times \R^{nm}$. 
 Problem \eqref{eq:def_saddle_prob2}  fits the form \eqref{eq:saddle_point}, hence  it can be solved by the DMP.
 }

\subsection{Convergence guarantee }

Before presenting the complexity bound for the DMP algorithm to solve the SPP  \eqref{eq:def_saddle_prob2}, we firstly 
 bound the norm of dual variable $\bz$ similarly to the general setup in Section \ref{sec:distributed_mirror_prox}. 

\begin{lemma}\label{lemma:constrained_lagrange_multipiers_wb}
    Let $R_\sZ^2 = \add{8}mn\max_{i,j}C_{ij}^2/(\lambda_{\min}^+(W))^2$. Then there exists a saddle point $(\bx^*, \bp^*, \bq^*, \bz^*)$ of problem \eqref{eq:def_saddle_prob2} such that $\|\bz^*\|_2\leq R_\sZ$.
\end{lemma}
This lemma   is  essentially a particular case of Lemma \ref{lemma:constrained_lagrange_multipliers}  for the WB problem. 
\begin{lemma}\label{lm:Lipsch1}
Objective $S(\u,\v)$ in \eqref{eq:def_saddle_prob2} is $(L_{\u\u},L_{\u\v},L_{\v\u},L_{\v\v})$-smooth with $L_{\u\u} = L_{\v\v} = 0$ and
\add{$L_{\u\v} = L_{\v\u} =  8m\sqrt{2n\ln n}\max_{i,j}C_{ij} \chi$} w.r.t. norms 
\add{\[
\|\u\|^2_{\U} = \frac{1}{m \ln n} \sum_{i=1}^m\left(  \|x_i\|^2_1 +\frac{1}{2} \|p_i\|^2_1\right) ~ \text{ and } ~ \|\v\|^2_{\V} = \frac{1}{mn}\left( \|\q\|^2_2 + \frac{(\lambda_{\min}^+(W))^2\|\z\|^2_{2}}{8 \max_{i,j}C_{ij}^2 } \right).
\]}
\end{lemma}


\add{Then similarly to \eqref{eq:def_L_zeta} $L_\zeta $  can be defined as follows
\begin{align}\label{eq:def_L_zetaWB}
L_\zeta = 2\max\{L_{\bu\bu},L_{\bu\bv},L_{\bv\bu},L_{\bv\bv}\}= 16m\sqrt{2n\ln n}\max_{i,j}C_{ij} \chi.
\end{align}}
These two Lemmas proved in Appendix
\ref{app:wb} together with \eqref{eq:def_L_zetaWB} make it possible to obtain the  convergence guarantee for  
the DMP algorithm (see  Algorithm \ref{MP_WB_distr}  in Appendix \ref{app:algWB}).    

\add{
\begin{theorem}\label{Th:first_alg11}

After 
$ N={(4 + 17\sqrt{2})L_\zeta}/{(m\e)}$
iterations, Algorithm \ref{MP_WB_distr}   with stepsize 
$\alpha=1/L_\zeta$, where $L_{\zeta}$ is defined in \eqref{eq:def_L_zetaWB},
yields output $\hat  \x^N = [(\hat x_1^N)^\top, \dots, (\hat x_m^N)^\top]^\top$ such that
\begin{align*}
     \frac{1}{m} \sum_{i=1}^m \mathcal  W(\hat x^N_{av}, y_i) -  \min_{x \in \Delta_n} \frac{1}{m} \sum_{i=1}^m \mathcal  W(x, y_i) \leq \e ~ ~ \text{ and } ~ ~ 
        \norm{\mW\hat\bx^N}_2 = O\left(  \frac{\sqrt{m}  \lambda_{\min}^+(W) \e}{\sqrt{n}\max_{i,j}C_{ij}}
        \right),
\end{align*}
where  $\hat x^N_{av} = \frac{1}{m}\sum_{i=1}^m \hat x_i^N$. 
The per node complexity of Algorithm \ref{MP_WB_distr} is 
\[O\left(\frac{n^2}{\e}\sqrt{n\ln n}  \max_{i,j}C_{ij}^{2} \chi  \right).\]
\end{theorem}}
The proof of  follows from Theorem \ref{Th:distr_MP} and is given in the Appendix 
\ref{app:wb}.

\subsubsection{\add{Discussion of the convergence results.}}

We comment on the   complexity of the DMP algorithm compared to the existing state-of-the-art methods: the iterative Bregman projections (IBP) algorithm, its accelerated versions and primal dual algorithm (ADCWB), see  Table \ref{Tab:distr_comp}. All of these methods use entropic regularization of Wasserstein metric with parameter $\gamma$ which must be taken proportionally to accuracy $\e$.

\begin{table}[ht!]
\centering
\caption{\add{ Distributed algorithms  for the WB problem } }
\begin{threeparttable}
\begin{tabular}{lll}
Algorithm   &  Per node complexity$^{(1)}$ \\
 \midrule
 IBP \cite{benamou2015iterative,kroshnin2019complexity}    &  
$n^2/(\gamma\e) $    \\
Accelerated IBP  \cite{guminov2021combination}              & 
   $ n^2 \sqrt{n}/\sqrt{\gamma \e} $ 
   \\
  FastIBP    \cite{lin2020fixed}       & 
$n^2\sqrt[3]{n} /(\gamma \e )^{2/3}$   
\\
ADCWB \cite{dvurechenskii2018decentralize,dvinskikh2019primal} &  $ n^2\sqrt{ n }/\sqrt{\gamma\e}$ \\
\midrule
    \textbf{DMP} (this work)    & 
   $ n^2\sqrt{ n}  / \e $
   \\
   \bottomrule
\end{tabular}
\begin{tablenotes}
\item[]  $^{(1)}$ The bounds are obtained by using the Chebyshev acceleration.
\end{tablenotes}    
\end{threeparttable}
\label{Tab:distr_comp}
\end{table}

\subsection{Numerical experiments}
In this section, we tested the DMP algorithm for 
computing WB (see Algorithm \ref{MP_WB_distr} in the Appendix \ref{app:algWB}) on Gaussians measures and the \textit{notMNIST} dataset.

 \subsubsection{Evaluation on different network architectures.} We demonstrate the performance of the DMP algorithm on different network architectures with different conditional number $\chi$:  complete graph,  star graph, cycle graph and the Erd\H{o}s-R\'enyi random graphs with the probability of edge creation $p=0.5$ and $p=0.4$  under the random seed $=10$. As the true barycenter of Gaussian measures can be  calculated  theoretically \cite{delon2020wasserstein}, we use them to study the convergence of the DMP to the non-optimality gap. 
 We randomly generated  10 Gaussian measures   with equally spaced support of 30 points in $[-10, -10]$,  mean from  
 $[-5, 5]$ and variance from $[0.8, 1.8]$. 
\add{ 
Figure \ref{fig:DMP conv} supports Theorem \ref{Th:first_alg11} and presents the convergence of the DMP to the function non-optimality gap and the distance to the consensus. The smaller the condition number, the faster the convergence.
}

\begin{figure}[ht!]
\centering
\begin{minipage}{0.5\textwidth}
  \centering
\includegraphics[width =  \textwidth ]{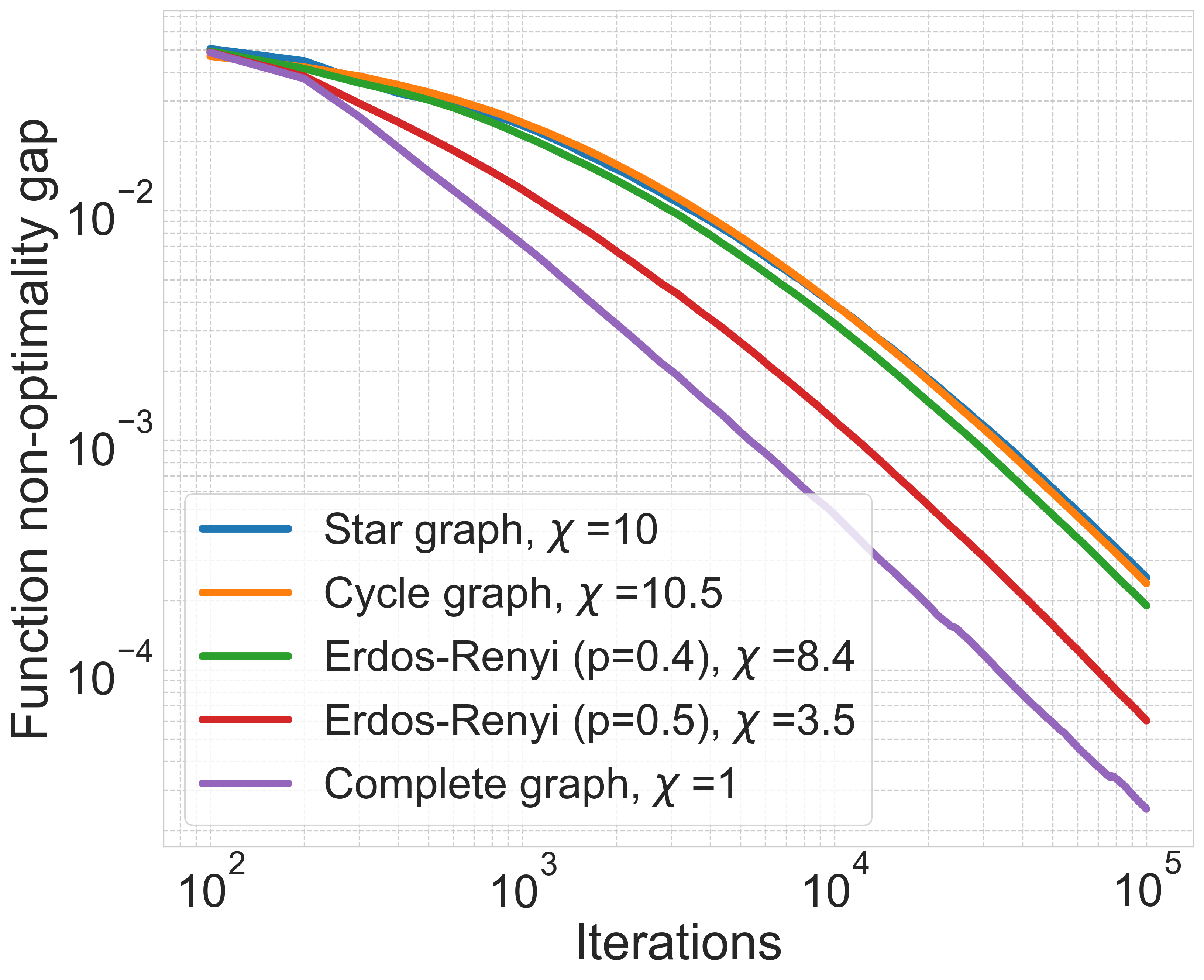}
\end{minipage}%
\begin{minipage}{0.5\textwidth}
  \centering
\includegraphics[width =  \textwidth ]{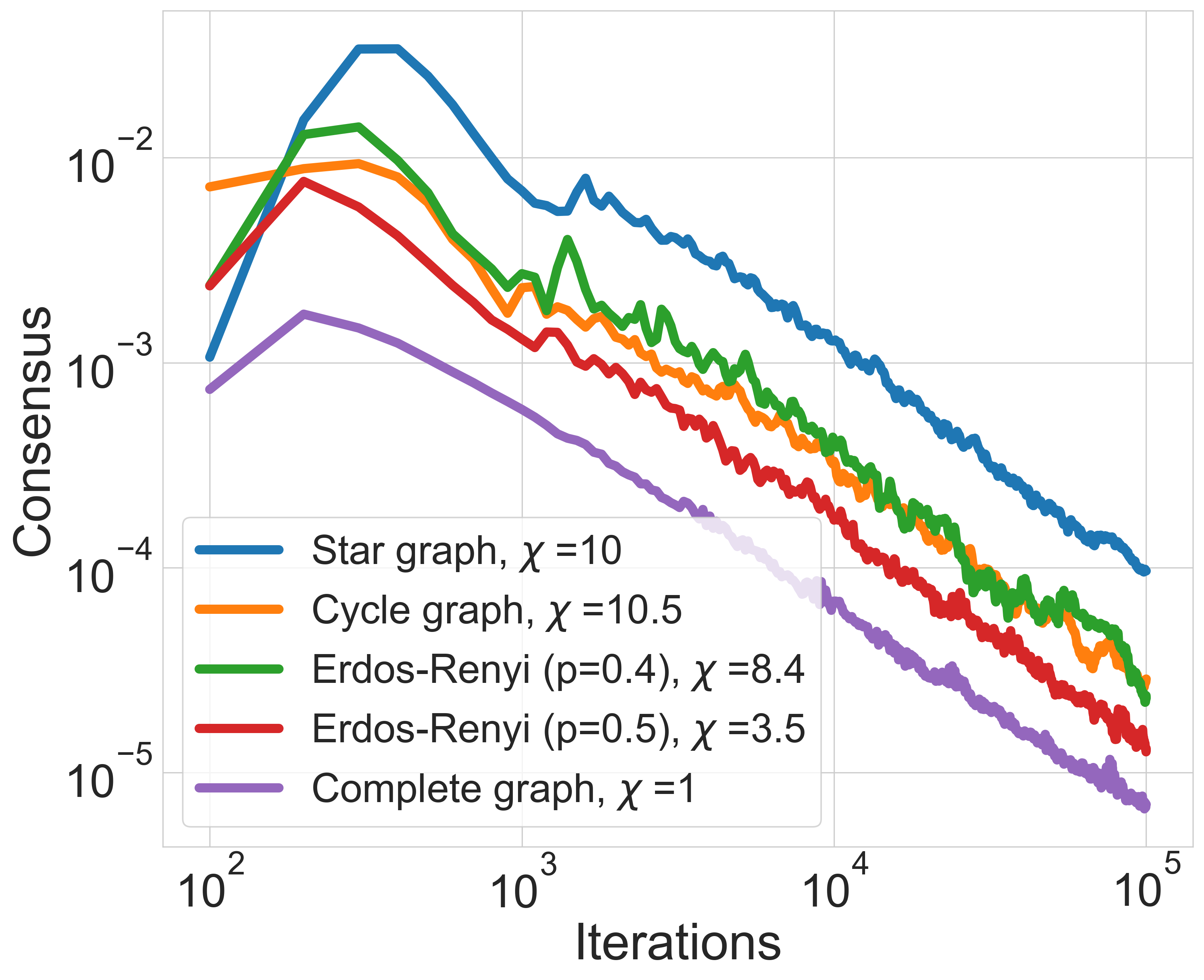}
\end{minipage}
\caption{\add{ Convergence of the DMP algorithm  on different network architectures.   }}
\label{fig:DMP conv}
\end{figure}


\subsubsection{Comparison with existing methods.}
Next we compare the DMP with 
  the most popular modern algorithms: the IBP   \cite{benamou2015iterative} and the ADCWB  \cite[Algorithm 4]{dvurechenskii2018decentralize}.
We illustrate  a well-known issue: 
 numerical instability of regularized algorithms with a small value of
parameter $\gamma$ to solve the WB problem. 
 We ran the IBP and the ADCWB algorithms with different values of the regularization parameter $\gamma $ starting from $\gamma = 0.1$ and gradually decreasing its value to $\gamma =10^{-4}$.   The number of iterations was taken proportionally to $1/\gamma$ in the IBP and proportionally to $1/ \sqrt \gamma$ in the ADCWB according to the theoretical bounds.  Figure \ref{fig:comparis_conv} shows that for a certain value of $\gamma$ (depending on the the experiment set and the number of method iterations) the regularized algorithms diverge.  Our unregularized  DMP algorithm is capable to achieve any accuracy, the more iteration the better accuracy. We ran it to achieve about $10^{-8}$ accuracy, probably the machine accuracy.

\begin{figure}[ht!]
 \centering
  \includegraphics[width = 0.6\textwidth]{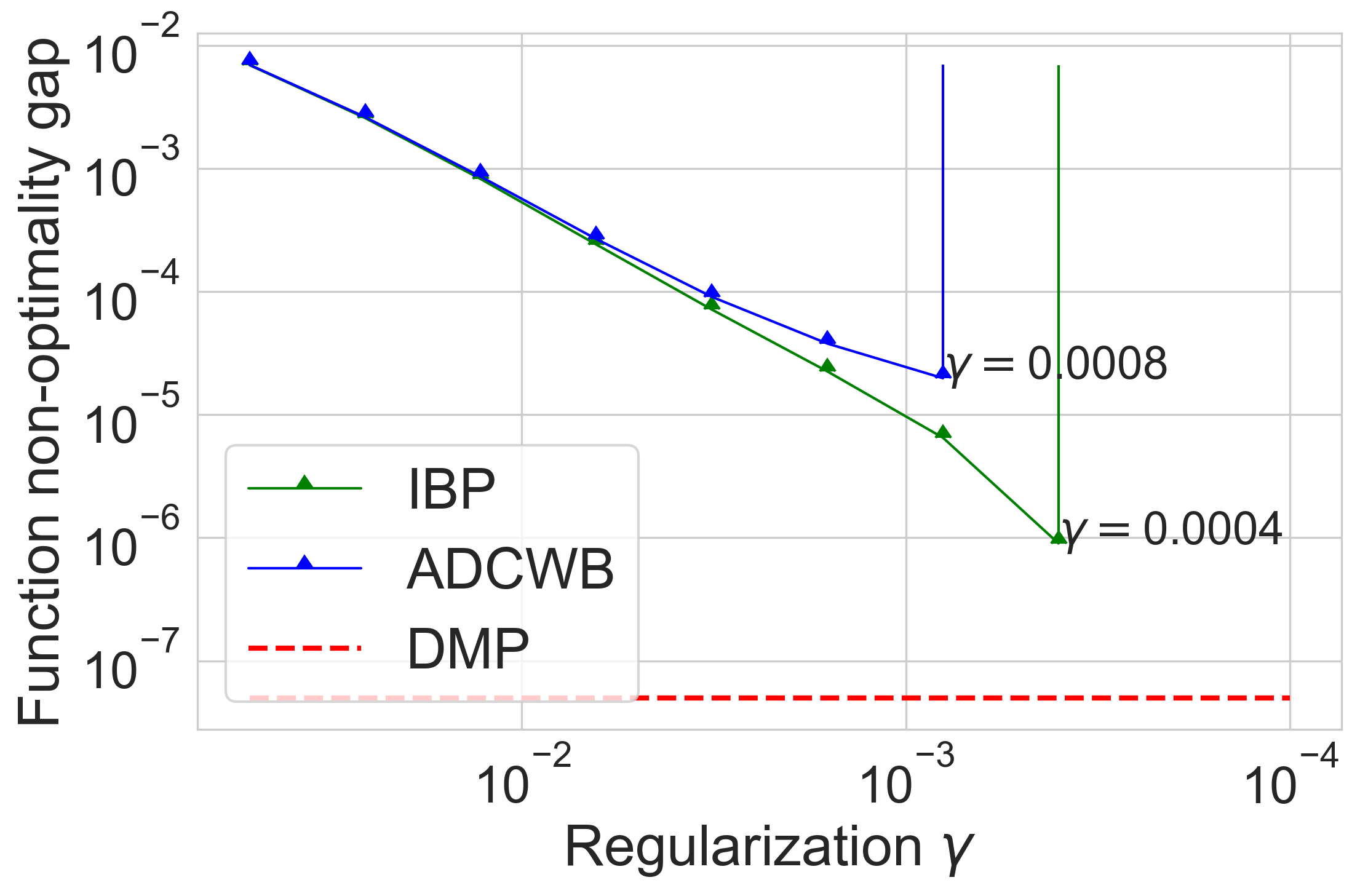}
  \caption{\add{Numerical instability of entropy-regularized based methods}}
\label{fig:comparis_conv}
\end{figure}

We also illustrate the non-stability of the IBP and the ADCWB algorithms run with $\gamma=10^{-4}$ on the \textit{notMNIST} dataset, in particular for the  letter `B' presented in various fonts. Figure \ref{fig:notMNIST} shows the best barycenters before the regularized algorithms will  diverge. 
\begin{figure}[ht!]
\centering
\begin{subfigure}[b]{.2\textwidth}
  \centering
  \includegraphics[width =  \textwidth ]{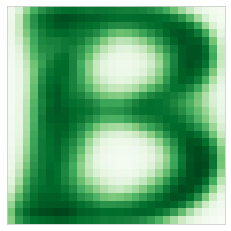}
\end{subfigure}
\begin{subfigure}[b]{.2\textwidth}
  \centering
  \includegraphics[width =  \textwidth ]{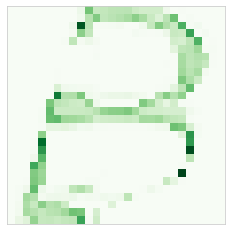}
\end{subfigure}
\begin{subfigure}[b]{.2\textwidth}
  \centering
  \includegraphics[width =  \textwidth ]{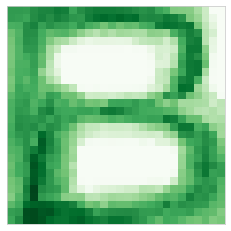}
\end{subfigure}
\caption{\add{Barycenters of letter `B' found by the DMP (left), the IBP  (middle) and  the ADCWB (right).}}
\label{fig:notMNIST}
\end{figure}

\section{Conclusion}\label{sec:conclusions}

We proposed a decentralized method for saddle point problems based on non-Euclidean Mirror-Prox algorithm. Our reformulation is built upon moving the consensus constraints into the problem by adding Lagrangian multipliers. As a result, we get a common saddle point problem that includes both primal and dual variables. After that, we employ the Mirror-Prox algorithm and bound the norms of dual variables at solution to assist the theoretical analysis. Finally, we demonstrate the effectiveness of our approach on the problem of computing Wasserstein barycenters  (both theoretically and numerically).


\bibliography{references}

\begin{thebibliography}{10}

\bibitem{arjevani2015communication}
Y.~Arjevani and O.~Shamir.
\newblock Communication complexity of distributed convex learning and
  optimization.
\newblock {\em arXiv preprint arXiv:1506.01900}, 2015.

\bibitem{auzinger2011iterative}
W.~Auzinger and J.~Melenk.
\newblock Iterative solution of large linear systems.
\newblock {\em Lecture notes, TU Wien}, 2011.

\bibitem{bach2011optimization}
F.~Bach, R.~Jenatton, J.~Mairal, and G.~Obozinski.
\newblock Optimization with sparsity-inducing penalties.
\newblock {\em arXiv preprint arXiv:1108.0775}, 2011.

\bibitem{bach2008convex}
F.~Bach, J.~Mairal, and J.~Ponce.
\newblock Convex sparse matrix factorizations.
\newblock {\em arXiv preprint arXiv:0812.1869}, 2008.

\bibitem{BenTal2009:book}
A.~Ben-Tal, L.~E. Ghaoui, and A.~Nemirovski.
\newblock {\em Robust Optimization}.
\newblock Princeton University Press, 2009.

\bibitem{ben2011lectures}
A.~Ben-Tal and A.~Nemirovski.
\newblock Lectures on modern convex optimization (2012).
\newblock {\em Online version: http://www2. isye. gatech. edu/\~{}
  nemirovs/Lect\_ ModConvOpt}, 2011.

\bibitem{benamou2015iterative}
J.-D. Benamou, G.~Carlier, M.~Cuturi, L.~Nenna, and G.~Peyré.
\newblock Iterative bregman projections for regularized transportation
  problems.
\newblock {\em SIAM Journal on Scientific Computing}, 37(2):A1111--A1138, 2015.

\bibitem{beznosikov2020local}
A.~Beznosikov, V.~Samokhin, and A.~Gasnikov.
\newblock Local sgd for saddle-point problems.
\newblock {\em arXiv e-prints}, pages arXiv--2010, 2020.

\bibitem{beznosikov2021distributed}
A.~Beznosikov, G.~Scutari, A.~Rogozin, and A.~Gasnikov.
\newblock Distributed saddle-point problems under data similarity.
\newblock {\em Advances in Neural Information Processing Systems}, 34, 2021.

\bibitem{boissard2015distribution}
E.~Boissard, T.~Le~Gouic, and J.-M. Loubes.
\newblock Distribution’s template estimate with wasserstein metrics.
\newblock {\em Bernoulli}, 21(2):740--759, 05 2015.

\bibitem{bubeck2014theory}
S.~Bubeck.
\newblock Theory of convex optimization for machine learning.
\newblock {\em arXiv preprint arXiv:1405.4980}, 15, 2014.

\bibitem{chambolle2011first}
A.~Chambolle and T.~Pock.
\newblock A first-order primal-dual algorithm for convex problems with
  applications to imaging.
\newblock {\em Journal of mathematical imaging and vision}, 40(1):120--145,
  2011.

\bibitem{del2019robust}
E.~Del~Barrio, J.~A. Cuesta-Albertos, C.~Matr{\'a}n, and A.~Mayo-{\'I}scar.
\newblock Robust clustering tools based on optimal transportation.
\newblock {\em Statistics and Computing}, 29(1):139--160, 2019.

\bibitem{delon2020wasserstein}
J.~Delon and A.~Desolneux.
\newblock A wasserstein-type distance in the space of gaussian mixture models.
\newblock {\em SIAM Journal on Imaging Sciences}, 13(2):936--970, 2020.

\bibitem{dvinskikh2021decentralized}
D.~Dvinskikh and A.~Gasnikov.
\newblock Decentralized and parallel primal and dual accelerated methods for
  stochastic convex programming problems.
\newblock {\em Journal of Inverse and Ill-posed Problems}, 2021.

\bibitem{dvinskikh2019primal}
D.~Dvinskikh, E.~Gorbunov, A.~Gasnikov, P.~Dvurechensky, and C.~A. Uribe.
\newblock On primal and dual approaches for distributed stochastic convex
  optimization over networks.
\newblock In {\em 2019 IEEE 58th Conference on Decision and Control (CDC)},
  pages 7435--7440, 2019.
\newblock arXiv:1903.09844.

\bibitem{dvinskikh2020improved}
D.~Dvinskikh and D.~Tiapkin.
\newblock Improved complexity bounds in wasserstein barycenter problem.
\newblock In {\em Proceedings of The 24th International Conference on
  Artificial Intelligence and Statistics}, pages 1738--1746. PMLR, 2021.

\bibitem{dvurechenskii2018decentralize}
P.~Dvurechenskii, D.~Dvinskikh, A.~Gasnikov, C.~Uribe, and A.~Nedich.
\newblock Decentralize and randomize: Faster algorithm for wasserstein
  barycenters.
\newblock {\em Advances in Neural Information Processing Systems},
  31:10760--10770, 2018.

\bibitem{esser2010general}
E.~Esser, X.~Zhang, and T.~F. Chan.
\newblock A general framework for a class of first order primal-dual algorithms
  for convex optimization in imaging science.
\newblock {\em SIAM Journal on Imaging Sciences}, 3(4):1015--1046, 2010.

\bibitem{facchinei2007finite}
F.~Facchinei and J.-S. Pang.
\newblock {\em Finite-dimensional variational inequalities and complementarity
  problems}.
\newblock Springer Science \& Business Media, 2007.

\bibitem{gasnikov2017universal}
A.~Gasnikov.
\newblock Universal gradient descent.
\newblock {\em arXiv preprint arXiv:1711.00394}, 2017.

\bibitem{gidel2018variational}
G.~Gidel, H.~Berard, G.~Vignoud, P.~Vincent, and S.~Lacoste-Julien.
\newblock A variational inequality perspective on generative adversarial
  networks.
\newblock {\em arXiv preprint arXiv:1802.10551}, 2018.

\bibitem{gorbunov2019optimal}
E.~Gorbunov, D.~Dvinskikh, and A.~Gasnikov.
\newblock Optimal decentralized distributed algorithms for stochastic convex
  optimization.
\newblock {\em arXiv:1911.07363}, 2019.

\bibitem{gorbunov2020recent}
E.~Gorbunov, A.~Rogozin, A.~Beznosikov, D.~Dvinskikh, and A.~Gasnikov.
\newblock Recent theoretical advances in decentralized distributed convex
  optimization.
\newblock {\em arXiv preprint arXiv:2011.13259}, 2020.

\bibitem{gramfort2015fast}
A.~Gramfort, G.~Peyr{\'e}, and M.~Cuturi.
\newblock Fast optimal transport averaging of neuroimaging data.
\newblock In {\em International Conference on Information Processing in Medical
  Imaging}, pages 261--272. Springer, 2015.

\bibitem{guminov2021combination}
S.~Guminov, P.~Dvurechensky, N.~Tupitsa, and A.~Gasnikov.
\newblock On a combination of alternating minimization and nesterov’s
  momentum.
\newblock In {\em International Conference on Machine Learning}, pages
  3886--3898. PMLR, 2021.

\bibitem{jakovetic2014fast}
D.~{Jakoveti\'c}, J.~{Xavier}, and J.~M.~F. {Moura}.
\newblock Fast distributed gradient methods.
\newblock {\em IEEE Transactions on Automatic Control}, 59(5):1131--1146, May
  2014.

\bibitem{jambulapati2019direct}
A.~Jambulapati, A.~Sidford, and K.~Tian.
\newblock A direct tilde $\{$O$\}$(1/epsilon) iteration parallel algorithm for
  optimal transport.
\newblock {\em Advances in Neural Information Processing Systems},
  32:11359--11370, 2019.

\bibitem{jin2020efficiently}
Y.~Jin and A.~Sidford.
\newblock Efficiently solving mdps with stochastic mirror descent.
\newblock In {\em International Conference on Machine Learning}, pages
  4890--4900. PMLR, 2020.

\bibitem{Thorsten}
T.~Joachims.
\newblock A support vector method for multivariate performance measures.
\newblock pages 377--384, 01 2005.

\bibitem{kovalev2022optimal}
D.~Kovalev, A.~Beznosikov, A.~Sadiev, M.~Persiianov, P.~Richt{\'a}rik, and
  A.~Gasnikov.
\newblock Optimal algorithms for decentralized stochastic variational
  inequalities.
\newblock {\em arXiv preprint arXiv:2202.02771}, 2022.

\bibitem{kovalev2020optimal}
D.~Kovalev, A.~Salim, and P.~Richt{\'a}rik.
\newblock Optimal and practical algorithms for smooth and strongly convex
  decentralized optimization.
\newblock {\em Advances in Neural Information Processing Systems},
  33:18342--18352, 2020.

\bibitem{kroshnin2019complexity}
A.~Kroshnin, N.~Tupitsa, D.~Dvinskikh, P.~Dvurechensky, A.~Gasnikov, and
  C.~Uribe.
\newblock On the complexity of approximating {W}asserstein barycenters.
\newblock In K.~Chaudhuri and R.~Salakhutdinov, editors, {\em Proceedings of
  the 36th International Conference on Machine Learning}, volume~97 of {\em
  Proceedings of Machine Learning Research}, pages 3530--3540, Long Beach,
  California, USA, 09--15 Jun 2019. PMLR.

\bibitem{lan2020first}
G.~Lan.
\newblock {\em First-order and Stochastic Optimization Methods for Machine
  Learning}.
\newblock Springer, 2020.

\bibitem{lan2017communication}
G.~Lan, S.~Lee, and Y.~Zhou.
\newblock Communication-efficient algorithms for decentralized and stochastic
  optimization.
\newblock {\em arXiv:1701.03961}, 2017.

\bibitem{lan2020communication}
G.~Lan, S.~Lee, and Y.~Zhou.
\newblock Communication-efficient algorithms for decentralized and stochastic
  optimization.
\newblock {\em Mathematical Programming}, 180(1):237--284, 2020.

\bibitem{li2021accelerated}
H.~Li and Z.~Lin.
\newblock Accelerated gradient tracking over time-varying graphs for
  decentralized optimization.
\newblock {\em arXiv preprint arXiv:2104.02596}, 2021.

\bibitem{lin2020fixed}
T.~Lin, N.~Ho, X.~Chen, M.~Cuturi, and M.~I. Jordan.
\newblock Fixed-support wasserstein barycenters: Computational hardness and
  fast algorithm.
\newblock 2020.

\bibitem{lin2020near}
T.~Lin, C.~Jin, and M.~I. Jordan.
\newblock Near-optimal algorithms for minimax optimization.
\newblock In {\em Conference on Learning Theory}, pages 2738--2779. PMLR, 2020.

\bibitem{liu2019decentralized}
W.~Liu, A.~Mokhtari, A.~Ozdaglar, S.~Pattathil, Z.~Shen, and N.~Zheng.
\newblock A decentralized proximal point-type method for saddle point problems.
\newblock {\em arXiv preprint arXiv:1910.14380}, 2019.

\bibitem{mateos2015distributed}
D.~Mateos-N{\'u}nez and J.~Cort{\'e}s.
\newblock Distributed subgradient methods for saddle-point problems.
\newblock In {\em 2015 54th IEEE Conference on Decision and Control (CDC)},
  pages 5462--5467. IEEE, 2015.

\bibitem{mukherjee2020decentralized}
S.~Mukherjee and M.~Chakraborty.
\newblock A decentralized algorithm for large scale min-max problems.
\newblock In {\em 2020 59th IEEE Conference on Decision and Control (CDC)},
  pages 2967--2972. IEEE, 2020.

\bibitem{NECOARA2011756}
I.~Necoara, V.~Nedelcu, and I.~Dumitrache.
\newblock Parallel and distributed optimization methods for estimation and
  control in networks.
\newblock {\em Journal of Process Control}, 21(5):756--766, 2011.
\newblock Special Issue on Hierarchical and Distributed Model Predictive
  Control.

\bibitem{nedic2009distributed}
A.~{Nedic} and A.~{Ozdaglar}.
\newblock Distributed subgradient methods for multi-agent optimization.
\newblock {\em IEEE Transactions on Automatic Control}, 54(1):48--61, 2009.

\bibitem{nemirovski2004prox}
A.~Nemirovski.
\newblock Prox-method with rate of convergence $o(1/t)$ for variational
  inequalities with lipschitz continuous monotone operators and smooth
  convex-concave saddle point problems.
\newblock {\em SIAM Journal on Optimization}, 15(1):229--251, 2004.

\bibitem{nesterov2004introduction}
Y.~Nesterov.
\newblock {\em Introductory Lectures on Convex Optimization: a basic course}.
\newblock Kluwer Academic Publishers, Massachusetts, 2004.

\bibitem{nesterov2005smooth}
Y.~Nesterov.
\newblock Smooth minimization of non-smooth functions.
\newblock {\em Mathematical Programming}, 103(1):127--152, 2005.

\bibitem{Omidshafiei2017:rl}
S.~Omidshafiei, J.~Pazis, C.~Amato, J.~P. How, and J.~Vian.
\newblock Deep decentralized multi-task multi-agent reinforcement learning
  under partial observability.
\newblock In {\em Proceedings of the 34th International Conference on Machine
  Learning (ICML)}, volume~70, pages 2681--2690. PMLR, 2017.

\bibitem{ouyang2019lower}
Y.~Ouyang and Y.~Xu.
\newblock Lower complexity bounds of first-order methods for convex-concave
  bilinear saddle-point problems.
\newblock {\em Mathematical Programming}, pages 1--35, 2019.

\bibitem{rabin2011wasserstein}
J.~Rabin, G.~Peyr{\'e}, J.~Delon, and M.~Bernot.
\newblock {W}asserstein barycenter and its application to texture mixing.
\newblock In {\em International Conference on Scale Space and Variational
  Methods in Computer Vision}, pages 435--446. Springer, 2011.

\bibitem{scaman2017optimal}
K.~Scaman, F.~Bach, S.~Bubeck, Y.~T. Lee, and L.~Massouli{\'e}.
\newblock Optimal algorithms for smooth and strongly convex distributed
  optimization in networks.
\newblock In D.~Precup and Y.~W. Teh, editors, {\em Proceedings of the 34th
  International Conference on Machine Learning}, volume~70 of {\em Proceedings
  of Machine Learning Research}, pages 3027--3036, International Convention
  Centre, Sydney, Australia, 06--11 Aug 2017. PMLR.

\bibitem{sca2018}
K.~Scaman, F.~Bach, S.~Bubeck, L.~Massouli\'{e}, and Y.~T. Lee.
\newblock Optimal algorithms for non-smooth distributed optimization in
  networks.
\newblock In {\em Advances in Neural Information Processing Systems 31}, pages
  2740--2749. 2018.

\bibitem{Solomon2015}
J.~Solomon, F.~De~Goes, G.~Peyr{\'e}, M.~Cuturi, A.~Butscher, A.~Nguyen, T.~Du,
  and L.~Guibas.
\newblock Convolutional {W}asserstein distances: Efficient optimal
  transportation on geometric domains.
\newblock {\em ACM Transactions on Graphics (TOG)}, 34(4):66, 2015.

\bibitem{song2021optimal}
Z.~Song, L.~Shi, S.~Pu, and M.~Yan.
\newblock Optimal gradient tracking for decentralized optimization.
\newblock {\em arXiv preprint arXiv:2110.05282}, 2021.

\bibitem{srivastava2015wasp}
S.~Srivastava, V.~Cevher, Q.~Dinh, and D.~Dunson.
\newblock {WASP: Scalable Bayes via barycenters of subset posteriors}.
\newblock In G.~Lebanon and S.~V.~N. Vishwanathan, editors, {\em Proceedings of
  the Eighteenth International Conference on Artificial Intelligence and
  Statistics}, volume~38 of {\em Proceedings of Machine Learning Research},
  pages 912--920, San Diego, California, USA, 09--12 May 2015. PMLR.

\bibitem{sun2019convergence}
Y.~Sun, A.~Daneshmand, and G.~Scutari.
\newblock Convergence rate of distributed optimization algorithms based on
  gradient tracking.
\newblock {\em arXiv preprint arXiv:1905.02637}, 2019.

\bibitem{uribe2020dual}
C.~A. Uribe, S.~Lee, A.~Gasnikov, and A.~Nedi{\'c}.
\newblock A dual approach for optimal algorithms in distributed optimization
  over networks.
\newblock {\em Optimization Methods and Software}, pages 1--40, 2020.

\bibitem{wai2018multi}
H.-T. Wai, Z.~Yang, Z.~Wang, and M.~Hong.
\newblock Multi-agent reinforcement learning via double averaging primal-dual
  optimization.
\newblock {\em arXiv preprint arXiv:1806.00877}, 2018.

\bibitem{NIPS2004_64036755}
L.~Xu, J.~Neufeld, B.~Larson, and D.~Schuurmans.
\newblock Maximum margin clustering.
\newblock In L.~Saul, Y.~Weiss, and L.~Bottou, editors, {\em Advances in Neural
  Information Processing Systems}, volume~17. MIT Press, 2005.

\bibitem{zhang2019lower}
J.~Zhang, M.~Hong, and S.~Zhang.
\newblock On lower iteration complexity bounds for the saddle point problems.
\newblock {\em arXiv preprint arXiv:1912.07481}, 2019.

\bibitem{zhang2021primal}
J.~Zhang, M.~Wang, M.~Hong, and S.~Zhang.
\newblock Primal-dual first-order methods for affinely constrained multi-block
  saddle point problems.
\newblock {\em arXiv preprint arXiv:2109.14212}, 2021.

\end{thebibliography}
\bibliographystyle{abbrv}

\section{Supplementary for Section \ref{sec:problem_statement}}\label{app:problem_statement}

Define a sequence of Chebyshev polynomials as $T_0(\beta) = 1,~ T_1(\beta) = \beta$ and $T_{k+1}(\beta) = 2\beta T_k(\beta) - T_{k-1}(\beta)$ for $k\geq 1$. After that, let $c_1 = \frac{\sqrt\chi - 1}{\sqrt\chi + 1}$, $c_2 = \frac{\chi + 1}{\chi - 1}$, $c_3 = \frac{2}{\lambda_{\max}(\tilde W) + \lambda_{\min}^+(\tilde W)}$ and introduce
\begin{align*}
    P_K(\beta) = 1 - \frac{T_K(c_2(1 - \beta))}{T_K(c_2)}.
\end{align*}
Now let us consider a new communication matrix $P_K(\tilde\mW)$ with $K = \lfloor \sqrt\chi \rfloor$. As shown in \cite{auzinger2011iterative}, the spectrum of $P_K(\tilde\mW)$ lies in $\sbraces{1 - \frac{2c_1^K}{1 + c_1^{2K}}; 1 + \frac{2c_1^K}{1 + c_1^{2K}}}$ and therefore $\lambda_{\max}(P_K(\tilde\mW)) = O(1)$, $\lambda_{\min}^+(P_K(\tilde\mW)) = O(1)$ and $\chi(P_K(\tilde\mW))= O(1)$.

\section{Smoothness constants for Mirror-Prox}\label{app:mp_smoothness_constants}

\subsection{Estimating Lipschitz constants for $S(\bu,\bv)$}
We start by deriving the Lipschitz constants $L_{\bu\bu}$, $L_{\bv\bv}$, $L_{\bu\bv}$, $L_{\bv\bu}$ of the function $S(\bu,\bv)$ in \eqref{eq:saddle_point}. Recall that for an operator $A$ which acts from a space $V$ with a norm $\|v\|_v$ to a space $W$ with a norm $\|w\|_w$, then these two norms naturally induce a norm of the operator $A$ as $\|A\|_{v\to w} = \max_{v,w}\{\|Av\|_w: \|v\|_v\leq 1\}$, which gives an inequality $\|Av\|_w \leq \|A\|_{v\to w} \|v\|_v$.
Recall also that \eqref{eq:saddle_point} is a reformulation of \eqref{eq:problem_reform} in a simpler form using the definitions $\bu=(\bx^\top,\bp^\top,\bs^\top)^\top$, $\bv=(\by^\top,\bq^\top,\bz^\top)^\top$ and 
\[
S(\bu,\bv) =   F(\bx, \bp, \by, \bq) + \angles{\bs, \mWy\by} + \angles{\bz, \mWx\bx}.
\]
Then for the corresponding partial derivatives we have
\begin{align*}
    \nabla_\bu S(\bu,\bv)= 
    \begin{bmatrix*}[l]
        ~~\nabla_\bx F(\bx, \bp, \by, \bq) + \mWx\bz \\
        ~~\nabla_\bp F(\bx, \bp, \by, \bq) \\
        ~~\mWy\by 
    \end{bmatrix*}, \;\;\;
    \nabla_\bv S(\bu,\bv)= 
    \begin{bmatrix*}[l]
        \nabla_\by F(\bx, \bp, \by, \bq) + \mWy\bs \\
        \nabla_\bq F(\bx, \bp, \by, \bq) \\
        \mWx\bx
    \end{bmatrix*}.
\end{align*}
Let $\sA_\sU = \sX\times\sP\times\R^{nd_q}$, $\sA_\sV = \sY\times\sQ\times \R^{nd_p}$ and $\sU = \sX\times\sP\times\{\bs:~ \norm{\bs}_2\leq R_\sS\}$, $\sV = \sY\times\sQ\times\{\bz:~ \norm{\bz}_2\leq R_\sZ\}$. \revoms{Recall that we introduced} norms $\norm{\bu}_{\sU}^2 = \norm{\bx}_{\sX}^2/R_\sX^2 + \norm{\bp}_{\sP}^2/R_\sP^2 + \norm{\bs}_{2}^2/R_\sS^2$, $\norm{\bv}_{\sV}^2 = \norm{\by}_{\sY}^2/R_\sY^2 + \norm{\bq}_{\sQ}^2/R_\sQ^2 + \norm{\bz}_{2}^2/R_\sZ^2$, which induce the dual norms $\norm{\bu}_{\sU,*}^2 = R_\sX^2\norm{\bx}_{\sX,*}^2 + R_\sP^2\norm{\bp}_{\sP,*}^2 + R_\sS^2\norm{\bs}_{2}^2$, $\norm{\bv}_{\sV,*}^2 = R_\sY^2\norm{\by}_{\sY,*}^2 + R_\sQ^2\norm{\bq}_{\sQ,*}^2 + R_\sZ^2\norm{\bz}_{2}^2$. The constant $L_{\bu\bu}$ has to satisfy $\|\nabla_\bu S(\bu,\bv)-\nabla_\bu S(\bu',\bv)\|_{\sU,*}\leq L_{\bu\bu}\|\bu-\bu'\|_{\sU}$ for all $\bu,\bu' \in \sA_\sU$. We have
\begin{align*}
    &\|\nabla_\bu S(\bu,\bv)-\nabla_\bu S(\bu',\bv)\|_{\sU,*}^2 =
    \left\|\begin{bmatrix*}[c]
        ~~\nabla_\bx F(\bx, \bp, \by, \bq) - \nabla_\bx F(\bx', \bp', \by, \bq) \\
        ~~\nabla_\bp F(\bx, \bp, \by, \bq) - \nabla_\bp F(\bx', \bp', \by, \bq) \\
        ~~0
    \end{bmatrix*}\right\|_{\sU,*}^2 \\
    &\quad = R_\sX^2\norm{\nabla_\bx F(\bx, \bp, \by, \bq) - \nabla_\bx F(\bx', \bp', \by, \bq)}_{\sX,*}^2 \\
    &\qquad + R_\sP^2\norm{\nabla_\bp F(\bx, \bp, \by, \bq) - \nabla_\bp F(\bx', \bp', \by, \bq)}_{\sP,*}^2 \\
\end{align*}
\begin{align*}
    &\leq (R_\sX^2 + R_\sP^2)
    \left\|\begin{bmatrix*}[l]
        ~~\nabla_\bx F(\bx, \bp, \by, \bq) - \nabla_\bx F(\bx', \bp', \by, \bq) \\
        ~~\nabla_\bp F(\bx, \bp, \by, \bq) - \nabla_\bp F(\bx', \bp', \by, \bq)
    \end{bmatrix*}\right\|_{(\sX,\sP),*}^2 \\
    &\leq L_{(\bx,\bp)(\bx,\bp)}^2(R_\sX^2 + R_\sP^2)\left\|\begin{bmatrix*}[l]
        ~~\bx-\bx' \\
        ~~ \bp-\bp'
    \end{bmatrix*}\right\|_{(\sX,\sP)}^2 \\
    &= L_{(\bx,\bp)(\bx,\bp)}^2 (R_\sX^2 + R_\sP^2) \cbraces{R_\sX^2\frac{\norm{\bx - \bx'}_\sX^2}{R_\sX^2} + R_\sP^2\frac{\norm{\bp - \bp'}_\sP^2}{R_\sP^2}} \\
    &\leq L_{(\bx,\bp)(\bx,\bp)}^2 (R_\sX^2 + R_\sP^2)^2\|\bu-\bu'\|_{\sU}^2, 
\end{align*}
where we used that $\norm{(\bx,\bp)}_{(\sX,\sP),*}^2=\norm{\bx}_{\sX,*}^2+\norm{\bp}_{\sP,*}^2$ since $\norm{(\bx,\bp)}_{(\sX,\sP)}^2=\norm{\bx}_{\sX}^2+\norm{\bp}_{\sP}^2$, Assumption \ref{assum:F_Lip_smooth}. Thus, $L_{\bu\bu}=L_{(\bx,\bp)(\bx,\bp)}(R_\sX^2 + R_\sP^2)$. The equality $L_{\bv\bv}=L_{(\by,\bq)(\by,\bq)}(R_\sY^2 + R_\sQ^2)$ is proved in the same way.

Let us estimate $L_{\bu\bv}$, which has to satisfy $\|\nabla_\bu S(\bu,\bv)-\nabla_\bu S(\bu,\bv')\|_{\sU,*}\leq L_{\bu\bv}\|\bv-\bv'\|_{\sV}$ for all $\bv,\bv' \in \sA_\sV$.
We have
\begin{align*}
    &\|\nabla_\bu S(\bu,\bv)-\nabla_\bu S(\bu,\bv')\|_{\sU,*}^2 \\
    &\quad=
    \left\|\begin{bmatrix*}[l]
        ~~\nabla_\bx F(\bx, \bp, \by, \bq) + \mWx\bz - \nabla_\bx F(\bx, \bp, \by', \bq') - \mWx\bz' \\
        ~~\nabla_\bp F(\bx, \bp, \by, \bq) - \nabla_\bp F(\bx, \bp, \by', \bq') \\
        ~~\mWy\by- \mWy\by'
    \end{bmatrix*}\right\|_{\sU,*}^2\\
    &\overset{\circledOne}{=} R_\sX^2\|\nabla_\bx F(\bx, \bp, \by, \bq) - \nabla_\bx F(\bx, \bp, \by', \bq') + \mWx(\bz - \bz')\|_{\sX,*}^2 \\
    &\qquad+ R_\sP^2\|\nabla_\bp F(\bx, \bp, \by, \bq)-\nabla_\bp F(\bx, \bp, \by', \bq') \|_{\sP,*}^2 + R_\sS^2\|\mWy(\by-\by') \|_2^2 \\
    &\overset{\circledTwo}{\leq} 2R_\sX^2\|\nabla_\bx F(\bx, \bp, \by, \bq) - \nabla_\bx F(\bx, \bp, \by', \bq')\|_{\sX,*}^2 + 2R_\sX^2\|\mWx(\bz - \bz') \|_{\sX,*}^2 \\
    &\qquad+ R_\sP^2\|\nabla_\bp F(\bx, \bp, \by, \bq) - \nabla_\bp F(\bx, \bp, \by', \bq')\|_{\sP,*}^2 + R_\sS^2\|\mWy(\by-\by') \|_2^2\\
    &\overset{\circledThree}{\leq} 2(R_\sX^2 + R_\sP^2) \left\|\begin{bmatrix*}[l]
        ~~\nabla_\bx F(\bx, \bp, \by, \bq) - \nabla_\bx F(\bx, \bp, \by', \bq') \\
        ~~\nabla_\bp F(\bx, \bp, \by, \bq) - \nabla_\bp F(\bx, \bp, \by', \bq')
    \end{bmatrix*}\right\|_{(\X,\Pp),*}^2 \\
    &\qquad+ 2R_\sX^2 \|\mWx\|_{2\to(\bx,*)}^{2}\|\bz - \bz'\|_2^2 + R_\sS^2\|\mWy\|_{\by\to2}^{{2}} \|\by-\by'\|_{\sY}^2 \\
    &\overset{\circledFour}{\leq} 2 L_{(\bx,\bp)(\by,\bq)}^2 (R_\sX^2 + R_\sP^2) \left\|\begin{bmatrix*}[l]
        ~~\by-\by' \\
        ~~ \bq-\bq'
    \end{bmatrix*}\right\|_{(\sY,\sQ)}^2 + 2R_\sX^2 \|\mWx\|_{2\to(\bx,*)}^{{2}} \|\bz - \bz'\|_{2}^2 \\
    &\qquad + R_\sS^2 \|\mWy\|_{\by\to2}^{{2}} \|\by-\by'\|_{\sY}^2 \\
    &\leq \Big(2L_{(\bx,\bp)(\by,\bq)}^2 (R_\sX^2 + R_\sP^2)(R_\sY^2 + R_\sQ^2) + 2R_\sX^2 R_\sZ^2\norm{\mWx}_{2\to(\bx,*)}^2 \\
    &\qquad+ 2R_\sY^2 R_\sS^2\norm{\mWy}^2_{\by\to2}\Big) \norm{\bv - \bv'}_\sV^2,
\end{align*}
where we used in $\circledOne$ the definition of $\|\bu\|_{\sU,*}^2$; in $\circledTwo$ the inequality $(a+b)^2\leq 2a^2+2b^2$; in $\circledThree$ that $\norm{(\bx,\bp)}_{(\bx,\bp),*}^2=\norm{\bx}_{\sX,*}^2+\norm{\bp}_{\sP,*}^2$ since $\norm{(\bx,\bp)}_{(\sX,\sP)}^2=\norm{\bx}_{\sX}^2+\norm{\bp}_{\sP}^2$ and the definition of the operator norm; in $\circledFour$ Assumption \ref{assum:F_Lip_smooth}; and, finally, in $\circledFive$ the definition of $\|\bv \|_{\sV}^2$. Taking the square root of the derived inequality, we obtain 
\begin{align*}
L_{\bu\bv}\leq\sqrt{2}\max\Big\{&L_{(\bx,\bp)(\by,\bq)}\sqrt{(R_\sX^2 + R_\sP^2)(R_\sY^2 + R_\sQ^2)}, \\
&R_\sX R_\sZ \norm{\mWx}_{2\to(\bx,*)},  R_\sY R_\sS \norm{\mWy}_{\by\to2}\Big\}.
\end{align*}
The bound for $L_{\bv\bu}$ is derived in the same way:
\begin{align*}
L_{\bv\bu}\leq\sqrt{2}\max\Big\{&L_{(\by,\bq)(\bx,\bp)}\sqrt{(R_\sX^2 + R_\sP^2)(R_\sY^2 + R_\sQ^2)}, \\
&R_\sX R_\sZ \norm{\mWx}_{x\to2},  R_\sY R_\sS \norm{\mWy}_{2\to(\by,*)}\Big\}.
\end{align*}

\section{Proof of lower bounds from Theorem \ref{Th:lower_bound}}\label{app:lower_bounds}

Let $B \subseteq \mathcal{V}$ be a subset of nodes of $\mathcal{G}$. For $\rho>0$ we define $B_{\rho} = \{v \in \mathcal{V} ~:~ \rho(B,v) \geq \rho\}$, where $\rho(B,v)$ is a distance between set $B$ and node $v$. Then, we construct the following  arrangement of bilinearly functions on nodes:
\begin{eqnarray}
\label{t2}
f_i (x,y) = 
\begin{cases}
\frac{m}{|B_{\rho}|}\cdot\frac{L}{4} x^T A_1 y + \frac{16\varepsilon}{R^2}\|x\|^2_2 - \frac{16\varepsilon}{R^2}\|y\|^2_2 - \frac{m}{|B_{\rho}|}\cdot \frac{L^2 R^2}{\varepsilon}e_1^T y, & i \in B_{\rho}, \\
\frac{m}{|B|} \cdot\frac{L}{4} x^T A_2 y + \frac{16\varepsilon}{R^2}\|x\|^2_2 - \frac{16\varepsilon}{R^2}\|y\|^2_2, & i \in B, \\
\frac{16\varepsilon}{R^2}\|x\|^2_2 - \frac{16\varepsilon}{R^2}\|y\|^2_2, & \text{otherwise},
\end{cases}
\end{eqnarray}
where $L, \varepsilon > 0$, $e_1 = (1,0 \ldots, 0)$ and
{\small
\begin{eqnarray*}
\label{t4}
A_1 = \left(
\begin{array}{cccccccc}
1&0 & & & & & &  \\
&1 &-2 & & & & &  \\
& &1 &0 & & & & \\
& & &1 &-2 & & & \\
& & & &\ldots &\ldots & & \\
& & & & &1  &-2   & \\
& & &   & & &1 &0 \\
& & &  & & & &1 \\
\end{array}
\right), ~~~
A_2 = \left(
\begin{array}{cccccccc}
1&-2 & & & & & &  \\
&1 &0 & & & & &  \\
& &1 &-2 & & & & \\
& & &1 &0 & & & \\
& & & &\ldots &\ldots & & \\
& & & & &1  &0   & \\
& & &   & & &1 &-2 \\
& & &  & & & &1 \\
\end{array}
\right).
\end{eqnarray*}
}
We will give the value of $\rho$ later.

\begin{lemma}
If $B_{\rho} \neq \varnothing$, in the global output of any procedure that satisfies Assumption \ref{ass_bbp}, after $T$ units of time, only the first $k =  \left\lfloor \frac{T-2t}{t + \rho\tau} \right\rfloor + 2$ coordinates can be non-zero, the rest of the $d-k$ coordinates are strictly equal to zero.
\end{lemma}

\begin{proof} Consider an arbitrary moment $T_0$. Following Assumption \ref{ass_bbp} one can write down how $H^x, H^y$ changes in one local step:
\begin{equation*}
   H^x_{i,T_0 + t} \subseteq 
   \begin{cases}
   \text{span}\left\{x, A_1y \right\},& i \in B_{\rho}\\
   \text{span}\left\{x, A_2y \right\},& i \in B \\
   \text{span}\left\{x \right\},& \text{otherwise} \\
   \end{cases}, 
\end{equation*}
\begin{equation*}
   H^y_{i,T_0 + t} \subseteq 
   \begin{cases}
   \text{span}\left\{e_1, y, A_1^T x \right\},& i \in B_{\rho}\\
   \text{span}\left\{y, A_2^T x \right\},& i \in B \\
   \text{span}\left\{y \right\},& \text{otherwise} \\
   \end{cases},
\end{equation*}
for $x \in H^x_{i,T_0}, \forall y \in H^y_{i,T}$. Here we take into account that the projection operator on the ball centered at $0$ has no effect on the expressions written above in terms of span.

The block-diagonal structure of matrices $A_1$ and $A_2$ plays an important role. In particular, let for some $i$ we have $H^x_{i,T_0}\subseteq \text{span}\left\{e_1,\ldots, e_{k_x} \right\}, H^y_{i,T_0} \subseteq \text{span}\left\{e_1,\ldots, e_{k_y} \right\}$, then for any number $H \in \mathcal{N}$ of local iterations (without communications) and $k = \max\{ k_x, k_y\}$, we get
\begin{equation}
\label{t56}
   H^x_{i,T_0 + tH}, H^y_{i,T_0 +tH} \subseteq 
   \begin{cases}
   \text{span}\left\{e_1,\ldots, e_{2\left\lceil\frac{k+1}{2}\right\rceil - 1} \right\} 
   ,& i \in B_{\rho}, \\
   \text{span}\left\{e_1,\ldots, e_{2\left\lceil\frac{k}{2}\right\rceil} \right\},
   & i \in B, \\
   \text{span}\left\{e_1,\ldots, e_{k} \right\},& \text{otherwise}. \\
   \end{cases}
\end{equation}
This fact leads to the main idea of the proof. At the initial moment of time $T=0$, we have all zero coordinates in the global output, since the starting points $x_0, y_0$ are equal to $0$. Using only local iterations (at least 2), we can achieve that for the nodes $B_{\rho}$ only the first coordinates of $x$ and $y$ can be non-zero, the rest coordinates are strictly zero. For the rest of the nodes, all coordinates remains strictly zero. Without communications, the situation does not change. Therefore, we need to make at least $\rho$ communications in order to have non-zero first coordinates in some node from $B$ (transfer of information from $B_{\rho}$ to $B$). Using \eqref{t56}, by local iterations (at least 1) at the node of the set $B$, one can achieve the first and second non-zero coordinates. Next the process continues with respect of \eqref{t56}. 

Hence, to get at least one node $i \in \mathcal{V}$ with the internal memory $H^x_{i}$,$H^y_{i} \subseteq \text{span}\left\{e_1,\ldots, e_{k_0} \right\}$, we need a minimum of $k_0 + 1$ local steps (at least 2 steps in the beginning, when we start from $0$, and then at least 1 local step in other cases -- see previous paragraph), as well as $(k_0 - 1) \rho$ communication rounds. In other words,
$$\max_{i}\left[\max \{k \in \mathcal{N} ~:~ \exists x\in H^x_{i,T}, ~\exists y\in H^y_{i,T}~ \text{s.t.}~ x_k \neq 0,  y_k \neq 0\} \right] \leq \left\lfloor \frac{T-2t}{t + \rho\tau} \right\rfloor + 2.$$ 
According to Assumption \ref{ass_bbp}, we have that the final global output is the union of all the outputs. Whence the statement of Lemma holds. 
\end{proof}

The previous lemma gives us an understanding of how quickly we can approximate the solution. In particular, in coordinates that can be non-zero we are able to have a value that absolutely coincides with the solution, but in zero coordinates this is impossible. It remains to understand what the solution even looks like. Considering the global objective function, we have:
\begin{align}
    \label{t144}
    f(x,y) &= \frac{1}{M} \sum\limits_{m=1}^M f_m(x,y) \nonumber \\
    &= \frac{1}{M} \left( |B_{\rho}| \cdot f_1(x,y) + |B| \cdot f_2(x,y) + (M - |B_{\rho}| - |B|) \cdot f_3(x,y)\right) \nonumber\\
    &= \frac{L}{2} x^T A y + \frac{16\varepsilon}{R^2}\|x\|^2_2 - \frac{16\varepsilon}{R^2}\|y\|^2_2 -  \frac{L^2 R^2}{\varepsilon} e_1^T y, ~~~\text{with}~~~A = \frac{1}{2}(A_1 + A_2).
\end{align}
With $\| A\| \leq 2$, it is easy to check that $f_i$ is convex-concave and $f$ is also $\max\{L, \tfrac{32\varepsilon}{R^2}\}$-smooth.


\begin{lemma}[Lemma 3.3 from \cite{zhang2019lower}]\label{lem2}
Let $\alpha = \frac{32\varepsilon^2}{L^2 R^4}$ and $q = \frac{1}{2}\left(2 + \alpha - \sqrt{\alpha^2 + 4\alpha} \right) \in (0;1)$ -- the smallest root of $q^2 - (2 + \alpha) q + 1 = 0$, and let introduce the approximation of the solution $\bar y^*$:
\begin{equation*}
    \bar y^*_i = \frac{q^i}{1-q}.
\end{equation*}
Then error between the approximation and the real solution of \eqref{t144} can be bounded:
\begin{equation*}
    \|\bar y^* - y^*\|_2 \leq \frac{q^{d+1}}{\alpha(1-q)}.
\end{equation*}
\end{lemma}
\begin{proof} Let us write the dual function for \eqref{t144}:
\begin{equation*}
    g(y) = -\frac{1}{2}y^T \left(\frac{L^2 R^2}{\varepsilon}A^T A + \frac{32\varepsilon}{ R^2} I \right)y + \frac{L^2 R^2}{\varepsilon} e_1^T y,
\end{equation*}
where one can easy found 
\begin{eqnarray*}
A A^T = \left(
\begin{array}{cccccccc}
1&-1 & & & & & &  \\
-1&2 &-1 & & & & &  \\
&-1 &2 & -1 & & & & \\
& & -1&2 &-1 & & & \\
& & &-1 &2 &-1 & & \\
& & & & &\ldots & & \\
& & & & &-1 &2 &-1 \\
& & & & & &-1 &2 \\
\end{array}
\right).
\end{eqnarray*}
The optimality of the dual problem ($\nabla g(y^*) = 0$) gives
\begin{equation*}
    \left(\frac{L^2 R^2}{\varepsilon}A^T A + \frac{32\varepsilon}{R^2} I \right)y^* = \frac{L^2 R^2}{\varepsilon} e_1,
\end{equation*}
or, with $\alpha = \tfrac{32\varepsilon^2}{L^2 R^4}$, we get \begin{equation*}
    \left(A^T A + \alpha I \right)y^* = e_1.
\end{equation*}
One can note that for the approximation of the solution $\bar y$ it holds
\begin{eqnarray*}
\left\{
\begin{array}{l}
(1+\alpha)\bar y_1^* - \bar y_2^* = 1 \\
-\bar y_1^* + (2 + \alpha) \bar y^*_2 - \bar y^*_3 = 0\\
\ldots \\
-\bar y_{d-2}^* + (2 + \alpha) \bar y^*_{d-1} - \bar y^*_d = 0\\
-\bar y^*_{d-1} + (2+\alpha) \bar y^*_d = \frac{q^{d+1}}{1-q}
\end{array}
\right .
\end{eqnarray*}
or in the form of a system of equations
\begin{equation*}
    \left(A^T A + \alpha I \right)\bar y^* = e_1 + \frac{q^{d+1}}{1-q}e_d.
\end{equation*}
Hence, we have
\begin{equation*}
    \bar y^* - y^* = \left(A^T A + \alpha I \right)^{-1}\frac{q^{d+1}}{1-q}e_d,
\end{equation*}
With the fact $\alpha^{-1} I \succeq \left(A^T A + \alpha I \right)^{-1} \succ 0$, we prove the statement of Lemma.
\end{proof}

Now we formulate a key lemma (similar to Lemma 3.4 from \cite{zhang2019lower}).

\begin{lemma} \label{lem56}
If $B_{\rho} \neq \varnothing$, let consider a distributed saddle point problem in the form \eqref{t2}. For any time $T$ one can found size of problem $d \geq \max \left\{ 2 \log_q \left( \frac{\alpha}{4\sqrt{2}}\right), 2k\right\}$, where $k = \left\lfloor \frac{T-2t}{t + \rho\tau} \right\rfloor + 2$, take $\alpha = \frac{32\varepsilon^2}{L^2 R^4}$ and $q = \frac{1}{2}\left(2 + \alpha - \sqrt{\alpha^2 + 4\alpha} \right) \in (0;1)$. Then, for any procedure satisfying Assumption \ref{ass_bbp} we get the following lower bound for the solution after $T$ units of time:
\begin{equation*}
    \|x_T - x^*\|^2_2 + \|y_T - y^*\|^2_2 \geq q^{2\left(\frac{T}{t + \rho\tau} + 2\right)} \frac{\| y_0 - y^*\|^2_2}{16}.
\end{equation*}
\end{lemma}
\begin{proof} By Lemma \ref{lem2} with $q < 1$ and $k \leq \frac{d}{2}$, we have
\begin{eqnarray*}
    \|y_T - \bar y^*\|^2_2 &\geq& \sqrt{\sum\limits_{j=k+1}^d  (\bar y^*_j)^2} = \frac{q^k}{1-q} \sqrt{q^2 + q^4 + \ldots + q^{2(d-k)}} \\
    &\geq& \frac{q^k}{\sqrt{2}(1-q)} \sqrt{q^2 + q^4 + \ldots + q^{2d}} = \frac{q^k}{\sqrt{2}} \| \bar y^*\|^2_2 = \frac{q^k}{\sqrt{2}} \| y_0 - \bar y^*\|^2_2.
\end{eqnarray*}
For $d \geq 2 \log_q \left(\frac{\alpha}{4\sqrt{2}} \right)$ we can guarantee that $\bar y^* \approx y^*$ and also (for more details see \cite{zhang2019lower})
\begin{align*}
    \|x_T - x^*\|^2_2 + \|y_T - y^*\|^2_2 &\geq \|y_T - y^*\|^2_2 \geq \frac{q^{2k}}{16} \| y_0 - y^*\|^2_2 = q^{2\left(\left\lfloor \frac{T-2t}{t + \rho\tau} \right\rfloor + 2\right)} \frac{\| y_0 - y^*\|^2_2}{16} \\
    &\geq q^{2\left( \frac{T}{t + \rho\tau} + 2\right)} \frac{\| y_0 - y^*\|^2_2}{16}.
\end{align*}
\end{proof}

In the conditions of Lemma \ref{lem56} there is a choice of $d$.  With given $y^*$ we can also determine the size of the ball $R$. More precisely we want $y^* \in \bar \sY$, then we can choose $R = \tfrac{1}{\sqrt{2}}\| y^* - y^0\| = \tfrac{1}{\sqrt{2}}\| y^*\|$ (see \eqref{eq:R_def_for_prox}).
Now we are ready to combine the facts obtained above and prove Theorem \ref{Th:lower_bound}.

\begin{theorem}[Theorem \ref{Th:lower_bound}]
Let $L > 0$, $\chi \geq 1$, $\varepsilon > 0$.  Additionally, we assume that $L \geq  \tfrac{32\varepsilon}{R^2}$. There exists a distributed saddle point problem of $m$ (defined in the proof) functions with decentralized architecture and a gossip matrix $W$, for which the following statements are valid:
\begin{itemize}[leftmargin=*]
    \item each function $f_i : \R^d \times \R^d \to \R$ is convex-concave,
    \item $f = \tfrac{1}{m} \sum_{i=1}^m f_i : \R^d \times \R^d \to \R$ is $L$-smooth and has $y^* \neq 0$,
    \item $d \geq \max \left\{ 2 \log_q \left( \tfrac{\alpha}{4\sqrt{2}}\right), 2k\right\}$, where $k = \left\lfloor \tfrac{T-2t}{t + \tau} \right\rfloor + 2$, $\alpha = \tfrac{32\varepsilon^2}{L^2 R^4}$ and $q = \frac{1}{2}\left(2 + \alpha - \sqrt{\alpha^2 + 4\alpha} \right) \in (0;1)$,
    \item $\bar\sX \times \bar\sY$ is bounded and has a size $R = \tfrac{1}{\sqrt{2}}\| y^*\|$,
    \item the gossip matrix $W$ have $\chi(W) = \chi$.
\end{itemize}
Then for any procedure, which satisfies Assumption \ref{ass_bbp}, one can bounded the time to achieve a $\varepsilon$-solution (i.e. $\max_{y \in \bar \sY}f(x_T, y) - \min_{x \in \bar \sX} f(x, y_T) \leq \varepsilon$) in the final global output:
\begin{equation*}
    T = \Omega\left( \frac{LR^2}{\varepsilon} (t + \sqrt{\chi}\tau)\right).
\end{equation*}
\end{theorem}

\begin{proof}
We start the proof from Lemma \ref{lem56} and get:
\begin{equation*}
    \|x_T - x^*\|^2 + \|y_T - y^*\|^2 \geq q^{2\left( \frac{T}{t + \rho\tau} + 2\right)} \frac{\| y_0 - y^*\|^2}{16}.
\end{equation*}
Using $\tfrac{32\varepsilon}{R^2}$ - strong convexity -- strong concavity of \eqref{t144}, we can obtain
\begin{align*}
    \max_{y \in \bar \sY} f(x^T, y) - \min_{x \in \bar \sX} f(x, y^T) &\geq 
    f(x_T, y^*) - f(x^*, y^*) + f(x^*, y^*) - f(x^*, y_T) \\
    &\geq \frac{16\varepsilon}{R^2}\|x_T - x^*\|^2 + \frac{16\varepsilon}{R^2}\|y_T - y^*\|^2,
\end{align*}
and 
\begin{equation*}
    \max_{y \in \bar \sY} f(x^T, y) - \min_{x \in \bar \sX} f(x, y^T) \geq q^{2\left( \frac{T}{t + \rho\tau} + 2\right)} \cdot \frac{\varepsilon}{R^2} \| y_0 - y^*\|^2\geq q^{2\left( \frac{T}{t + \rho\tau} + 2\right)} \cdot 2\varepsilon.
\end{equation*}
If $\max_{y \in \bar \sY}f(x^T, y) - \min_{x \in \bar \sX} f(x, y^T) \leq \varepsilon$, we get
\begin{equation}
    \label{eq:temp808}
    \frac{2T}{t + \rho\tau} = \Omega\left( \log_{q^{-1}}(2)\right).
\end{equation}
The next steps of the proof follows similar way with the  proof of  \cite[Theorem 2]{scaman2017optimal}. 
Let $\chi_m = \tfrac{1 + \cos \tfrac{\pi}{m}}{1 - \cos \tfrac{\pi}{m}}$ be an  increasing sequence of positive numbers (this sequence is also the condition numbers for the Laplacian of a linear graph of $m$ vertexes). Since $\gamma_2 = 1$ and $\lim_m \chi_m = +\infty$, there exists $m \geq 2$ such that $\chi_m \leq \chi < \chi_{m+1}$.

$\bullet$ If $m \geq 3$, let us consider a linear graph with $m$ vertexes $v_1, \ldots v_m$ and with weighted edges $w_{1,2} = 1 - a$ and  $w_{i,i+1} = 1$ for $i \geq 2$. If $W_a$ is the Laplacian of this weighted graph, one can note that with $a = 0$, $\chi(W_a) = \gamma_m$, with $a = 1$, we have $\chi(W_a) = +\infty$. Hence, there exists $a \in [0;1)$ such that $\chi(W_a) = \chi$. Then $\chi < \chi_{m+1} \leq \tfrac{(m+1)^2}{2}$. Finally, if $B = \{v_1\}$, $\bar B = \{v_m\}$, $\rho = m - 1$ and then, $\rho \geq \sqrt{2\chi} - 2$. Using \eqref{eq:temp808}, we have
\begin{equation*}
    T = \Omega\left( (t + \sqrt{\chi}\tau)\frac{1}{\ln q^{-1}}\right).
\end{equation*}
$\bullet$ If $M = 2$, we construct a fully connected network with 3 nodes and a weight $w_{1,3} = a \in [0;1]$. Let $W_a$ is the Laplacian. If $a = 0$, then the network is a linear graph and $\chi(W_a) = \chi_3 =3$. Hence, there exists $a \in [0;1)$ such that $\chi(W_a) = \chi$. We take $B = \{v_1\}$, $\bar B = \{v_3\}$ and get $\rho \geq 1 \geq \tfrac{\chi}{3}$. Finally, we have  the same estimate as in the previous point
\begin{equation*}
    T = \Omega\left( (t + \sqrt{\chi}\tau)\frac{1}{\ln q^{-1}}\right).
\end{equation*}
Next, we work with
\begin{eqnarray*}
    \frac{1}{\ln (q^{-1})} &=& \frac{1}{\ln (1+ (1-q)/q))} \geq \frac{q}{1-q} \\
    &=&
    \frac{1 + \frac{\alpha}{2} - \sqrt{\frac{\alpha^2}{4} + \alpha}}{\sqrt{\frac{\alpha^2}{4} + \alpha} - \frac{\alpha}{2}} = \frac{\sqrt{\frac{\alpha^2}{4} + \alpha} - \frac{\alpha}{2}}{\alpha} \\
    &=& \sqrt{\frac{1}{4} + \frac{1}{\alpha}} - \frac{1}{2} 
    = \sqrt{\frac{1}{4} + \frac{1}{32}\left(\frac{L R^2}{\varepsilon}\right)^2} - \frac{1}{2}.
\end{eqnarray*}
Hence, we get
\begin{equation*}
    T = \Omega\left( \frac{L R^2}{\varepsilon} (t + \sqrt{\chi}\tau)\right).
\end{equation*}
\end{proof}

\section{Missing proofs from Section \ref{sec:wb}}\label{app:wb}




\begin{proof}[Proof of Lemma \eqref{lemma:constrained_lagrange_multipiers_wb}]
\add{From Lemma \ref{lemma:constrained_lagrange_multipliers} it holds
\begin{equation}\label{eq:newR}
    R_\sZ^2 = \frac{{2m} M_x^2}{(\lambda_{\min}^+(W))^2}
\end{equation}
where positive scalar $M_x$ is such that: 
$\norm{\nabla_x f_i(x, p_i, q_i)}_2\leq M_x$ and $f_i(x, p_i, q_i)$ is defined in \eqref{eq:def_saddle_prob}. Next we calculate its gradient
\[
    \nabla_x  f_i(x, p_i, q_i) = - 2\|d\|_\infty  [q_{i}]_{1...n},
\]
where $[q_{i}]_{1...n}$ is the first $n$ component of vector $q_i\in[-1,1]^{2n}$.
As $q_i\in [-1,1]^{2n}$ we have
\begin{equation}\label{eq:fffdlldddj}
    \|\nabla  f_i(x, p_i, q_i) \|_2 = 2\|d\|_\infty \| [q_{i}]_{1...n}\|_2 \leq  2\|d\|_\infty  \sqrt n = M_x.
\end{equation}
Thus, using this and the definition of $d$ which provides $\|d\|_\infty = \max_{i,j}C_{ij}$ for \eqref{eq:newR} we get 
\[R_\sZ^2 = \frac{{8  mn \max_{i,j}C_{ij}^2} }{(\lambda_{\min}^+(W))^2}.\]
}
\end{proof}


\begin{proof}[Proof of Lemma \eqref{lm:Lipsch1}]
As $S(\u,\v)$ from \eqref{eq:saddle_point} is bilinear, $L_{\u\u}=L_{\v\v}=0$. Then from Lemma 
\ref{lemma:smothness_constants} it follows
    that $S(\bu,\bv)$  is $(0, L_{\u\v}, L_{\v\u}, 0)$-smooth, where
    \begin{align}\label{eq:constLbu}
        L_{\bu\bv} &= \sqrt{2}\cbraces{L_{(\bx,\bp)\bq} \sqrt{(R_\sX^2 + R_\sP^2)R_\sQ^2} + R_\sX R_\sZ\norm{\W}_{2\to(\bx,*)}}, \notag \\
         L_{\bv\bu} &= \sqrt{2}\cbraces{L_{\bq(\bx,\bp)} \sqrt{(R_\sX^2 + R_\sP^2) R_\sQ^2} + R_\sX R_\sZ\norm{\W}_{\bx\to2}}.
        \end{align}
  As $x_i \in \Delta_n$ and $p_i \in \Delta_{n^2}$, we use the simplex setup: $\ell_1$-norm and  the following prox-functions on $\Delta_n$ and $\Delta_{n^2}$
  \begin{equation}\label{eq:prox_xp}
      d_{\Delta_n}(x_i)  =  \la x_i,\log x_i \ra  \quad \text{ and } \quad d_{\Delta_{n^2}}(p_i) =  \la p_i,\log p_i \ra.
  \end{equation}
  For $\X \triangleq \prod^m \Delta_{n}$ and $\Pp  \triangleq \prod^m \Delta_{n^2}$, we choose the following $\ell_1/\ell_2$-norms
  \begin{equation}\label{eq:normsXPp}
        \|\x\|^2_{\X} = \sum_{i=1}^m\|x_i\|^2_{1} \quad \text{ and } \quad  \|\p\|^2_{\Pp} = \sum_{i=1}^m\|p_i\|^2_{1}.
  \end{equation}
Then from \eqref{eq:R_def_for_prox} we have
\begin{align*}
    &R^2_{\Delta_n} = \max\limits_{x_i \in  {\Delta_{n}} }d_{{\Delta_{n}}}(x_i) - \min\limits_{x_i \in  {\Delta_{n}} }d_{\Delta_{n}}(x_i) = \ln n, \\
     &R^2_{\Delta_{n^2}} = \max\limits_{p_i \in {\Delta_{n^2}}}d_{\Delta_{n^2}}(p_i) - \min\limits_{p_i \in {\Delta_{n^2}} }d_{\Delta_{n^2}}(p_i) = 2  \ln n.
\end{align*}
Then from this and definitions of  $\X \triangleq \prod^m \Delta_{n}$ and $\Pp  \triangleq \prod^m \Delta_{n^2}$, it follows
\begin{equation}\label{eq:rxrx}
R^2_{\X}  = \sum_{i=1}^m R^2_{\Delta_n} = m\ln n, \quad \text{ and } \quad R^2_{\Pp}  = \sum_{i=1}^m R^2_{\Delta_{n^2}} = 2m\ln n.
\end{equation}
  Using this and \eqref{eq:prox_xp}, we get \begin{align}\label{eq:prox-func_me}
       d_{\bu}(\bu) =
\sum_{i=1}^m\left(\frac{d_{\Delta_n}(x_i)}{R_\sX^2} + \frac{d_{\Delta_{n^2}}(p_i)}{R_\sP^2} \right) = \frac{1}{m \ln n} \sum_{i=1}^m\left(  \la x_i, \ln x_i\ra + \frac{1}{2}\la p_i, \ln p_i \ra\right) 
  \end{align}
  
From  \eqref{eq:normsUV} and \eqref{eq:normsXPp}, it holds
\begin{align*}
    \|\u\|^2_{\U} = \frac{\norm{\bx}_{\sX}^2}{R_\sX^2} + \frac{\norm{\bp}_{\sP}^2}{R_\sP^2} = \frac{1}{m \ln n} \sum_{i=1}^m\left(  \|x_i\|^2_1 + \frac{1}{2}\|p_i\|^2_1\right) 
\end{align*}

Now as $\sA_\sV$ is unbounded, we define $\sV = \sQ\times \{ \z :~ \|\z\|_2 \leq R_{\Z} \}$. On $\sV$ we define the Euclidean prox-setup. 
Let $ \sQ_i \triangleq [-1,1]^{2n} $. Then for $ q_i\in \sQ_i$ we define 
\begin{equation}\label{eq:proxqqfunc}
    d_{\sQ_i}(q_i) = \frac{1}{2}\|q_i\|_2^2.
\end{equation}
Then 
\begin{equation}\label{eq:rqrq}
    R^2_\sQ = \sum_{i=1}^m R^2_{\sQ_i} = \frac{1}{2}\sum_{i=1}^m \max\limits_{q_i \in  \sQ_i }\|q_i\|_2 - \min\limits_{q \in  \sQ_i }\|q_i\|_2 = m n
\end{equation}
Then using this, \eqref{eq:proxqqfunc} and Lemma \eqref{lemma:constrained_lagrange_multipiers_wb}, we have
  \begin{align}\label{eq:prox-func_v_me}
      d_{\bv}(\bv) =\sum_{i=1}^m
 \frac{d_{\sQ_i}(q_i)}{R_\sQ^2} + \frac{d_{\Z_i}(z_i)}{R_\sZ^2} =  \frac{1}{2mn}\left( \|\q\|^2_2 + \frac{(\lambda_{\min}^+(W))^2\|\z\|^2_{2}}{8 \max_{i,j}C_{ij}^2 } \right).
  \end{align}  
From  \eqref{eq:normsUV} we have 
\begin{align*}
    \|\v\|^2_{\V} = \frac{\norm{\bq}_{\sQ}^2}{R_\sQ^2} + \frac{\norm{\bz}_{2}^2}{R_\sZ^2}= \frac{1}{2mn}\left( \|\q\|^2_2 + \frac{(\lambda_{\min}^+(W))^2\|\z\|^2_{2}}{8 \max_{i,j}C_{ij}^2 } \right).
\end{align*}
Thus, we use \eqref{eq:rxrx}, \eqref{eq:rqrq}  and rewrite Eq. \eqref{eq:constLbu} 
\begin{align}\label{eq:newLvu}
L_{\bu\bv} &= m\sqrt{n\ln n}\cbraces{\sqrt{6} L_{(\bx,\bp)\bq}  + \frac{4\|d\|_\infty }{\lm_{\min}(W)}\norm{\W}_{2\to(\bx,*)}}, \notag \\
 L_{\bv\bu} &= m\sqrt{n\ln n}\cbraces{\sqrt{6}L_{\bq(\bx,\bp)}  + \frac{4\|d\|_\infty }{\lm_{\min}(W)}\norm{\W}_{\bx\to2}}.
\end{align}
 Now we consider 
 \begin{align}\label{eq:W calcnorm}
  \norm{\W}^2_{\bx\to2} = 
   \max_{\|\x\|^2_\X \leq 1}  \left\|\W\x \right\|_2^2 
   =  \max_{\sum_{i=1}^m\|x_i\|^2_1 \leq 1}  \left\|\W\x \right\|_2^2.
 \end{align}
 The set $\sum_{i=1}^m\|x_i\|^2_1 \leq 1$ is contained in the set $\sum_{j=1}^n\sum_{i=1}^m[x_i]^2_j \leq 1$  as cros-product terms of $\|x_i\|^2_1$ are non-negative. Thus, we can change the constraint in the minimum in Eq. \eqref{eq:W calcnorm} as follows
  \begin{align}\label{eq:W calcnormfin} 
  \max_{\sum_{i=1}^m\|x_i\|^2_1 \leq 1}  \left\|\W\x \right\|_2^2 
  &\leq   \max_{\sum_{j=1}^n\sum_{i=1}^m[x_i]^2_j \leq 1}  \left\|\W\x \right\|_2^2 =\max_{ \|\x\|^2_2 \leq 1}  \left\|\W\x \right\|_2^2 \notag \\ 
  &= \max_{ \|\x\|_2 \leq 1}  \left\|\W\x \right\|_2^2 \triangleq \lm^2_{\max}(\W) =\lm^2_{\max}(W).
  \end{align}
 The last inequality holds due to $ \W \triangleq W \otimes  I_n$ and the properties of the Kronecker product for eigenvalues. Thus, $ \norm{\W}_{\bx\to2} = \lm^2_{\max}(W)$.\\
Next, we consider 
\begin{align}\label{eq:W2bxto}
    \norm{\W}^2_{\bx\to2} &=    \max_{\|\x\|^2_2 \leq 1}  \left\|\W\x \right\|_{\X,*}^2 =    \max_{\|\x\|^2_2 \leq 1} \sum_{i=1}^m \left\|Wx \right\|^2_\infty \leq \max_{\|\x\|^2_2 \leq 1} \sum_{i=1}^m \left\|Wx \right\|^2_2 \notag \\
    &=  \max_{\|\x\|^2_2 \leq 1}  \left\|\W\x \right\|^2_2 = \lm_{\max}^2(W), 
\end{align}
where the last equality holds due to \eqref{eq:W calcnormfin}.
Next, we estimate $L_{(\bx,\bp)\bq}$ and $L_{\bq(\bx,\bp)}$  by the definition 
\begin{equation}\label{eq:Ldef_conj}
  \left\|\begin{bmatrix*}[l]
        ~~\nabla_\bx F(\bx, \bp, \bq) - \nabla_\bx F(\bx, \bp, \bq') \\
        ~~\nabla_\bp F(\bx, \bp, \bq) - \nabla_\bp F(\bx, \bp,  \bq')
    \end{bmatrix*}\right\|_{(\X,\Pp),*} \leq L_{(\bx,\bp)\bq}  \|
 \bq-\bq' \|_2  
\end{equation}
 From the definition of dual norm and  $\norm{\cbraces{\bx,\bp}}_{(\sX,\sP)}^2 = \norm{\bx}_{\X}^2 + \norm{\bp}_{\Pp}^2$, it follows
 \begin{align*}
     &\left\|\begin{bmatrix*}[l]
        ~~\nabla_\bx F(\bx, \bp, \bq) - \nabla_\bx F(\bx, \bp, \bq') \\
        ~~\nabla_\bp F(\bx, \bp, \bq) - \nabla_\bp F(\bx, \bp,  \bq')
    \end{bmatrix*}\right\|_{(\X,\Pp),*} \\
    &= \max_{\norm{\bx}_{\sX}^2 + \norm{\bp}_{\sP}^2 \leq 1}\left\la \begin{bmatrix}
\x\\
\p
\end{bmatrix}, \begin{bmatrix*}[l]
        ~~\nabla_\bx F(\bx, \bp, \bq) - \nabla_\bx F(\bx, \bp, \bq') \\
        ~~\nabla_\bp F(\bx, \bp, \bq) - \nabla_\bp F(\bx, \bp,  \bq')
    \end{bmatrix*}\right\ra.
 \end{align*}
  From this and \eqref{eq:Ldef_conj} we get
    \begin{align}\label{eq:Lreform}
    \max_{\norm{\bx}_{1}^2 + \norm{1}_{\sP}^2 \leq 1}\left\la \begin{bmatrix}
\x\\
\p
\end{bmatrix} , \begin{bmatrix*}[l]
        ~~\nabla_\bx F(\bx, \bp, \bq) - \nabla_\bx F(\bx, \bp, \bq') \\
        ~~\nabla_\bp F(\bx, \bp, \bq) - \nabla_\bp F(\bx, \bp,  \bq')
    \end{bmatrix*}\right\ra \leq L_{(\bx,\bp)\bq}  \|
 \bq-\bq' \|_2.  
    \end{align}
By the definition of $F(\bx, \bp,  \bq)$  we have

\begin{align*}
    \begin{bmatrix}
  \nabla_\x F (\bx, \bp,  \bq)\\
 \nabla_\p F(\bx, \bp,  \bq)\\
 -\nabla_\q F(\bx, \bp,  \bq) 
    \end{bmatrix} = 
  \begin{bmatrix}
   -   2\|d\|_\infty \{[q_{i}]_{1...n}\}_{i=1}^m \\
  \boldsymbol d + 2\|d\|_\infty \boldsymbol A^\top \q \\
    -2\|d\|_\infty(\A\p - \b)   
    \end{bmatrix}\end{align*}
Here $\{[q_{i}]_{1...n}\}_{i=1}^m$ is a short form of vector $([q_{1}]_{1...n},[q_{2}]_{1...n},..., [q_{m}]_{1...n})$.\\
From this it follows that $\begin{bmatrix}
  \nabla_\x F(\bx, \bp,  \bq) \\
  \nabla_\p F (\bx, \bp,  \bq)
    \end{bmatrix}$ is linear function in $\q - \q'$, then \eqref{eq:Lreform}  can be rewritten as
\begin{align}\label{eq:Lcong2nor}
L_{(\bx,\bp)\bq}   &= \max_{ \|
 \bq-\bq' \|_2 \leq 1}  ~ \max_{\norm{\bx}_{\sX}^2 + \norm{\bp}_{\sP}^2 \leq 1}\left\la \begin{bmatrix}
\x\\
\p
\end{bmatrix}, \begin{bmatrix*}[l]
        ~~\nabla_\bx F(\bx, \bp, \bq) - \nabla_\bx F(\bx, \bp, \bq') \\
        ~~\nabla_\bp F(\bx, \bp, \bq) - \nabla_\bp F(\bx, \bp,  \bq')
    \end{bmatrix*}\right\ra.          
    \end{align} 
Then 
\begin{align*}
    \begin{bmatrix*}[l]
        ~~\nabla_\bx F(\bx, \bp, \bq) - \nabla_\bx F(\bx, \bp, \bq') \\
        ~~\nabla_\bp F(\bx, \bp, \bq) - \nabla_\bp F(\bx, \bp,  \bq')
    \end{bmatrix*} 
    &=  \begin{bmatrix*}[l]
  -    2\|d\|_\infty (\{[q_{i} - q'_{i}]_{1...n}\}_{i=1}^m) \\
  2\|d\|_\infty\boldsymbol A^ \top ( \q - \q') 
    \end{bmatrix*}  \\
        &= 2\|d\|_\infty 
    \begin{pmatrix}
-    \mathcal E  \\
 \boldsymbol A^\top
    \end{pmatrix}
    \begin{pmatrix}
        \q - \q'  
    \end{pmatrix}, 
    \end{align*}
    where $\mathcal E \in \{1,0\}^{mn\times 2mn}$ is a block-diagonal matrix
    \[
    \mathcal E = 
\begin{bmatrix}
\begin{bmatrix}
       I_n  ~
         0_{n\times n} 
\end{bmatrix} &   \cdots &  0_{2n\times n}\\
\vdots  & \ddots  & \vdots  \\
0_{2n\times n} &  \cdots &  \begin{bmatrix}
       I_n ~
         0_{n\times n} 
\end{bmatrix}
\end{bmatrix}.
\]
Thus, we use this for \eqref{eq:Lcong2nor} and get
\begin{align}\label{eq:Lcong2nor22}
L_{(\bx,\bp)\bq}   &= \max_{\norm{\bx}_{\sX}^2 + \norm{\bp}_{\sP}^2 \leq 1} \max_{\norm{\bx}_{1}^2 + \norm{1}_{\sP}^2 \leq 1} \left\la \begin{bmatrix}
\x\\
\p
\end{bmatrix},  ~2\|d\|_\infty 
    \begin{bmatrix}
-    \mathcal E  \\
 \boldsymbol A^\top
    \end{bmatrix}
    \begin{bmatrix}
        \q - \q'  
    \end{bmatrix}\right \ra.          
    \end{align}
By the same arguments we can get the same expression for $L_{\bq(\bx,\bp)}$ up to rearrangement of maximums.
Next, we use the fact that the $\ell_2$-norm is the conjugate norm for the $\ell_2$-norm. From this and  \eqref{eq:Lcong2nor} it follows
\begin{align}\label{eq:Lxy}
L_{(\bx,\bp)\bq} &=  \max_{\norm{\bx}_{\sX}^2 + \norm{\bp}_{\sP}^2 \leq 1}   2\|d\|_\infty \left\| \begin{bmatrix}
 -\mathcal E^\top  & \boldsymbol A  
    \end{bmatrix}  \begin{bmatrix}
        \x\\
        \p
    \end{bmatrix}\right\|_2 \notag \\
    & = \max_{\norm{\bx}_{\sX}^2 + \norm{\bp}_{\sP}^2 \leq 1}   2\|d\|_\infty \left\| -\mathcal E^\top \x  + \A \p\right\|_2
\end{align} 
Then
\begin{align}\label{eq:Axnorm}
\|- \mathcal E^\top \x + \boldsymbol A \p \|_2^2
&= \sum_{i=1}^{m}\left\|    -  \begin{pmatrix}
x_i\\
0_n
\end{pmatrix} + Ap_i\right\|^2_2 \leq   \sum_{i=1}^{m}\|x_i\|_2^2 +\sum_{i=1}^{m}\|   Ap_i\|_2^2.
\end{align}
The last bound holds due to $\la A p_i, (
x_i^\top,
0_n^\top) \ra \geq 0$ as the entries of $A,\p,\x$ are non-negative.\\
Next we use this for  \eqref{eq:Lxy}
\begin{align}\label{eq:normest}
L_{(\bx,\bp)\bq}^2 &= 4\|d\|^2_\infty \max_{\norm{\bx}_{\sX}^2 + \norm{\bp}_{\sP}^2 \leq 1}
\|\boldsymbol A \p - \mathcal E \x \|_2^2 
= \max_{\norm{\bx}_{\sX}^2 + \norm{\bp}_{\sP}^2 \leq 1}\|\boldsymbol A \p - \mathcal E \x \|_2^2  \notag\\
&\stackrel{\eqref{eq:Axnorm}}{\leq}4\|d\|^2_\infty\max_{\alpha \in \Delta_{2m}} \left( \sum_{i=1}^{m} \max_{\|p_i\|_1^2 \leq {\alpha_{i}}} \|   Ap_i\|_2^2 + \sum_{i=1}^{m} \max_{\|x_i\|_1^2 \leq {\alpha_{i+m}}} \|x_i\|_2^2\right) \notag\\
&=4\|d\|^2_\infty\max_{\alpha \in \Delta_{2m}} \left( \sum_{i=1}^{m} \alpha_{i}  \max_{ \|p_i\|_1 \leq 1}   \|   Ap_i\|_2^2 +\sum_{i=1}^{m}\alpha_{i+m} 
\max_{\|x_i\|_1\leq 1} \|x_i\|_2^2\right).
\end{align}
By the definition of incidence matrix $A$ we get
$
Ap_i = (h_1^\top, h_2^\top)
$, where $h_1$ and $h_2$ such that $\boldsymbol 1^\top h_1 = \boldsymbol 1^\top h_2= \sum_{j=1}^{n^2} [p_i]_j$ = 1 as $p_i \in \Delta_{n^2}  ~\forall i =1,...,m$.
Thus,
\begin{equation}\label{eq:Ax22}
  \|A p_i\|_2^2 = \|h_1\|_2^2 + \|h_2\|_2^2 \leq \|h_1\|_1^2 + \|h_2\|_1^2 = 2.
  \end{equation}
As $x_i \in \Delta_n, \forall i =1, ...,m$ we have
\begin{equation}\label{eq:p_estim}
    \max_{\|x_i\|_1\leq 1} \|x_i\|_2^2 \leq \max_{\|x_i\|_1\leq 1} \|x_i\|_1^2 = 1.
\end{equation}

Using    \eqref{eq:Ax22} and  \eqref{eq:p_estim} in \eqref{eq:normest} we get
\begin{align}\label{eq:Adefnormfind}
L_{(\bx,\bp)\bq}^2 &= 4\|d\|^2_\infty\max_{\|\p\|^2_1 + \|\x\|^2_1 \leq 1}
\|\boldsymbol A \p - \mathcal E \x \|_2^2 \notag \\
&\leq
4\|d\|^2_\infty\max_{\alpha \in \Delta_{2m}} m\ln n\left( 2\sum_{i=1}^{m} \alpha_{i}  
 + \sum_{i=m+1}^{2m} \alpha_{i}  \right) \leq \max_{\alpha \in \Delta_{2m}}  2 \sum_{i=1}^{2m} \alpha_{i} = 8\|d\|^2_\infty.
\end{align}
Hence,
\begin{equation}\label{eq:finLfin}
    L_{(\bx,\bp)\bq} =    L_{\bq(\bx,\bp)}=2\sqrt 2\|d\|_\infty.
\end{equation}
Using this, \eqref{eq:W calcnormfin} and \eqref{eq:W2bxto}    for \eqref{eq:newLvu} 
and 
we get 
\begin{align}\label{eq:newLvunew}
L_{\bu\bv} = L_{\bv\bu} =      L_{\bv\bu} &= 4m \sqrt{n\ln n}\|d\|_\infty\cbraces{\sqrt{3}  +  \chi  }.
\end{align}
Then we use that for any $a,b$, 
$\sqrt{a^2+b^2}\leq \sqrt 2 \max\{a,b\}$ and rewrite this as follows
\[L_{\bv\bu} = 4m \sqrt{n\ln n} \|d\|_\infty\cbraces{\sqrt{3}  + \chi  }\leq 4m \sqrt{2n\ln n} \|d\|_\infty\max\{\sqrt{3}, \chi \}  \leq 8m \sqrt{2n\ln n}\|d\|_\infty \chi.  \]
The last holds due to $\chi \geq1$ for any graph. Finally, we use  $\|d\|_\infty = \max_{i,j}C_{ij}$ and finish the proof.

\end{proof}

\subsection{Proof of Theorem \ref{Th:first_alg11}}\label{app:Th_WB}

\begin{proof}[Proof of Theorem \ref{Th:first_alg11}]

\add{
From the Theorem \ref{Th:distr_MP},
 it holds
\begin{align}\label{eq:boundWbbar}
    &\max_{ \bq\in\sQ} \sum_{i=1}^m f_i(\hat x_{av}^N, \hat p_i^N, q_i) - \min_{\substack{x\in\bar\sX,\\ \bp\in\sP}}  \sum_{i=1}^m f_i(x, p_i, \hat q_i^N)  
\leq \frac{(4 + 17\sqrt{2})L_\zeta}{N}
\end{align}
where from \eqref{eq:def_saddle_prob}
\begin{equation}\label{eq:funcfuncf}
    f_i(x_i,p_i,q_i) \triangleq d^\top p_i +2\|d\|_\infty\left(q_i^\top Ap_i -b_i^\top q_i\right)\end{equation}   and
\begin{align}\label{eq:consW}
    \norm{\mW\hat\bx^N}_2\leq \frac{17L_\zeta}{NR_\sZ}.
\end{align}
By the definition of optimal transport 
\begin{align}\label{eq:baryWBc}
    \mathcal W(x_i,y_i) &= \min_{\pi \in  U(x_i,y_i)} \langle C, \pi \rangle = \min_{Ap_i= b_i, p_i \in \Delta_n} d^\top p_i 
    = \min_{p_i \in \Delta_n} d^\top p_i +2\|d\|_\infty \| Ap_i -b_i\|_1, \notag\\
    &= \min_{p_i \in \Delta_n}\max_{q_i \in   [-1,1]^{2n}} d^\top p_i +2\|d\|_\infty\left(q_i^\top Ap_i -b_i^\top q_i\right), 
\end{align}
where we used the representation of optimal transport from paper \cite{jambulapati2019direct} based on the definition of the $\ell_1$-norm. \\
 Together with   \eqref{eq:baryWBc} and \eqref{eq:funcfuncf}, we rewrite \eqref{eq:boundWbbar} 
as follows
\[ \sum_{i=1}^m \mathcal  W(\hat x^N_{av}, y_i) -  \min_{x \in \Delta_n} \sum_{i=1}^m \mathcal  W(x, y_i) 
\leq \frac{(4 + 17\sqrt{2})L_\zeta}{N}\]
 Equating this to $m\e$ we get the number of iteration for the Algorithm \ref{app:algWB}
 \begin{equation}\label{eq:NNN}
     N = \frac{(4 + 17\sqrt{2})L_\zeta}{m\e}.
 \end{equation}
 Thus, the initial SPP scaled by $1/m$ will be solved with $\e$-precision.\\
 Next, we use Lemma \ref{lemma:constrained_lagrange_multipiers_wb} for \eqref{eq:consW} and get
 \begin{align}\label{eq:consW2}
    \norm{\mW\hat\bx^N}_2\leq \frac{17L_\zeta}{N} \cdot \frac{\lambda_{\min}^+(W)}{2\sqrt{2mn}\max_{i,j}C_{ij}}.
\end{align}
Then we use \eqref{eq:NNN} and get 
 \begin{align}\label{eq:consW3}
   \norm{\mW\hat\bx^N}_2\leq  \frac{17\sqrt{m}  \lambda_{\min}^+(W) \e}{2(4+17\sqrt 2)\sqrt{2n}\max_{i,j}C_{ij}}.
\end{align}
The complexity of one iteration of Algorithm \ref{MP_WB_distr}  per node is $O\left(n^2\right)$ as the number of non-zero elements in matrix A is $2n^2$. Multiplying this by the number of iterations $ N={(4 + 17\sqrt{2})L_\zeta}/{(m\e)}$  we get 
\begin{align*}
     O(n^2N) &=O\left(n^2 \frac{L_{\zeta}}{m\e} \right) 
     =O\left(\frac{n^2}{\e}\sqrt{n\log n} \max_{i,j}C_{ij}  \chi
    \right),
\end{align*}
where we used the definition of $L_{\zeta}$ from \eqref{eq:def_L_zetaWB}.
}

\end{proof}

\section{Algorithm and experimental data } \label{app:algWB}
\subsection{Decentralized Mirror-Prox for Wasserstein Barycenters}
To present the algorithm we introduce soft-max function for an $x\in \R^n$
\[
{\rm Softmax}(x) =\frac{ \exp(x)}{\sum_{i=1}^n(\exp(x))}
\]
    
\begin{algorithm}[ht!]
    \caption{Decentralized Mirror-Prox for Wasserstein Barycenters}
    \label{MP_WB_distr}
    \begin{algorithmic}[1]
\REQUIRE measures $y_1,...,y_m$, linearized cost matrix $d$, incidence matrix $A$, communication matrix $W$, number of iterations $N$, stepsize $\alpha  = 1/ \cbraces{16m\sqrt{2n\ln n}\max_{i,j}C_{ij} \chi}$ 
\STATE Initialize starting points $x^1=\frac{1}{n}\boldsymbol 1_{n}$, $p_1^1=...= p_m^1 = \frac{1}{n^2}\boldsymbol 1_{n^2}$, $q_1^1 = ... =q_m^1 =\boldsymbol 0_{2n}$
        \STATE Initialize $\eta = \alpha m n \max_{i,j}C_{ij}^3 / \lm_{\min}^2(W)$,
         $\beta =6 \alpha m \log n $,
          $\theta = 2\alpha m n \max_{i,j}C_{ij}^2 / \lm_{\min}^2(W)$
        \FOR{ $k=1,2,\cdots,N-1$ }
         \STATE   Each node $i=1,..,m$ computes 
         \begin{align*}
             u^{k+1}_i &=
   {\rm Softmax}\left(\log(p^{k}_i) 
       - \beta \left( d+ 2\max_{i,j}C_{ij} A^\top q^{k}_i \right)
        \right) \\
        s_i^{k+1} &= 
       {\rm Softmax}\left( \log(x_i^{k})+ 
        \beta \left(
         [q^k_i]_{1...n} -  \textstyle\sum_{j=1}^m  W_{ij}z^k_j\right)
        \right)\\
        v_i^{k+1} &= q^k_i + \eta\left(  
         A p_i^k - 
         \begin{pmatrix}
         x_i^k\\
         y_i
         \end{pmatrix}
        \right), \quad \text{project } v_i^{k+1} \text{ onto } [-1,1]^{2n}.\\
        \lm_i^{k+1} &= z^k_i + \theta \textstyle \sum_{j=1}^m  W_{ij} x^k_j
         \end{align*}
             \STATE   Each node $i=1,..,m$ computes 
             \begin{align*}
                    p^{k+1}_i &=
     {\rm Softmax}\left(\log(p^{k}_i)
       - \beta \left( d+ 2\max_{i,j}C_{ij} A^\top v^{k+1}_i \right)
         \right) \\
          x_i^{k+1} &= 
           {\rm Softmax}\left(
        \log(x_i^{k})+ 
        \beta
       \left([v^{k+1}_i]_{1...n} -\textstyle \sum_{j=1}^m  W_{ij} \lm^{k+1}_j\right)
        \right)\\
        q_i^{k+1} &= q^k_i + \eta\left( A u_i^{k+1}-
        \begin{pmatrix}
          s_i^{k+1}\\
          y_i
        \end{pmatrix}
        \right), \quad \text{project } q_i^{k+1} \text{ onto } [-1,1]^{2n}. \\
        z_i^{k+1} &= z^k_i + \theta \textstyle \sum_{j=1}^m  W_{ij} s^{k+1}_j
             \end{align*}
           \ENDFOR
        \ENSURE 
        $ \hat \x^N = (\hat x_1^N , \dots, \hat x_m^N )^\top$, where  $\hat x_i^N = \frac{1}{N}\sum\limits_{k=1}^{N} s_i^k$
    \end{algorithmic}
\end{algorithm}


\subsection{Letter 'B'  in a variety of fonts from the notMNIST dataset}
\begin{figure}[H]
\centering
\includegraphics[width=0.8\textwidth]{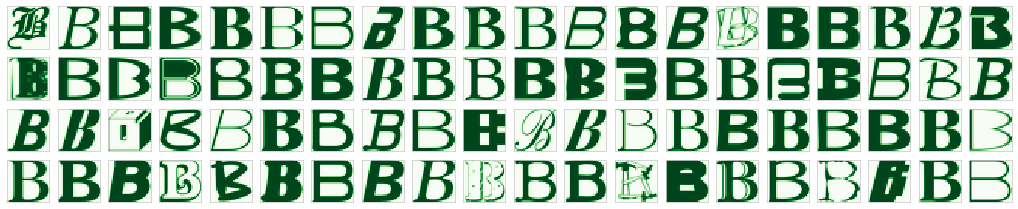}
\label{fig:Bdata}
\end{figure}

\subsection{Network architectures}

\begin{figure}[H]
  \centering
\begin{subfigure}{.24\textwidth}
  \centering
  \includegraphics[width=.7\linewidth]{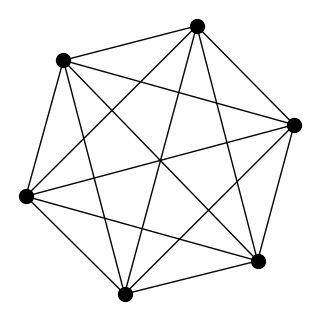}
  \caption*{The complete graph}
  \label{fig:sfig1}
\end{subfigure}%
\begin{subfigure}{.24\textwidth}
  \centering
  \includegraphics[width=.7\linewidth]{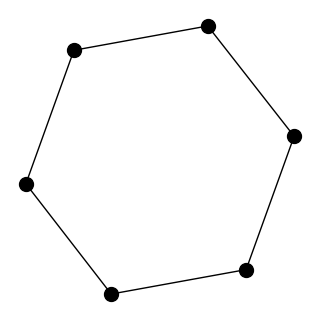}
  \caption*{The cycle graph}
\end{subfigure}
\begin{subfigure}{.24\textwidth}
  \centering
  \includegraphics[width=.7\linewidth]{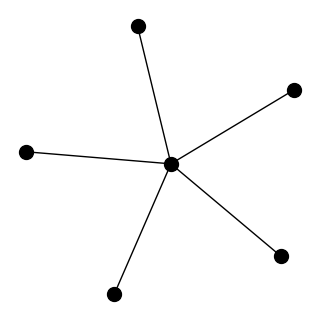}
  \caption*{The star graph}
  \label{fig:sfig2}
\end{subfigure}
\begin{subfigure}{.24\textwidth}
  \centering
  \includegraphics[width=.7\linewidth]{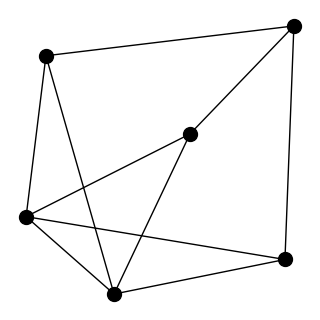}
  \caption*{ Erd\H{o}s-R\'enyi random graph with probability of edge creation $p=0.5$}
\end{subfigure}
\label{fig:Fig3}
\end{figure}

\section{Future work}
First of all, the approach of this paper can be generalized to non-symmetric matrix $W$, but with Network compatibility and Kernel properties. This observation allows to improve the bounds by using proper weighting, see \cite{gasnikov2017universal}.

By using the standard restarts or regularization arguments, all the results of this paper have convex-concave or strongly convex-concave analogues. Unfortunately, optimalilty w.r.t. $\varepsilon$ take places only for the convex-concave case not for the strongly convex-concave one.\footnote{The analysis developed in this paper also does not well fitted to the strongly convex-concave saddle-point problems with different constants of strong convexity and concavity, see the lower bound. In the non-distributed setup, this a popular direction of research \cite{lin2020near}. For distributed setup, this is an open problem.}
Our paper technique can be generalized to non-smooth problems by using another variant of sliding procedure \cite{lan2020first,dvinskikh2021decentralized,gorbunov2019optimal}. By using batching technique, the results can be generalized to stochastic saddle-point problems \cite{dvinskikh2021decentralized,gorbunov2019optimal}. Instead of the smooth convex-concave saddle-point problem we can consider general sum-type saddle-point problems with common variables in more general form. For each group of common variable, we introduce corresponding communication network which includes the nodes correspond to the terms contain this variable. The bounds change according to the worth condition number of Laplacian matrices of these networks. 
Based on the lower bound, we expect that optimal algorithms for all the parameters for smooth (strongly) convex-concave saddle-point problems one can search as a combination Mirror-Prox with accelerated consensus algorithm, see \cite{gorbunov2020recent} and references therein.

An interesting and open problem is 
generalizing
the results of this paper to $\beta$-similar terms in  the smooth convex-concave saddle-point problem \cite{arjevani2015communication,sun2019convergence}.

The results will probably change with the replacement of $L$ by $\beta$.

\end{document}